\documentclass[11pt]{article}
\usepackage{songmei}
\usepackage[toc,page]{appendix}
\usepackage{graphicx}

\usepackage[mathscr]{euscript}

\DeclareSymbolFont{rsfs}{U}{rsfs}{m}{n}
\DeclareSymbolFontAlphabet{\mathscrsfs}{rsfs}

\def\KRR{{\sf KRR}}
\def\cH{{\mathcal H}}
\def\bsh{{\boldsymbol h}}
\def\bPsi{{\boldsymbol \Psi}}
\def\bLambda{{\boldsymbol \Lambda}}
\def\bS{{\boldsymbol S}}
\def\bO{{\boldsymbol O}}
\def\bQ{{\boldsymbol Q}}
\def\blambda{{\boldsymbol \lambda}}
\def\bY{{\boldsymbol Y}}
\def\bH{{\boldsymbol H}}
\def\bE{{\boldsymbol E}}
\def\beps{{\boldsymbol \eps}}
\def\boldf{{\boldsymbol f}}
\def\KR{{\sf KR}}
\def\balpha{{\boldsymbol \alpha}}
\def\bbH{{\mathbb H}}
\def\bK{{\boldsymbol K}}
\def\K{{\mathbb K}}
\def\cL{{\mathcal L}}
\def\T{{\mathbb T}}
\def\endd{{\rm end}}
\def\bw{{\boldsymbol w}}
\def\de{{\rm d}}
\def\bx{{\boldsymbol x}}
\def\by{{\boldsymbol y}}
\def\bW{{\boldsymbol W}}
\def\ba{{\boldsymbol a}}
\def\hba{\hat{\boldsymbol a}}
\def\cF{{\mathcal F}}
\def\Unif{{\sf Unif}}
\def\normal{{\sf N}}
\def\bU{{\boldsymbol U}}
\def\bV{{\boldsymbol V}}
\def\bM{{\boldsymbol M}}
\def\bZ{{\boldsymbol Z}}
\def\bS{{\boldsymbol S}}
\def\bxi{{\boldsymbol \xi}}

\def\hy{\hat{y}}

\def\cuH{\mathscrsfs{H}}

\def\proj{{\mathsf P}}
\def\bbHe{{\rm He}}
\def\cE{{\mathcal E}}

\def\stest{\mbox{\tiny\rm test}}
\def\bdelta{{\boldsymbol\delta}}
\def\Trace{{\rm Tr}}

\def\bbeta{{\boldsymbol \beta}}
\def\bDelta{{\boldsymbol \Delta}}
\def\bB{{\boldsymbol B}}
\def\bH{{\boldsymbol H}}
\def\bX{{\boldsymbol X}}
\def\bG{{\boldsymbol G}}

\def\bh{{\boldsymbol h}}
\def\be{{\boldsymbol e}}
\def\bu{{\boldsymbol u}}
\def\bg{{\boldsymbol g}}

\def\sM{{\sf M}}

\def\bA{{\boldsymbol A}}
\def\bD{{\boldsymbol D}}

\def\bv{{\boldsymbol v}}
\def\bxi{{\boldsymbol \xi}}
\def\btheta{{\boldsymbol \theta}}
\def\bTheta{{\boldsymbol \Theta}}

\def\Coeff{{\rm Coeff}}

\def\bfone{{\boldsymbol 1}}

\def\RF{{\sf RF}}
\def\NT{{\sf NT}}
\def\NN{{\sf NN}}
\def\KRR{{\sf KRR}}
\def\reals{{\mathbb R}}
\def\integers{{\mathbb Z}}
\def\naturals{{\mathbb N}}

\def\hf{\hat{f}}
\def\bi{{\boldsymbol i}}
\def\bj{{\boldsymbol j}}
\def\bk{{\boldsymbol k}}
\def\tcT{\widetilde{\mathcal T}}
\def\cC{{\mathcal C}}
\def\cQ{{\mathcal Q}}

\def\sk{{\rm sk}}

\def\cT{{\mathcal T}}
\def\hR{\hat{R}}

\usepackage{hyperref}
\hypersetup{
    colorlinks,
    linkcolor={blue!80!black},
    citecolor={green!50!black},
}
\colorlet{linkequation}{blue}

\begin{document}

\title{Linearized two-layers neural networks in high dimension}

\author{Behrooz Ghorbani\thanks{Department of Electrical Engineering, Stanford University}, \;\;Song Mei\thanks{Institute for Computational and Mathematical Engineering, Stanford University},\;\; Theodor Misiakiewicz\thanks{Department of Statistics, Stanford University}, \;\; Andrea Montanari\thanks{Department of Electrical Engineering and Department of Statistics, Stanford University}}

\maketitle

\begin{abstract}
We consider the problem of learning an unknown function $f_{\star}$ on the $d$-dimensional sphere with respect to the square loss,  given i.i.d. samples
$\{(y_i,\bx_i)\}_{i\le n}$ where $\bx_i$ is a feature vector uniformly distributed on the sphere and $y_i=f_{\star}(\bx_i)+\eps_i$. 
We study two popular classes of models that can be regarded as linearizations of two-layers neural networks around a random initialization:
the random features model of Rahimi-Recht (\RF);  the neural tangent kernel model of Jacot-Gabriel-Hongler (\NT).
Both these approaches can also be regarded as randomized  approximations of  kernel ridge regression (with respect to different kernels),
and  enjoy universal approximation properties when the number of neurons
$N$ diverges, for a fixed dimension $d$. 

We consider two specific regimes: the approximation-limited regime, in which $n=\infty$ while $d$ and $N$ are large but finite; and
the sample size-limited regime in which $N=\infty$ while $d$ and $n$ are large but finite.
In the first regime, we prove that if $d^{\ell + \delta} \le N\le d^{\ell+1-\delta}$ for small $\delta > 0$, then \RF\, effectively fits a
degree-$\ell$ polynomial in the raw features, and \NT\, fits a  degree-$(\ell+1)$ polynomial.
In the second regime, both \RF\, and \NT\, reduce to kernel methods with rotationally invariant kernels.
We prove that, if the number of samples is $d^{\ell + \delta} \le n \le d^{\ell +1-\delta}$, then kernel methods can fit at most a
a degree-$\ell$ polynomial in the raw features. This lower bound is achieved by kernel ridge regression. 
Optimal prediction error is achieved for vanishing ridge regularization.
\end{abstract}

\tableofcontents

\section{Introduction and main results}

In the canonical statistical learning problem, we are given
independent and identically distributed (i.i.d.) pairs $(y_i,\bx_i)$, $1 \le i\le n$, where $\bx_i\in\reals^d$ is a feature vector and $y_i\in\reals$ is a label or response variable. 
We would like to construct a function $f$ which allows us to predict future responses. Throughout this paper, we will measure the quality of a predictor 
$f$ via its square prediction error (risk): $R(f) \equiv \E\{(y-f(\bx))^2\}$. 

\subsection{Background}

For a number of important applications, state-of-the-art  performances are obtained by representing the function $f$ as
a multi-layers neural network. The simplest 
model in this class is given by two-layers networks (\NN):
\begin{align}\tag{\NN}
\cF_{\NN} \equiv \Big\{ f(\bx) = \sum_{i=1}^N a_i\,\sigma(\<\bw_i,\bx\>)\; :\;\;\; a_i\in\reals, \bw_i\in \reals^d \;\,\,\, \forall i\le N\Big\}\, .
\end{align}
Here $N$ is the number of neurons and $\sigma:\reals\to\reals$ is an
activation function.

Two-layers neural networks have been extensively studied in the nineties, with a focus on two goals: $(i)$~Establishing
approximation guarantees over classical function spaces; $(ii)$~Controlling the generalization error via Rademacher
complexity arguments. We  refer to
\cite{pinkus1999approximation,anthony2009neural} for surveys of these results.

Computational aspects were notably under-represented within these early theoretical contributions. On the contrary,
it is nowadays increasingly clear that computational and statistical aspects cannot be separated in the analysis of
neural networks (see, e.g. \cite{soudry2018implicit,mei2018mean,chizat2018global}).  Indeed, the optimization algorithm
does not simply compute the unique minimizer
of a regularized empirical risk: it instead selects one among many possible near-minimizers, whose generalization properties can vary significantly. 
Therefore, the specific optimization algorithm is an integral part of the definition of the regularization  method.

A concrete scenario in which this interplay can be understood precisely is the so-called `neural tangent kernel' regime.
First explicitly described in \cite{jacot2018neural}, this regime has attracted considerable amount of work.
The basic idea is that, for highly overparametrized networks, the network weights barely change from their random initialization.
We can therefore replace the nonlinear function class $\cF_{\NN}$  by its first order Taylor expansion around this initialization.

Denoting by $(a_{0,i},\bw_{0,i})_{i\le N}$ the weights at initialization, a first order Taylor expansion yields
\begin{align*}
f_{\NN}(\bx) &= \sum_{i=1}^N a_i\,\sigma(\<\bw_i,\bx\>) \\
&\approx f_{\NN,0}(\bx) + \sum_{i=1}^N (a_i-a_{0,i})\,\sigma(\<\bw_{0,i},\bx\>)+\sum_{i=1}^N a_{0,i}\<\bw_i-\bw_{0,i},\bx\>\,
\sigma'(\<\bw_{0,i},\bx\>)\, ,
\end{align*}
where  $f_{\NN,0}$ is the neural network at initialization.
In other words,  $f_{\NN}-f_{\NN,0}$ is a function in the direct sum $\cF_{\NT}(\bW)\oplus \cF_{\RF}(\bW)$, where we defined 
\begin{align}
\cF_{\RF}(\bW) &\equiv \Big\{f(\bx) = \sum_{i=1}^N a_i\,\sigma(\<\bw_i,\bx\>)\; :\;\;\; a_i\in\reals \; \forall i\le N\Big\}\, , \tag{\RF}\label{eqn:RF}\\
\cF_{\NT}(\bW) &\equiv \Big\{ f(\bx) = \sum_{i=1}^N \<\ba_i,\bx\>\,\sigma'(\<\bw_i,\bx\>)\; :\;\;\; \ba_i\in\reals^d \; \forall i\le N\Big\}\, .  \tag{\NT}
\end{align}
Here $\bW\in\reals^{N\times d}$ is a matrix whose $i$-th row is the vector $\bw_i$, and
$\sigma'$ is the derivative of the activation function
with respect to its argument (if $\<\bw_i,\bx\>$ has a density, $\sigma$ only needs to be weakly differentiable).

We will refer to $\cF_{\RF}(\bW)$ as the `random features' (\RF) model: it amounts to fixing the first layer, and
only optimizing the coefficients in the second layer.
Equivalently, $\cF_{\RF}(\bW)$ corresponds to the first order Taylor expansion of $f_{\NN}$ with respect to the second layer weights $(a_i)_{i\le N}$.
This model can be traced back to the work of Neal \cite{neal1996priors},
and was successfully developed by Rahimi and Recht \cite{rahimi2008random} as a randomized approximation to kernel methods.

The second function class $\cF_{\NT}(\bW)$  corresponds to the first order Taylor expansion of $f_{\NN}$ with respect to the 
first layer weights $(\bw_i)_{i\le N}$ \cite{jacot2018neural}. We will refer to $\cF_{\NT}(\bW)$
as the neural tangent class\footnote{Often the term `neural tangent' is reserved for the direct sum $\cF_{\NT}(\bW)\oplus \cF_{\RF}(\bW)$.
  We find it more convenient to give distinct names to each of the two terms, especially since $\cF_{\RF}(\bW)$ has much smaller
  dimension than $\cF_{\NT}(\bW)$ for large $d$.}.

A sequence of recent papers proves that, in a certain overparametrized regime, gradient descent (GD) applied to the nonlinear
neural network class $\cF_{\NN}$ effectively converges to a model in $\cF_{\NT}(\bW)\oplus \cF_{\RF}(\bW)$.
Namely, if the number of neurons $N$ is larger than a threshold $N_0(n,d)$, and training is initialized with
$f_0(\bx) = N^{-1/2}\sum_{i=1}^N a_{0,i}\,\sigma(\<\bw_{0,i},\bx\>)$ where $\{(a_{0,i},\bw_{0,i})\}_{i\le N}\sim_{iid}
\normal(0,1)\otimes \normal(0,\id_d/d)$, then gradient descent converges exponentially fast to
weights $\{(a_{i},\bw_{i})\}_{i\le N}$ such that $f-f_0$ is well approximated by a function in  $\cF_{\NT}(\bW)\oplus \cF_{\RF}(\bW)$.
The specific value of the threshold $N_0(n,d)$ for the onset of this \NT\, regime has been steadily pushed down
over the last year \cite{du2018gradient,du2018gradient2,allen2018convergence,zou2018stochastic,arora2019fine}.

Does the \NT\, regime explain the power of multi-layers neural networks, when trained by gradient descent methods?
From an empirical point of view, the evidence is not univocal  \cite{lee2019wide,geiger2019disentangling,chizat2019lazy}.
From a theoretical point of view, while the expressivity of neural networks is superior to the one of \NT\, models,
this hypothesis is not easy to dismiss for at least two reasons.
First,  neural networks learned by gradient descent algorithms form a significantly smaller class than general
networks. Second, the answer depends on the data distribution, the target function $f_*$ and the sample size.

In order to clarify this question, we explore the behavior of \RF\, and \NT\, models in the high-dimensional
setting. More precisely, we consider two specific asymptotic regimes:
\begin{enumerate}
\item[$(i)$] The infinite sample size case in which $n=\infty$, and $N,d$ diverge while being polynomially related.
  In this case the prediction error reduces to the approximation error $\inf_{f\in\cF_{\sM}}\E\{[f_*(\bx)-f(\bx)]^2\}$,
  for either model $\sM\in\{\NT,\RF\}$.
\item[$(ii)$] The infinite width regime in which $N=\infty$ and $n,d$ diverge while being polynomially related.
  In this case (and under a suitable bound on the $\ell_2$ norm of the coefficients) both classes $\cF_\RF$, $\cF_\NT$
  reduce to certain reproducing kernel Hilbert spaces (RKHS).
\end{enumerate}
In both cases we obtain sharp results, up to errors vanishing as $d\to\infty$. Crucially, our results hold \emph{pointwise},
i.e. they provide a characterization of approximation and generalization error which hold \emph{for a given function $f_*$}.
This allows us to derive precise separation results between \NN\, and \NT\, models.

\subsection{A parenthesis}
\label{sec:Parenthesis}

The approximation properties of neural networks have been studied for over three decades \cite{devore1989optimal, cybenko1989approximation, hornik1991approximation, barron1993universal, mhaskar1994dimension, girosi1995regularization, mhaskar1996neural, petrushev1998approximation, maiorov1999best, pinkus1999approximation}.
It is useful to discuss the  relation between the questions outlined above and existing literature.

A number of results are available on the approximation of functions in certain smoothness classes by two-layers
neural networks. In particular \cite{barron1993universal} controls smoothness by the average frequency content
in the Fourier transform (the `Barron norm'), while \cite{mhaskar1996neural, petrushev1998approximation, maiorov1999best} use
classical Sobolev norms.
 For instance \cite{maiorov1999best} proves that $N$-neurons
\NN\, approximate functions in the Sobolev ball $W^{r}_2$ with worst case error
\begin{align}
  C_1(d)N^{-r/(d-1)}\le \sup_{f\in W_2^r}\inf_{\hf\in\cF_{\NN}}\E\{[f(\bx)-\hf(\bx)]^2\}\le C_2(d)N^{-r/(d-1)}\, .\label{eq:Maiorov}
\end{align}
for some unspecified functions $C_1,C_2$. (Similar results are found in \cite{petrushev1998approximation}.)
These results cannot be used for our purposes.

First of all, we are interested in the \NT\, class which is potentially much less powerful than \NN.

Second, bounds of the type \eqref{eq:Maiorov} make it hard to prove separation
results between \NN\, and \NT. In order to prove such a separation, we would have to prove that neural networks
trained by gradient descent have good approximation properties, uniformly  over Sobolev balls.
This objective is currently out of reach.
Our pointwise approximation results make it much easier to prove separation statements.

Third, earlier work neglects polynomial dependencies in $d$.  Bounds of the type \eqref{eq:Maiorov}
have weak implications when both $d$ and $N$ are large, say $d=100$, $N=10^6$.
We will instead prove sharp asymptotic results that are valid in this regime. As illustrated in  the next section, our analysis
captures the actual behavior in a quantitative manner, already when $d\ge 100$. 

Quantitative results in the high-dimensional regime have been proved only recently. 
In particular, Bach \cite{bach2017equivalence} established quantitative upper and lower bounds for the approximation error
in the \RF\, model. However, these results do not have direct implications
on the \NT\, model which is our main interest here. Further, 
lower bounds in \cite{bach2017equivalence} are, as before, worst case over a certain RKHS.
(See also \cite{bach2013sharp,alaoui2015fast,rudi2017generalization} for related work.)

Similar considerations apply to the generalization error of kernel methods. While this is a classical topic
\cite{cristianini2000introduction,caponnetto2007optimal,rudi2017generalization,liang2018just}, 
earlier work proves minimax upper and lower bounds.
Establishing pointwise lower bounds  is instead important in order to understand precisely the separation between
neural networks and their linearized counterparts.
We refer to Section \ref{sec:Related} for further discussion of related work.

\subsection{A numerical experiment}
\label{sec:Numerical}

\begin{figure}
\includegraphics[width=0.49\linewidth]{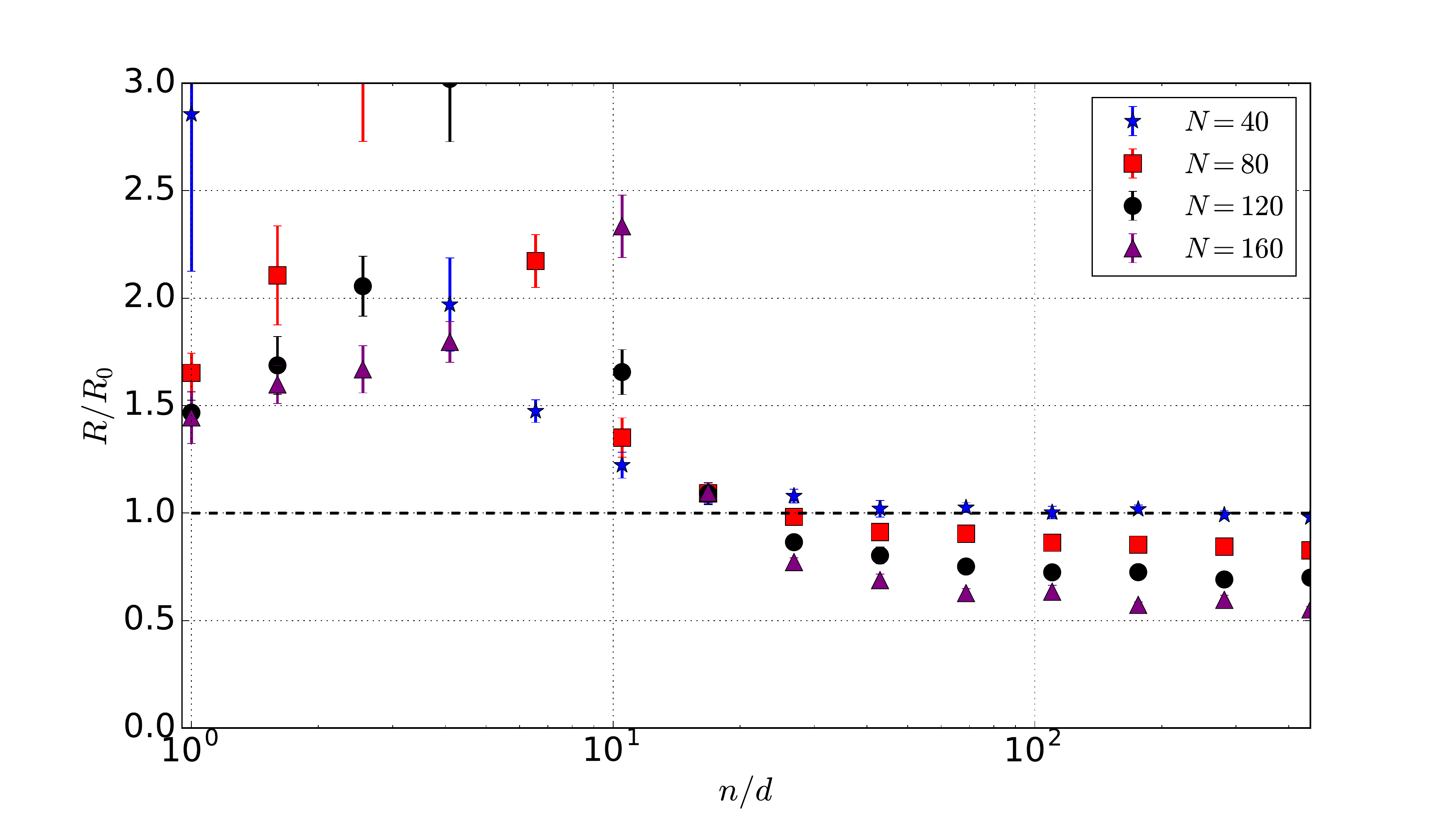}
\includegraphics[width=0.49\linewidth]{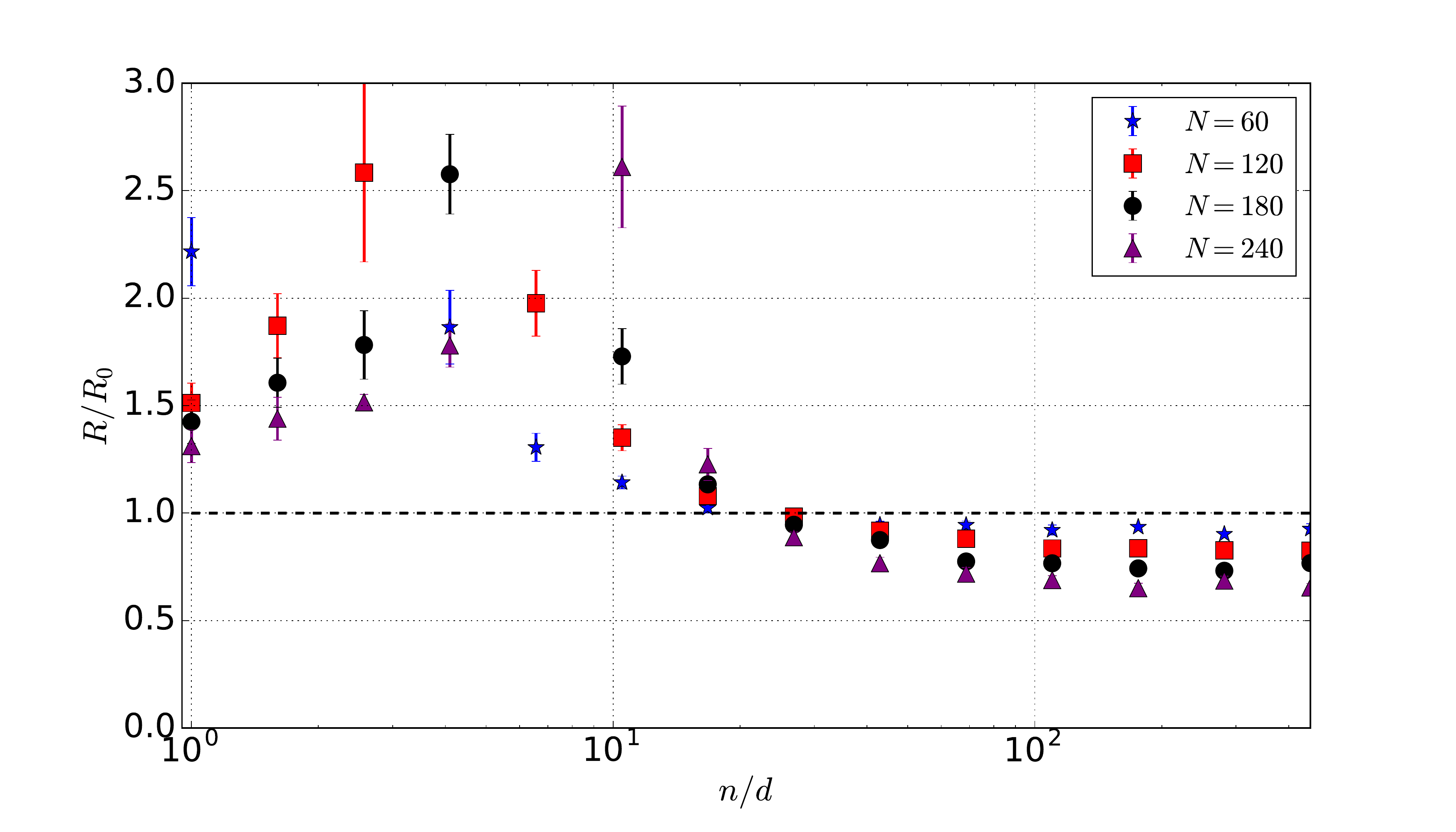}\\
\includegraphics[width=0.49\linewidth]{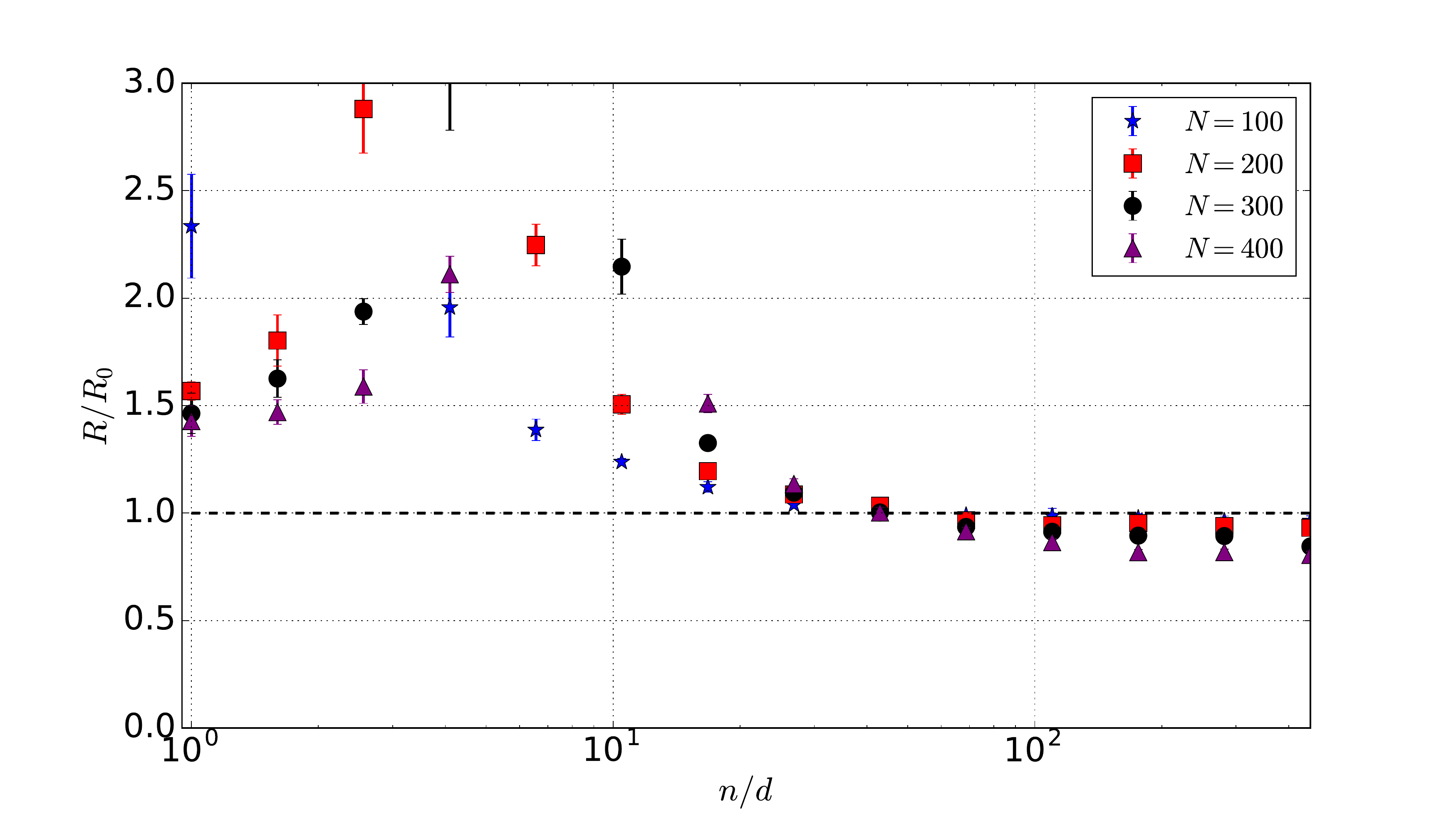}
\includegraphics[width=0.49\linewidth]{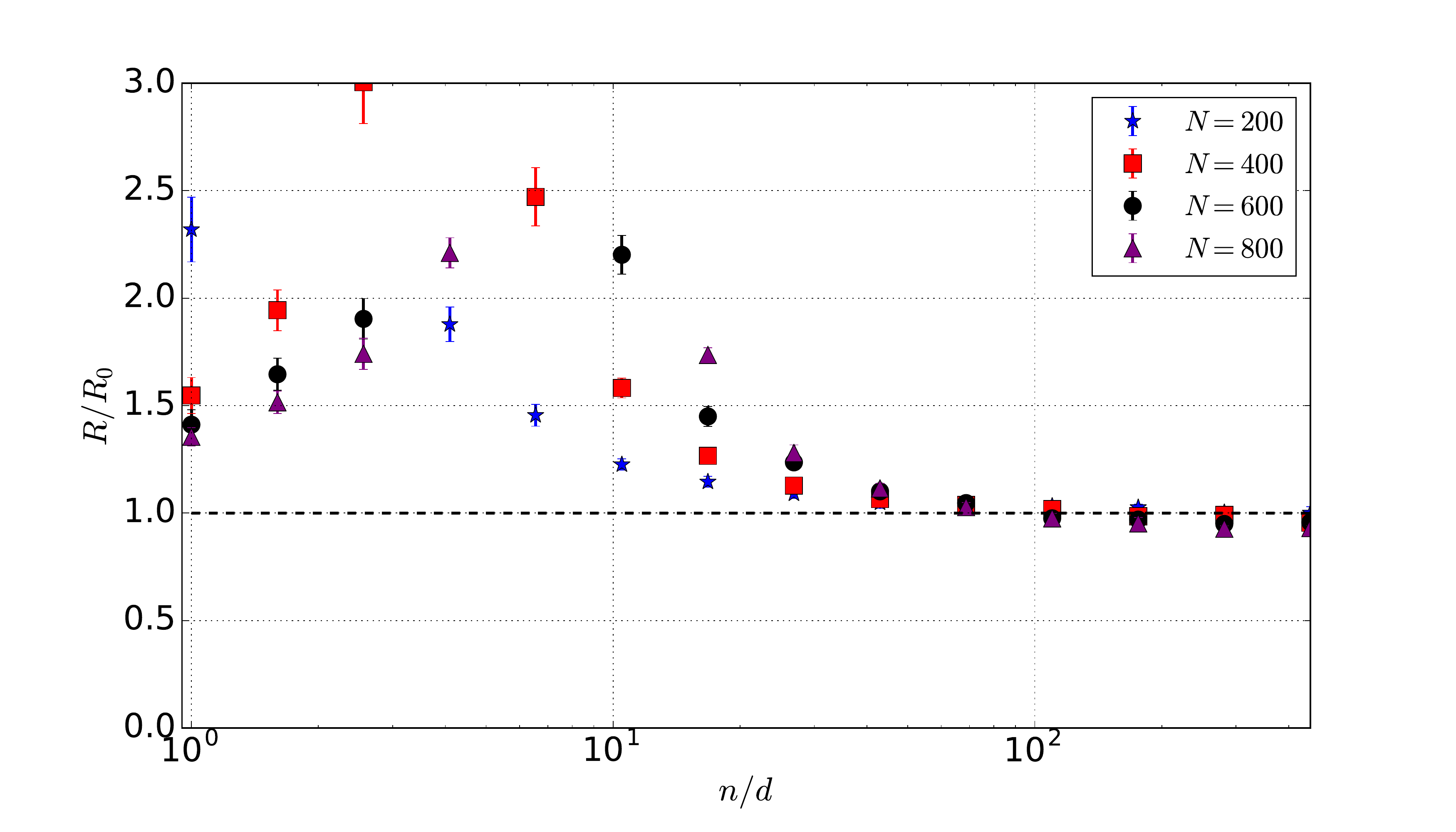}
\caption{Risk of the random features model for learning a quadratic function $f_{\star,2}$,
for $d=20$ (top left), $d=30$ (top right), $d=50$ (bottom left) and $d=100$ (bottom right).  We use least square to estimate
the model coefficients from $n$ samples and report the test error over $n_{\stest}=1500$ fresh samples.
 Data points correspond to averages over $10$ independent repetitions, and the risk is normalized by the risk $R_0$
of the trivial (constant) predictor. }\label{fig:RF-SecondDeg}
\end{figure}

\begin{figure}
\includegraphics[width=0.49\linewidth]{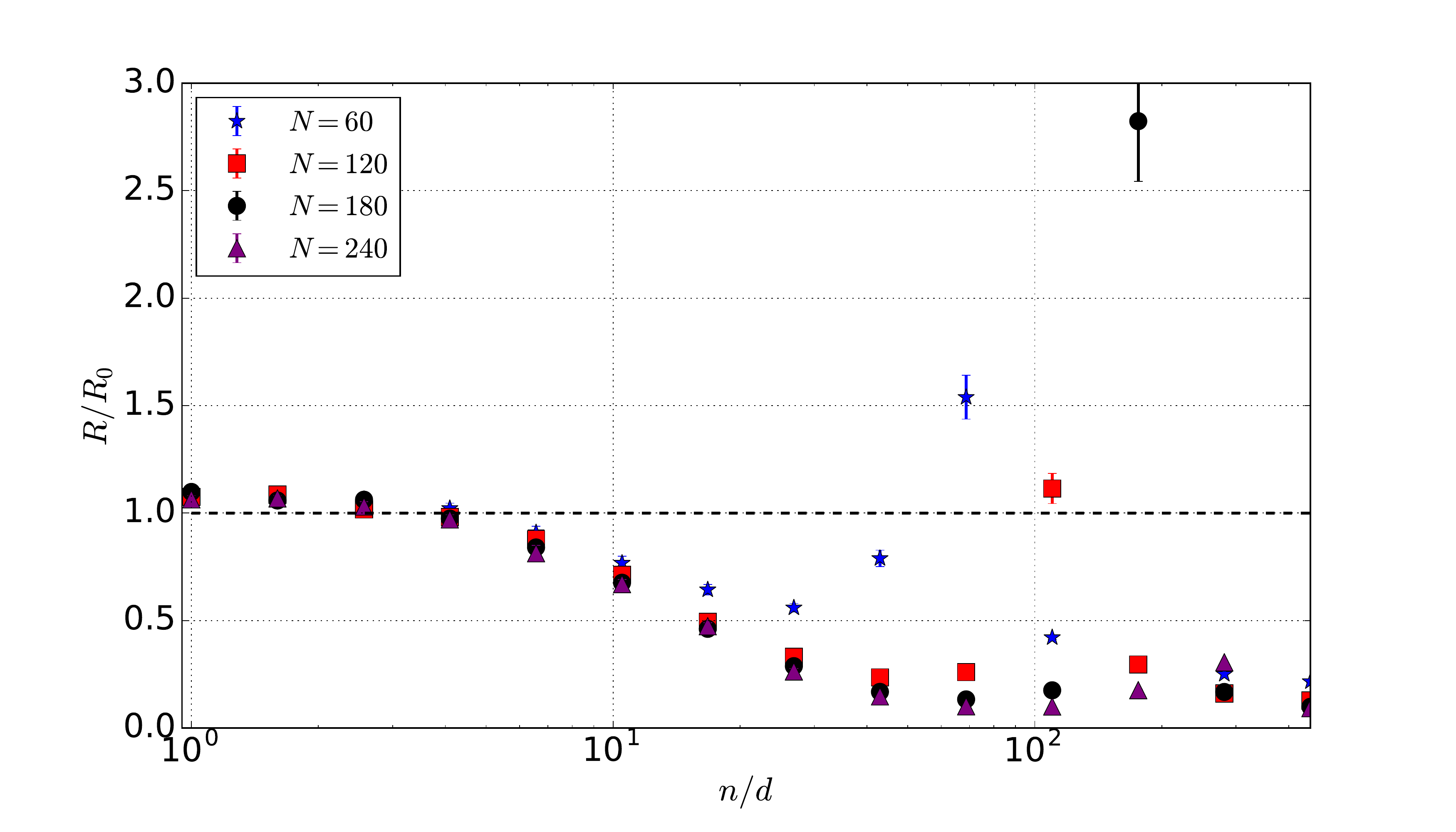}
\includegraphics[width=0.49\linewidth]{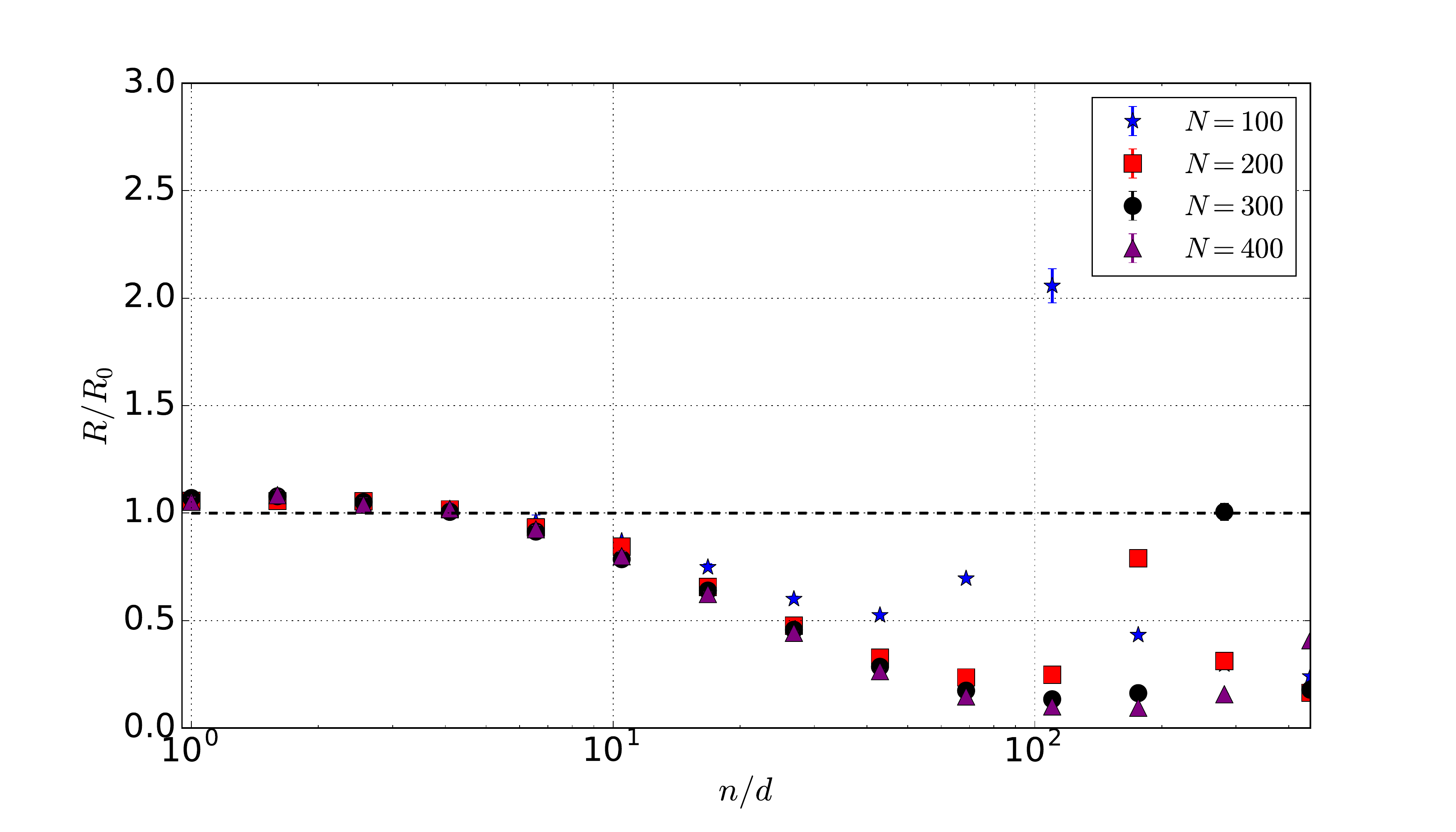}
\caption{Risk  (test error) of the neural tangent  model in learning a quadratic function $f_{\star,2}$,
for $d=30$ (left frame) and $d=50$ (right frame). The other settings are the same as in Figure \ref{fig:RF-SecondDeg}.}\label{fig:NT-SecondDeg}
\end{figure}
\begin{figure}
\includegraphics[width=0.49\linewidth]{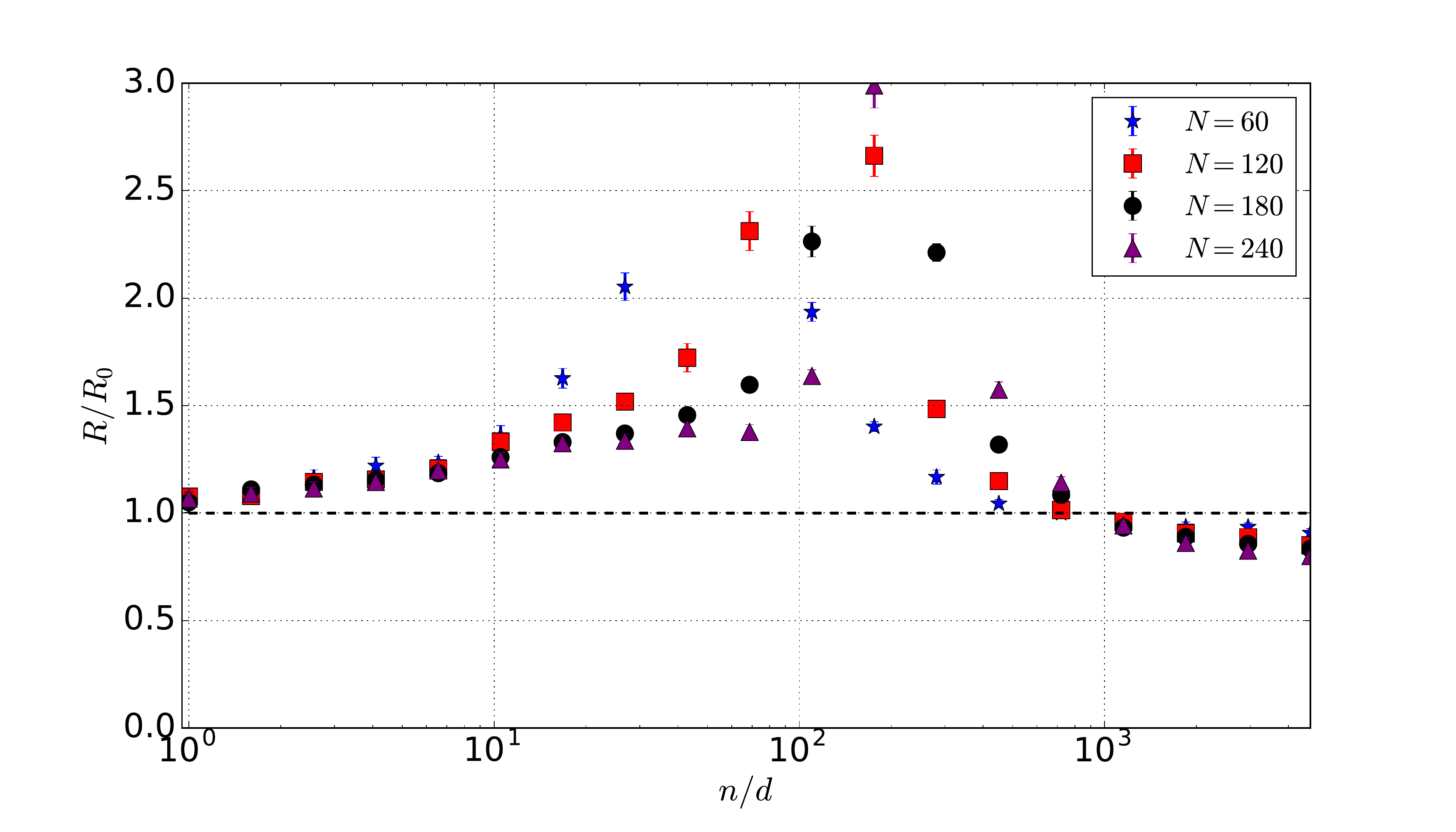}
\includegraphics[width=0.49\linewidth]{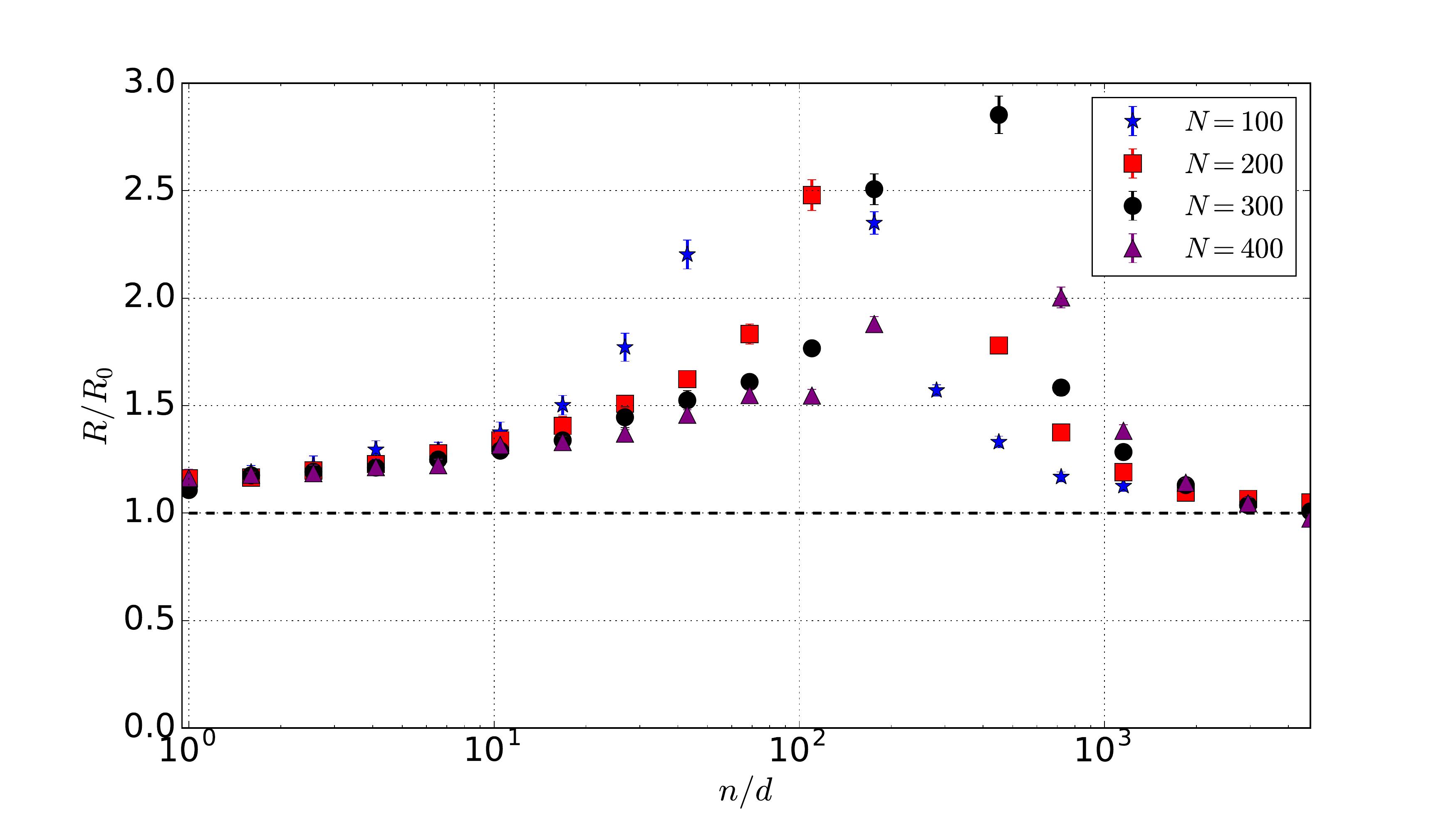}
\caption{Risk  (test error) of the neural tangent  model in learning a third order polynomial $f_{\star,3}$,
for $d=30$ (left frame) and $d=50$ (right frame). The other settings are the same as in Figures \ref{fig:RF-SecondDeg}
and \ref{fig:NT-SecondDeg}.}\label{fig:NT-ThirdDeg}
\end{figure}
In order to illustrate the approximation behavior of \RF\, and \NT\, models, we present a simple simulation study.
We consider feature vectors normalized so that $\|\bx_i\|_2^2= d$, and otherwise
uniformly random, and responses $y_i = f_{\star}(\bx_i)$, for a certain function $f_{\star}$. Indeed, this will be the setting
throughout the paper: $\bx_i\sim\Unif(\S^{d-1}(\sqrt{d}))$ (where $\S^{d-1}(r)$ denotes the sphere with radius $r$ in $d$ dimensions) and
$f_{\star}:\S^{d-1}(\sqrt{d})\to\reals$. We draw random weights $(\bw_i)_{i\le N}\sim_{iid} \Unif(\S^{d-1}(1))$. We use $n$ samples to fit a model 
in $\cF_{\RF}(\bW)$ or $\cF_{\NT}(\bW)$.
 We learn the model parameters using least squares. If the model is overparametrized, we select the minimum $\ell_2$-norm solution.
(We refer to Appendix \ref{app:Ridge} for simulations using ridge
regression instead.)
We estimate the risk (test error) using $n_{\stest}=1500$ fresh samples, and normalize it by the risk of the trivial
model $R_0 = \E\{f_{\star}(\bx)^2\}$.

Figures \ref{fig:RF-SecondDeg}, \ref{fig:NT-SecondDeg}, \ref{fig:NT-ThirdDeg} report the results of such a simulation using \RF\, --for Figure~\ref{fig:RF-SecondDeg}-- and \NT\,  --for Figures~\ref{fig:NT-SecondDeg} and \ref{fig:NT-ThirdDeg}.
We use shifted ReLU activations $\sigma(u) = \max(u-u_0,0)$, $u_0=0.5$.  The choice of $u_0 = 0.5$ is not essential: (Lebesgue-)almost every $u_0 \neq 0$ has similar behavior. In contrast, the case $u_0=0$ is degenerate because $\max(u,0)$ is equal
to a linear function plus an even function.

The target functions $f_{\star}$ in these examples are quite simple.
Figures \ref{fig:RF-SecondDeg} and \ref{fig:NT-SecondDeg} use a quadratic function
$f_{\star,2}(\bx) = \sum_{i\le \lfloor d/2\rfloor}x_i^2-
 \sum_{i>\lfloor d/2\rfloor}x_i^2$. In Figure \ref{fig:NT-ThirdDeg}, the target function is a third-order polynomial 
$f_{\star,3}(\bx) =\sum_{i=1}^{d}(x_i^3-3x_i)$.

The results are somewhat disappointing: in two cases (first and third figures) \RF\, and \NT\, models do not beat the trivial predictor. In one case (the second one),
the \NT\, model surpasses the trivial baseline, and it appears to decrease to $0$ as the number of samples $n$ increase.
We also note that the risk shows a cusp when $n\approx p$, with $p$ the number of parameters ($p=N$ for \RF, and $p = Nd$ for \NT).
This  phenomenon is related to overparametrization, and will not be discussed further in this paper 
(see \cite{belkin2018reconciling,belkin2019two,hastie2019surprises,mei2019generalization} for relevant work).
We will instead focus on the population behavior $n\to\infty$.

\begin{figure}
\includegraphics[width=0.49\linewidth]{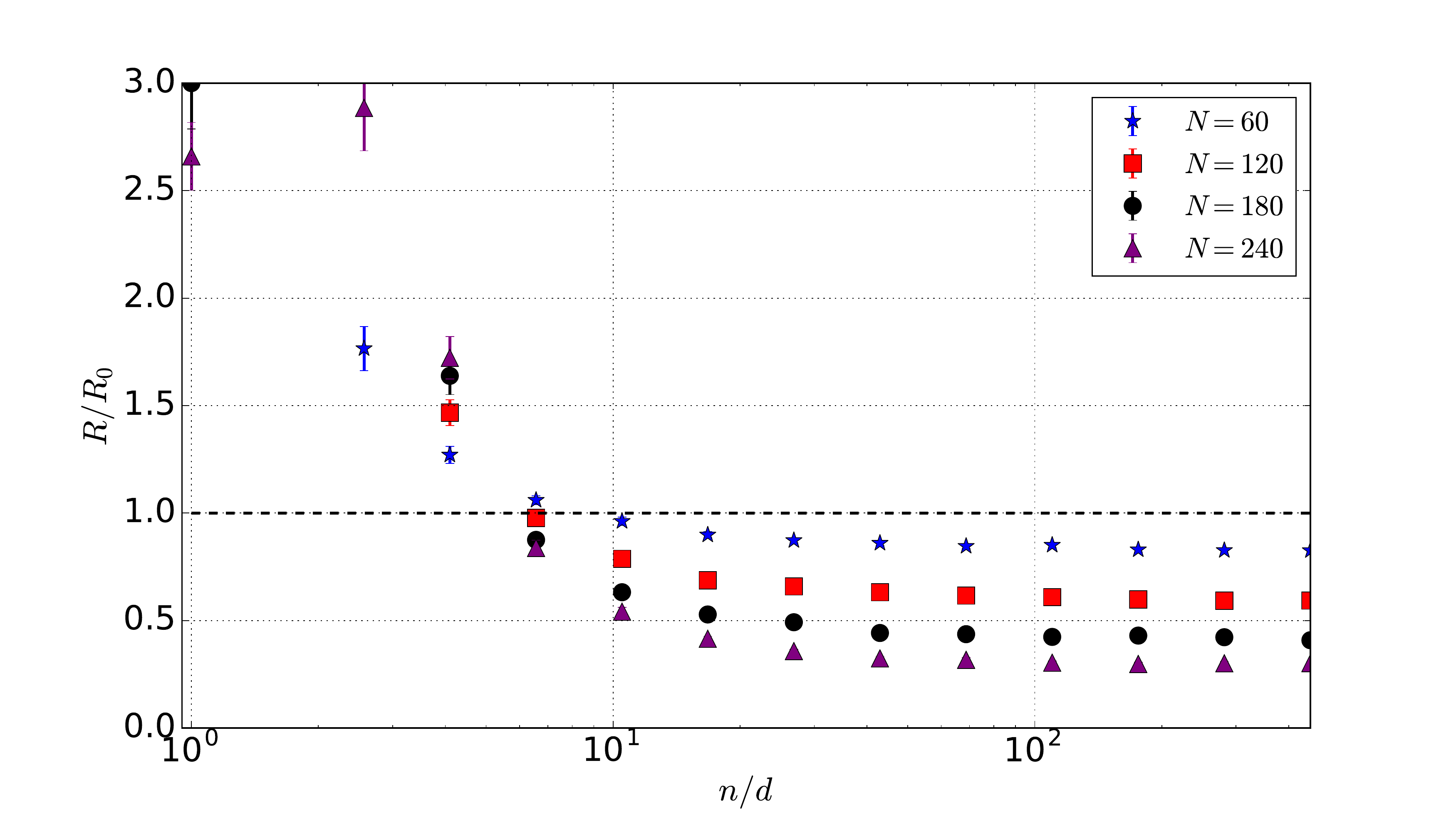}
\includegraphics[width=0.49\linewidth]{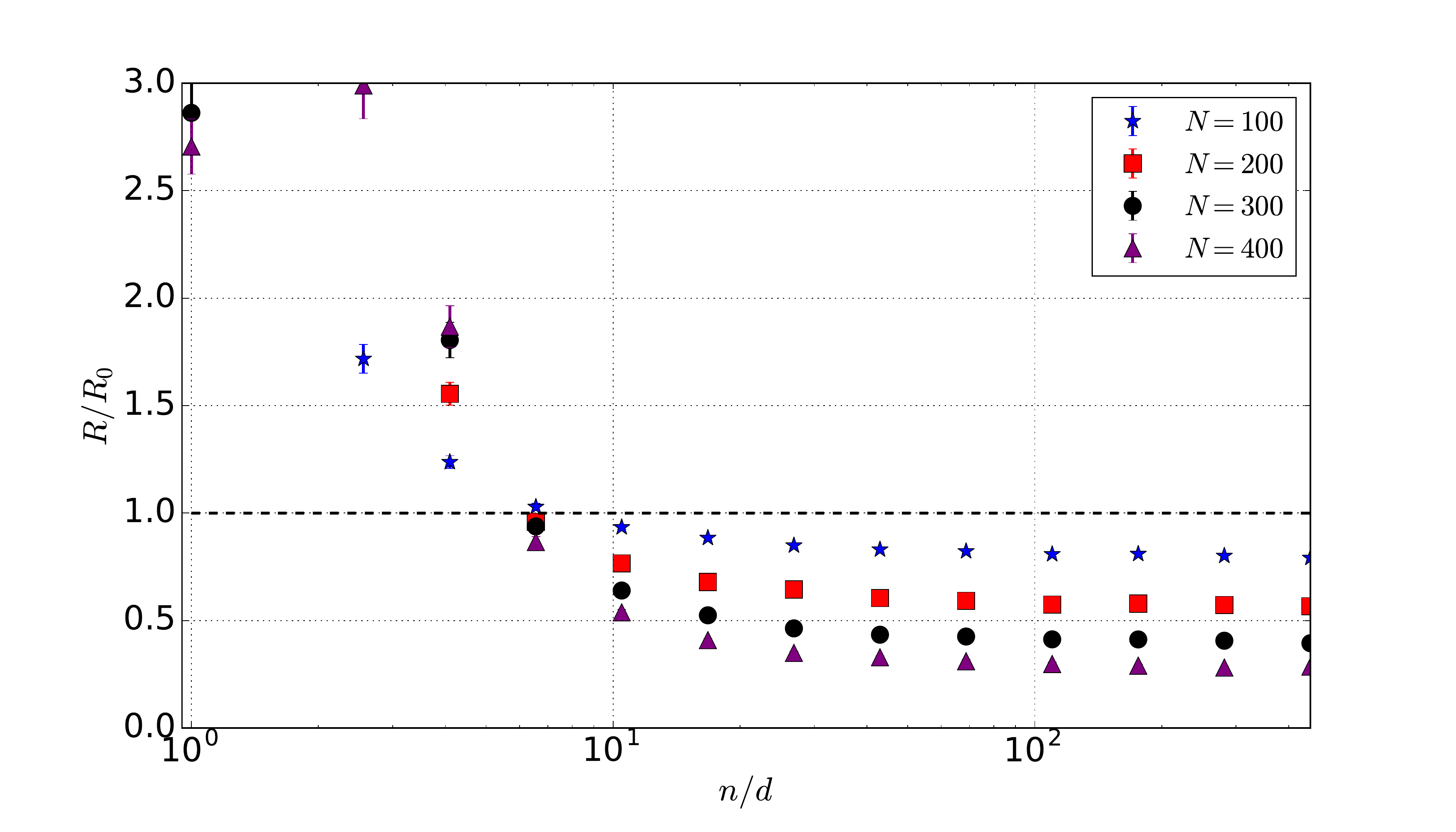}
\caption{Upper bounds on the optimal risk of the neural network model $\cF_{\NN}$ when used to learn the third order polynomial $f_{\star,3}$ (same target function 
as in Figure \ref{fig:NT-ThirdDeg}), for $d=30$ (left frame) and  $d=50$ (right frame). 
We use $n$ train samples and report the test error over $n_{\stest}=1500$ fresh samples.
 Data points correspond to averages over $50$ independent repetitions, and the risk is normalized by the risk $R_0$ of the trivial (constant) predictor. Training uses oracle knowledge of the function $f_{\star,3}$.}\label{fig:NN-SparseFeatures}
\end{figure}
 In other words, the \RF\, model does not appear to be able to learn  a simple quadratic function,
 and the \NT\, model does not appear to be able to learn a third order polynomial.
 Our main theorems (presented in the next sections) capture in a precise manner this behavior. In particular,
\begin{itemize}
\item We will prove that for $N= O_d(d^{2-\delta})$ , \RF\, does not outperform the trivial predictor on \emph{any function} that has
  vanishing projection on linear functions.  Similarly, \NT\, does not outperform the trivial predictor
  on \emph{any function} that has
  vanishing projection on  linear and quadratic functions.
\item  In contrast, there exists neural networks in $\cF_{\NN}$ with $N=O_d(d)$ neurons,
  and a small approximation error both for $f_{\star,2}$ and $f_{\star,3}$ (see, e.g., \cite{bach2017equivalence}, or \cite[Proposition 1]{mei2018mean}). 
\end{itemize}
These two points illustrate the gap in approximation power between \NT\, (or \RF) and \NN.

We demonstrate the second point empirically in Fig.~\ref{fig:NN-SparseFeatures} by choosing weight vectors $\bw_i = s_i\be_{r(i)}$, where 
$r(i)\sim \Unif([n])$ are i.i.d. uniformly random indices, and the scaling factor is  $s_i \sim\normal(0,1)$. Fixing these random 
first-layer weights, we fit the second-layer weights $a_i$ by least squares. The risk achieved is an upper bound on the minimum 
risk in the $\NN$\, model, namely $R_{\NN}(f_{\star}) \equiv \inf_{f\in\cF_{\NN}}\E\{(f_{\star}(\bx)-f(\bx))^2\}$,
and is significantly smaller than the baseline $R_0$.  
(The risk reported in Fig.~\ref{fig:NN-SparseFeatures} can also be interpreted as a `random features' risk. However,
the specific distribution of the vectors $\bw_i$ is tailored to the function $f_{\star}$, and hence not achievable within the \RF\, model.)

\subsection{Summary of main results}
\label{sec:Main}

\begin{description}
\item[Approximation error of \RF\, models.] If $d^{1 + \delta} < N \le  d^{2-\delta}$ for some $\delta > 0$, then the approximation error of \RF\, is asymptotically equivalent to the approximation error of fitting a linear function in the raw covariates $\bx$ (i.e. least squares with the model $f(\bx) = b_0+\<\bbeta, \bx\>$, $b_0\in\reals$, $\bbeta\in\reals^d$). More generally, if $d^{\ell + \delta}\le N \le d^{\ell+1-\delta}$, then \RF\, is equivalent to fitting a linear function over all monomials of degree at most $\ell$ in $\bx$.

  The equivalence between \RF\, regression and polynomial regression holds \emph{pointwise} for target function $f_\star$.
\item[Approximation error of \NT\, models.] If $d^{1 + \delta} \le N \le d^{2-\delta}$, then the approximation error of \NT\, is asymptotically equivalent to the approximation error of fitting a linear function over monomials of degree at most two in $\bx$ (i.e. least squares with the model $f(\bx) = b_0 + \<\bbeta,\bx\> + \<\bx,\bB\bx\>$, $b_0\in\reals$, $\bbeta\in\reals^d$, $\bB\in\reals^{d\times d}$). More generally, if $d^{\ell + \delta} \le N \le d^{\ell+1 - \delta}$, then \NT\, is equivalent to fitting a linear function over all monomials of degree at most $\ell+1$ in $\bx$.

Again, this result holds pointwise over the choice of $f_\star$.
\item[Generalization error of kernel methods.] We study the generalization error of kernel methods under the same data distribution described above, for any rotationally invariant kernel on the sphere $\S^{d-1}(\sqrt{d})$. 
  We prove two results:
  \begin{enumerate}
\item If the sample size is $n \le d^{\ell + 1 - \delta}$, then the generalization error of \emph{any kernel method} is lower bounded by the approximation error of linear regression over monomials of degree at most $\ell$ in $\bx$. 
\item If the sample size satisfies $d^{\ell + \delta} \le n \le d^{\ell+1-\delta}$, then the generalization error of Kernel Ridge Regression (\KRR) is given by the approximation error of linear regression over monomials of degree at most $\ell$ in $\bx$.
  \end{enumerate}
\end{description}
It is worth emphasizing two aspects of this last result. The first one is its generality. 
The \NT\, kernel associated to an infinitely wide  \emph{multi-layers} fully connected neural network is always rotational invariant (assuming an
i.i.d. Gaussian initialization of weights, which is common in practice). Therefore --in the \NT\, regime-- multi-layers neural networks
cannot outperform the trivial predictor on a target function $f_\star(\bx)$ that has vanishing projection onto degree-$\ell$
polynomials, unless the sample size satisfies $n\ge d^{\ell+1-\delta}$.
(For instance, they cannot outperform the trivial predictor  for $f_{\star}(\bx) =x_1^3-3x_1$
unless $n\ge d^{3-\delta}$.)

The second aspect can be summarized as follows.
\begin{description}
\item[Optimality of near interpolators.] For  $d^{\ell + \delta} \le n \le d^{\ell+1-\delta}$,
   the ideal behavior of KRR is achieved for all regularization values $\lambda\le \lambda_*$,
   with $\lambda_*$ depending on $N,d$ and the activation function. In particular, it is achieved by `near interpolators'
   (corresponding to $\lambda\approx 0$) i.e.  functions $\hf$ that have negligible training error.
\end{description}

\section{Approximation error of linearized neural networks}
\label{sec:Approximation}
  
In this section, we state formally our results about the approximation error of $\RF$ and $\NT$ models.
We define the minimum population error for any of the models $\sM\in\{\RF,\NT\}$ by
\begin{align}
R_{\sM}(f_\star, \bW) = \inf_{f\in \cF_{\sM}(\bW)} \E \big[(f_\star(\bx) - f(\bx))^2\big]\, ,\;\;\;\;  \sM\in\{\RF,\NT\}\, .
\end{align}
Notice that this is a random variable because of the random features encoded in the matrix $\bW\in\reals^{N\times d}$. Also, it depends implicitly on $d,N$, but
we will make this dependence explicit only when necessary. 

For $\ell\in\naturals$, we denote by $\proj_{\le \ell}: L^2(\S^{d-1}(\sqrt d))\to L^2(\S^{d-1}(\sqrt{d}))$ the orthogonal projector 
onto the subspace of polynomials of degree at most $\ell$. (We also let $\proj_{>\ell} = \id-\proj_{\le \ell}$.) In other words, $\proj_{\le \ell}f$ is
the function obtained by linear regression of $f$ onto monomials of degree at most $\ell$. 
Throughout this paper `with high probability' means `with probability converging to one
as $d,N\to\infty$'. The notations $s_d = \omega_d (t_d)$, $s_d = o_d(t_d)$, $s_d=O_d(t_d)$, $s_d=\Omega_d(t_d)$ mean, respectively,
$\lim_{d \to \infty} |s_d / t_d| = \infty$, $\lim_{d \to \infty} |s_d / t_d| = 0$, $\lim\sup_{d \to \infty} |s_d / t_d| <\infty$,
$\lim\inf_{d \to \infty} |s_d / t_d| > 0$. Given random variables $X_d$, and deterministic quantities $t_d$, we write $X_d=o_{d,\P}(t_d)$ (and so on) if the above holds in probability.

\subsection{Approximation error of random features models}
\begin{assumption}[Assumptions for the \RF\, model at level $\ell\in\naturals$]\label{ass:activation_lower_upper_RF_v2}
Let $\{ \sigma_d \}_{d \ge 1}$ be a sequence of functions $\sigma_d:\reals\to\reals$. 
\begin{itemize}
\item[(a)] $\sigma_d \in L^2([-\sqrt d, \sqrt d], \tau^1_{d-1})$, where $\tau^1_{d-1}$ is the distribution of $\< \bx, \be\>$ for $\bx \sim \Unif(\S^{d-1}(\sqrt d))$, and $\be = (1, 0, \ldots, 0)^\sT \in \R^d$. 
\item[(b)] We have
\[
\Big[ d^\ell \cdot \min_{k \le \ell} \lambda_{d,k}(\sigma_d)^2 \Big] / \| \sigma_d(\< \be, \cdot\>) \|_{L^2(\S^{d-1}(\sqrt d))}^2 = \Omega_d(1),
\]
where $\lambda_{d,k}(\sigma_d) = \< \sigma_d(\< \be, \cdot\>), Q_k(\sqrt{d} \< \be, \cdot\>) \>_{L^2(\S^{d-1}(\sqrt d))}$, and $Q_k$
is the $k$-th Gegenbauer polynomial (see Section \ref{sec:Background}).
\end{itemize}
\end{assumption}

\begin{theorem}[Risk of the \RF\, model]\label{thm:RF_lower_upper_bound}
Let $\{ f_d \in L^2(\S^{d-1}(\sqrt d))\}_{d \ge 1}$ be a sequence of functions. Let $\bW = (\bw_i)_{i \in [N]}$ with $(\bw_i)_{i \in [N]} \sim \Unif(\S^{d-1})$ independently. Then the following hold.
\begin{itemize}
\item[(a)] Assume $N \le d^{\ell+1-\delta_d}$ for a fixed integer $\ell$ and any sequence $\delta_d$ such that $\delta_d^2\log d\to\infty$
(in particular, $N\le d^{\ell+1-\delta}$ is sufficient for any fixed $\delta>0$). Let $\{ \sigma_d \}_{d \geq 1}$ satisfy Assumption \ref{ass:activation_lower_upper_RF_v2}.(a). Then, for any $\eps > 0$, the following holds with high probability: 
\begin{align}
\Big \vert R_{\RF}(f_{d}, \bW) - R_{\RF}(\proj_{\le \ell} f_d, \bW) - \| \proj_{> \ell} f_d \|_{L^2}^2 \Big \vert &\le \eps \| f_d \|_{L^2} \| \proj_{> \ell} f_d \|_{L^2}\, .
\label{eq:RF_lower_bound}
\end{align}
\item[(b)] Assume $N =\omega_d(d^{\ell})$ for some integer $\ell$, and $ \{ \sigma_d \}_{d \ge 1}$ satisfy Assumption \ref{ass:activation_lower_upper_RF_v2}.(b) at level $\ell$. Then for any $\eps > 0$, the following holds with high probability: 
\begin{align}\label{eqn:RF_upper_bound}
0 \le R_{\RF}(\proj_{\le \ell} f_d, \bW) \le \eps \| \proj_{\le \ell} f_d \|_{L^2}^2.
\end{align}
\end{itemize}
\end{theorem}
See Section \ref{sec:proof_RFK_lower} for the proof of lower bound, and Section \ref{sec:proof_RFK_upper} for the proof of upper bound.

In words, Eq.~\eqref{eq:RF_lower_bound} amounts to say that when $N = O_d(d^{\ell + 1 - \delta_d})$, the risk of the random feature model can be approximately decomposed in two parts, each non-negative,  and each with a simple interpretation:
\begin{align}
 R_{\RF}(f_{d}, \bW) \approx  R_{\RF}(\proj_{\le \ell} f_d, \bW) + \| \proj_{> \ell} f_d \|_{L^2}^2\, .
\end{align}
The second contribution, $\| \proj_{> \ell} f_d \|_{L^2}^2$ is simply the risk achieved by linear regression with respect to polynomials of degree at most $\ell$. The first contribution $R_{\RF}(\proj_{\le \ell} f_d, \bW)$ is the risk of the \RF\,  model when applied to the low-degree component of $f_d$. Equation~\eqref{eqn:RF_upper_bound} implies that when $N = \omega_d(d^{\ell})$, the first contribution $R_{\RF}(\proj_{\le \ell} f_d, \bW)$ vanishes asymptotically. 

If both Assumptions \ref{ass:activation_lower_upper_RF_v2}.$(a)$ and \ref{ass:activation_lower_upper_RF_v2}.$(b)$ hold
and $\omega_d ( d^\ell ) \leq N \leq O_d (d^{\ell + 1 - \delta})$ for some integer $\ell$, we thus obtain
\[
 R_{\RF}(f_{d}, \bW)  =  \| \proj_{> \ell} f_d \|_{L^2}^2 + \| f_d \|_{L^2}^2 \cdot o_{d,\P} (1).
\]
In particular, this shows that \RF\, fits a linear function over polynomials of maximum degree $\ell$.

\begin{remark}\label{rmk:GeneralityRF}
Note that Theorem \ref{thm:RF_lower_upper_bound}.$(a)$ holds under very weak conditions on the activation function, which may depend on the dimension $d$.
The condition   $\sigma_d ( \< \be_1 , \cdot \> ) \in L^2 ( \S^{d-1}  (\sqrt{d} ) )$ can also be rewritten as $\sigma_d\in L^2(\reals, \tau^1_{d-1})$, where
 $\tau^1_{d-1}$ is the one-dimensional projection of the uniform measure over $\S^{d-1}(\sqrt{d})$.  In particular:
\begin{enumerate}
\item[$(i)$] $\tau^1_{d-1}$ is supported on $[-\sqrt{d},\sqrt{d}]$. It is therefore sufficient that $\sup_{|u|\le \sqrt{d}}|\sigma_d(u)| = C_1(d)<\infty$.
\item [$(ii)$] By an explicit calculation, the density of  $\tau^1_{d-1}(\de u) =C_2(d)(1-u^2/d)^{(d-3)/2}\de u$. Since this density is bounded, it is sufficient that $\sigma_d$
is square integrable with respect to the Lebesgue measure on $[-\sqrt{d},\sqrt{d}]$. 
\end{enumerate}
\end{remark}

\begin{remark}\label{rmk:upperBoundConditions}
If the activation $\sigma$ is independent of $d$, Assumption \ref{ass:activation_lower_upper_RF_v2}.$(b)$ is satisfied as long as $\mu_k ( \sigma) \neq 0 $ for $k = 0, \ldots, \ell$, where $\mu_k ( \sigma)$ is the $k$-th Hermite coefficient of $\sigma$ (see Section \ref{sec:Background} for definitions).
\end{remark}

\begin{remark}\label{rmk:SimplerBound}
  The conclusion of Theorem \ref{thm:RF_lower_upper_bound}.$(a)$ can be established\footnote{A first version of this manuscript, posted
    on arXiv, assumed such conditions. } 
by a somewhat simpler proof if the activation function $\sigma$
is independent of $d$ and satisfies the following regularity conditions: $(i)$ $\sigma(u)^2\le c_0\exp(c_1u^2/2)$ for some $c_1<1$;
$(ii)$ $\sigma$ is not a polynomial of degree smaller than $2\ell+3$. Under these conditions, the conclusion holds for $N= o_d(d^{\ell+1})$.
\end{remark}

Note that Assumption \ref{ass:activation_lower_upper_RF_v2}.$(b)$  requires in particular that $\sigma$ is not a polynomial of degree strictly smaller than $\ell$.
 This is easily seen to be a necessary condition, since any linear
combination of polynomials of degree $k<\ell$ is a polynomial of degree $k$. For the same reason, this
condition also arises in the approximation theory of neural networks \cite{pinkus1999approximation}.

\subsection{Approximation error of neural tangent models}

For the \NT\, model, the proof, while following the same scheme as for  \RF, is more challenging. We restrict our setting to a fixed activation function
$\sigma$ (independent of dimensions) which is weakly differentiable, with weak derivative $\sigma'$ that does not grow too fast (in particular, exponential growth is fine). We further require the Hermite decomposition of $\sigma'$ to satisfy a mild `genericity' condition. Recall that the
$k$-th Hermite coefficient of a function $h$ can be defined as  $\mu_k(h) \equiv \E_{G\sim\normal(0,1)}\{h(G) \bbHe_{k}(G)\}$, where $\bbHe_k(x)$ is the $k$-th Hermite polynomial (see Section \ref{sec:Background} for further background).

\begin{assumption}[Assumptions for the \NT\,model at level $\ell\in\naturals$.]\label{ass:activation_lower_upper_NT_v2}
Let $\sigma$ be an activation function $\sigma:\reals\to\reals$.
\begin{itemize}
\item[(a)] The function $\sigma$ is weakly differentiable, with weak derivative $\sigma'$ such that $\sigma'(u)^2\le c_0\exp(c_1u^2/2)$ for some constants $c_0, c_1$, with $c_1<1$.
\item[(b)] The Hermite coefficients $\{\mu_k(\sigma')\}_{k\ge 0}$ are such that there exist $k_1,k_2\ge 2\ell+7$ such that $\mu_{k_1}(\sigma') ,\mu_{k_2}(\sigma')\neq 0$  and
\begin{align}
\frac{\mu_{k_1}(x^2\sigma')}{\mu_{k_1}(\sigma')}\neq \frac{\mu_{k_2}(x^2\sigma')}{\mu_{k_2}(\sigma')} \, .
\end{align}
\item[(c)] The Hermite coefficients of $\sigma$ satisfy $\mu_k ( \sigma) \neq 0$ for any $k \leq \ell + 1$. 
\end{itemize}
\end{assumption}

\begin{theorem}[Risk of the \NT\, model]\label{thm:NT_lower_upper_bound}
Let $\{ f_d \in L^2(\S^{d-1}(\sqrt d))\}_{d \ge 1}$ be a sequence of functions. Let $\bW = (\bw_i)_{i \in [N]}$ with $(\bw_i)_{i \in [N]} \sim \Unif(\S^{d-1})$ independently. We have the following results. 
\begin{itemize}
\item[(a)] Assume $N = o_d(d^{\ell+1})$ for a fixed integer $\ell$, and let $\sigma$ satisfy Assumptions \ref{ass:activation_lower_upper_NT_v2}.(a) and \ref{ass:activation_lower_upper_NT_v2}.(b) at level $\ell$. Then, for any $\eps > 0$, the following holds with high probability: 
\begin{align}
\Big \vert R_{\NT}(f_{d}, \bW) - R_{\NT}(\proj_{\le \ell+1} f_d, \bW) - \| \proj_{> \ell+1} f_d \|_{L^2}^2 \Big \vert &\le \eps \| f_d \|_{L^2} \| \proj_{> \ell+1} f_d \|_{L^2}\, .
\label{eqn:NT_lower_bound}
\end{align}
\item[(b)] Assume $N =\omega_d(d^{\ell})$ for some integer $\ell$, and let $ \sigma$ satisfy Assumptions \ref{ass:activation_lower_upper_NT_v2}.(a) and \ref{ass:activation_lower_upper_NT_v2}.$(c)$ at level $\ell$. Then for any $\eps > 0$, the following holds with high probability: 
\begin{align}
0 \le R_{\NT}(\proj_{\le \ell+1} f_d, \bW) \le \eps \| \proj_{\le \ell+1} f_d \|_{L^2}^2. 
\end{align}
\label{eqn:NT_upper_bound}
\end{itemize}
\end{theorem}

See Section \ref{sec:proof_NTK_lower} for the proof of lower bound, and Section \ref{sec:proof_NTK_upper} for the proof of upper bound.

\begin{remark}
It is easy to check that Assumptions \ref{ass:activation_lower_upper_NT_v2}.$(a)$ and  \ref{ass:activation_lower_upper_NT_v2}.$(b)$ hold for all $\ell$, 
for all commonly used activations.

For instance the ReLU activation $\sigma(u) = \max(u,0)$ and its weak derivative $\sigma'(x)=\bfone_{x\ge 0}$ have subexponential growth. Further its Hermite coefficients are $\mu_0(\sigma') = 1/2$ and
\begin{align}
\mu_k(\sigma') = \frac{(-1)^{(k-1)/2}}{\sqrt{2 \pi}}\, (k-2)!! \, \bfone_{k\, \mbox{\tiny\rm odd}}\, .
\end{align}
which satisfy the required condition of Theorem \ref{thm:NT_lower_upper_bound}.$(a)$ for each $\ell$. (In checking the condition, it might be useful to 
notice the relation $\mu_k(x^2\sigma') =\mu_{k+2}(\sigma')+ (2k +1 ) \mu_{k} (\sigma') + k(k-1) \mu_{k-2} (\sigma')$.)

Assumption  \ref{ass:activation_lower_upper_NT_v2}.(c) does not hold for ReLU activation $\sigma(u) = \max(u,0)$,
since $\mu_k(\sigma) = 0$ for $k$ even. However it holds for shifted ReLU $\sigma(u) = \max(u-u_0,0)$,
for a generic value of the shift $u_0$.
\end{remark}

\begin{figure}[!ht]
\centering
\includegraphics[width=0.6\linewidth]{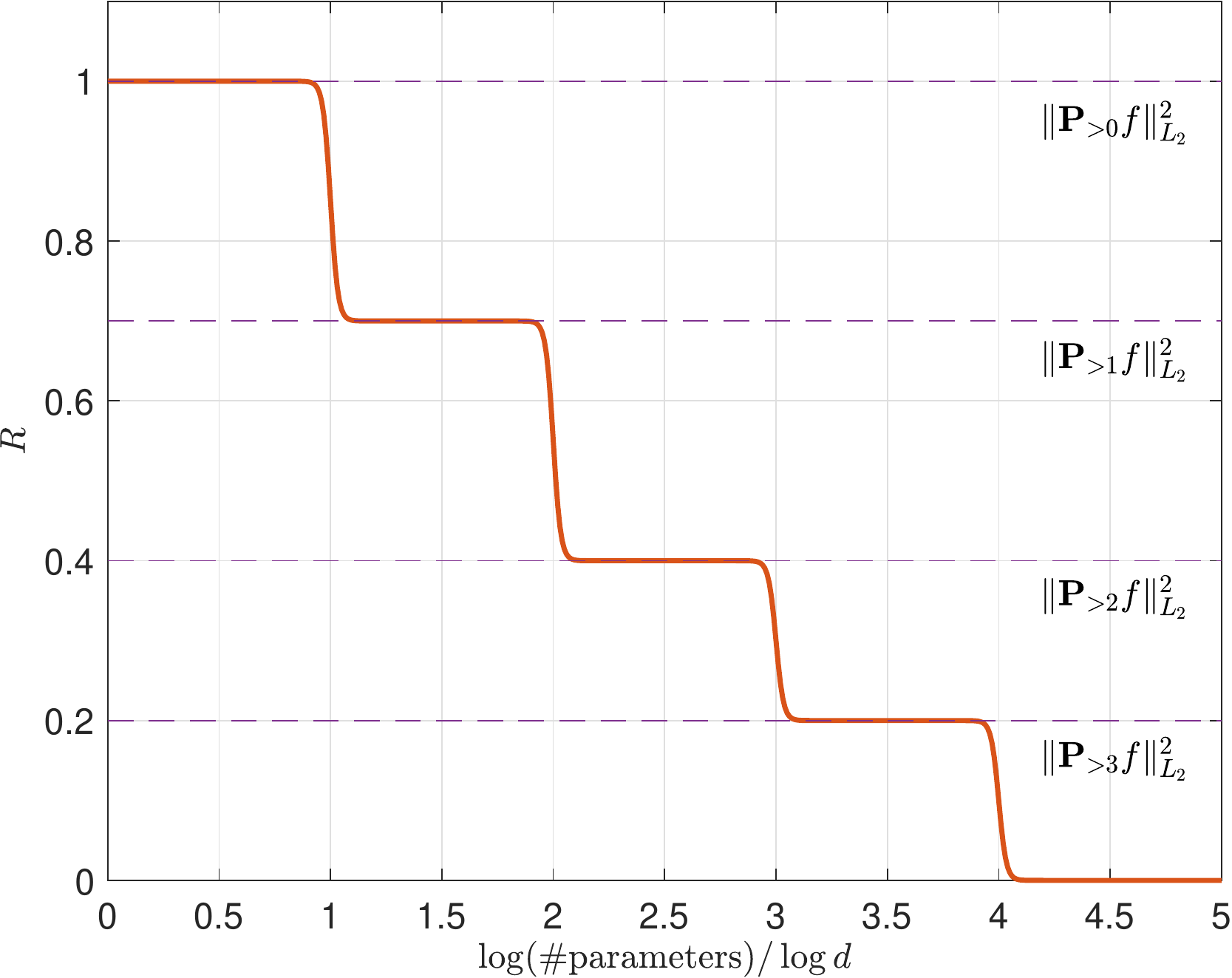}
\caption{A cartoon of the approximation error versus number of parameters in the \RF\, and \NT\, models.} \label{fig:cartoon_staircase}
\end{figure}

Theorems \ref{thm:RF_lower_upper_bound} and \ref{thm:NT_lower_upper_bound} can be illustrated by a cartoon, which we show
as Figure \ref{fig:cartoon_staircase}. In words, the approximation error plotted as a function of
$\log(\# \text{parameters})  / \log d$ follows a staircase: it drops close to integer values of this ratio, with each drop
corresponding to the projection onto homogeneous polynomials of that degree. We can extract three useful statistical insights
from these findings:
\begin{enumerate}
\item There is no difference between plain \RF\, and the more recent \NT\, approach in terms of approximation error,
  once we compare them at fixed number of parameters $p$. All that changes is the relation between number of parameters
  and number of neurons: $p=N$ for \RF, and $p=Nd$ for \NT. The recent work \cite{ghorbani2019limitations} actually shows
  some advantage for the \RF\,
  model, although in a special case.
  It is worth mentioning that the same equivalence holds when we consider the dependence on the sample size $n$, at $N=\infty$,
  see Section \ref{sec:Kernel}. 

  We notice however an important computational advantage for \NT, at constant parameters number. Indeed, the complexity at prediction
  time is $O(Nd)=O(p)$ for \NT, while it is $O(Nd) = O(pd)$ for \RF. 
\item \RF\, or \NT\, models behave similarly to expansions into orthogonal monomial basis. Also in that case, if only  $o_d(d^{\ell+1})$ basis elements are included,
  for a `typical' functions $f_\star$, the approximation error\footnote{Here by `typical' function we mean the following. Choose a function $f_{0,\star}\in L^2(\S^{d-1}(\sqrt{d})$, draw a Haar distributed orthogonal matrix $\bS\in\reals^{d\times d}$, and set $f_\star(\bx) = f_{0,\star}(\bS\bx)$.} is $\|\proj_{>\ell}f_{\star}\|^2$. 
\item Our results also suggest interesting directions to improve random feature expansions.
  First, if $f_{\star}$ is known to primary depends on a small subset of $d_1\ll d$ directions in $\reals^d$, there will be a significant advantage in choosing the random features
 along that $d_1$-dimensional subspace. 
Second, if the data points $\bx_i$ lie close to to such a subspace $V\subseteq \reals^d$, $\dim(V) = d_1$, one might hope that --even if the $\bw_i$
are sampled isotropically in $\reals^d$-- random feature methods will be sensitive to $d_1$ rather than $d$. We plan to report on these topics in a future
publication~\cite{OursUnpub}.
\end{enumerate}

\subsection{Separation between \NN\, and \RF,\, \NT}
\label{sec:Separation}

Theorems \ref{thm:RF_lower_upper_bound} and \ref{thm:NT_lower_upper_bound} imply a separation of approximation
power between two-layers neural networks and their linearization. As a simple example, consider the target function
$f_{\star}(\bx)=\sigma(\<\bw_{\star},\bx\>)$, for $\|\bw_{\star}\|_2=1$.
This can be represented exactly by a neural network with $N=1$, i.e. by a single neuron.
On the other hand, the above results imply that any \RF\, or \NT\, model is bound to have a non-vanishing population error,
if $d^{\ell+\delta}\le N\le d^{\ell+1-\delta}$. Provided $\sigma$ satisfies the Assumptions \ref{ass:activation_lower_upper_RF_v2},
\ref{ass:activation_lower_upper_NT_v2}, we get
\begin{align}
  R_{\RF}(\sigma;\bW)  = \|\sigma_{>\ell}\|^2_{L^2(\reals,\gamma)} +o_{d,\P}(1)\, ,\;\;\;\;\;\;
  R_{\NT}(\sigma;\bW) = \|\sigma_{>\ell+1}\|^2_{L^2(\reals,\gamma)} +o_{d,\P}(1)\, .\label{eq:Separation}
  \end{align}
  Here $\sigma_{>k}(x)$ is the projection of $\sigma$ orthogonal to the subspace of polynomials of maximum degree $k$,
  in  $L^2(\reals,\gamma)$, where $\gamma(\de x) = e^{-x^2/2}\de x/\sqrt{2\pi}$  is the standard Gaussian measure.

  Crucially,
  as proven in \cite{mei2016landscape}, running gradient descent over the space of neural networks consisting of a single neuron
  allows to learn the target function $f_\star(\bx) =\sigma(\<\bw_{\star},\bx\>)$ efficiently.
  In other words, we do not have simply a separation between the function classes $\cF_{\NN}$ and $\cF_{\RF}$ or $\cF_{\NT}$,
  but a separation between linearized neural networks, and neural networks trained by gradient descent.
  
  Essentially the same example was independently considered by Yehudai and Shamir in concurrent work \cite{yehudai2019power}.
  These authors prove that there exist finite constants $c_0,c_1>0$ such that, if  $N\le \exp\{c_1d\}$ and the coefficients
  $a_i, \ba_i$
  have magnitude at most $\exp\{c_1d\}$, then  \emph{there exists} a vector $\bw_\star$ such that, setting $f_{\star}(\bx) = \sigma(\<\bw_{\star},\bx\>)$, then
  $R_{\RF}(f_*;\bW), R_{\NT}(f_*;\bW)\ge c_0$. An important difference with respect to our separation result 
  is in the fact that Eq.~\ref{eq:Separation} holds --once again-- pointwise, i.e. for any fixed $\bw_\star$, while in
  \cite{yehudai2019power} $\bw_\star$ is chosen by an adversary who has knowledge of the vectors $(\bw_i)_{i\le N}$.
  Let us emphasize there are other important differences between our setting and the one of \cite{yehudai2019power},
  and neither of the two analysis implies the other.
  
  The same blueprint can be followed to prove further separation results. For instance, consider
  $f_{\star}(\bx) = \varphi(\bQ^{\sT}\bx)$, for $\bQ\in\reals^{d\times r}$ an orthogonal matrix and $\varphi: \reals^r\to\reals$
  a bounded smooth function, which is not a polynomial. If $r$ is kept constant as $d^{\ell+\delta}\le N\le d^{\ell+1-\delta}$, Theorems
  \ref{thm:RF_lower_upper_bound} and \ref{thm:NT_lower_upper_bound} can be used to show that
  $R_{\RF}(f_*;\bW), R_{\NT}(f_*;\bW)$ are bounded away from zero and to compute their limits.
  On the other hand, by classical results \cite{maiorov1999best} can be used to show that such $f_*(\bx)$ can be approximated arbitrarily
  well by neural networks with $O_d(1)$ neurons (with first layer weights $\bw_i$ in the span of columns of $\bQ$).
  Unfortunately, we are not aware of general results implying that such neural networks can be learnt by gradient descent,
although we expect this to be the case for certain choices of $\varphi$. Whenever such a result is available, 
  it implies a separation between \RF, \NT, and practical neural networks.

\section{Generalization error of kernel methods}
\label{sec:Kernel}


We consider next the limit of very wide networks. Namely,  we let $N\to\infty$ before $n,d\to\infty$. It is known since the work of Rahimi and
Recht \cite{rahimi2008random} that ridge regression over the function class $\cF_{\RF}(\bW)$ converges in this limit
to kernel ridge regression (KRR) with respect to the kernel (here expectation is with respect to $\bw\sim\Unif(\S^{d-1}(1))$)
\begin{align}
  H^{\RF}_d\big(\bx_1,\bx_2\big) := h^{\RF}_d\big(\<\bx_1,\bx_2\>/d\big) = \E\{\sigma(\<\bw,\bx_1\>)\sigma(\bw,\bx_2\>)\big\}\, .
  \label{eq:HRF}
 \end{align}
 Analogously, ridge regression in $\cF_{\NT}(\bW)$ can be shown to converge to KRR with respect to the kernel
\begin{align}
  H^{\NT}_d\big(\bx_1,\bx_2\big) := h^{\NT}_d\big(\<\bx_1,\bx_2\>/d\big) = (\<\bx_1,\bx_2\>/d)\E\{\sigma'(\<\bw,\bx_1\>)\sigma'(\bw,\bx_2\>)\big\}\, .\label{eq:HNT}
 \end{align}
 We will denote the corresponding RKHS by $\cH_{\RF}$ and $\cH_{\NT}$.
 Quantitative estimates on the relation between $\cF_{\RF}(\bW)$ and $\cH_{\RF}$ are obtained in
 \cite{bach2017equivalence}, which shows that the unit ball of $\cH_{\RF}$ is well approximated by the unit ball of
 $\cF_{\RF}(\bW)$ (endowed with the $\ell_2$ norm of the coefficients $(a_i)_{i\le N}$), for $N$ large enough.

 Notice that both kernels $H^{\RF}_d$, $H^{\NT}_d$ are rotationally invariant, namely
 $H_d(\bS\bx_1,\bS\bx_2) = H_d(\bx_1,\bx_2)$ for $H_d\in\{H^{\RF}_d, H^{\NT}_d\}$ and any $d\times d$ orthogonal matrix $\bS$.
 Any rotationally invariant kernel on the sphere $\S^{d-1}(\sqrt{d})$ takes the form
 \begin{align}
H_d(\bx_1, \bx_2) = h_d(\< \bx_1, \bx_2\> / d), \label{eq:GeneralKernel}
 \end{align}
 for some function $h_d:[-1,1]\to \reals$. 
 (The scaling factor $d$ is introduced here to make contact with the normalization used in previous sections, and is not necessary: indeed, $h_d$ can depend itself on $d$.)

 Our results apply to general rotational invariant kernels under very weak conditions on the function $h_d$.
 In particular, they apply to \emph{multilayer neural networks} in the neural tangent regime.
 Namely consider a $L$-layers network with matrix weights $\bW_1\in\reals^{N_1\times d}$, $\bW_2\in\reals^{N_2\times N_1}$,
 \dots $\bW_{L-1}\in\reals^{N_{L-1}\times N_{L-2}}$, $\ba\in\reals^{N_{L-1}}$. As long as all the weights are initialized
 as independent centered  Gaussians, with variance dependent only on the layer, the resulting \NT\, kernel is rotationally
 invariant.
 The recent papers  \cite{du2018gradient,du2018gradient2,allen2018convergence,zou2018stochastic,arora2019fine}
 provide conditions under which the \NT\, approximation is accurate for SGD-trained multilayer  neural networks.

 Section \ref{sec:KRR} presents a lower bound on the prediction error of general kernel methods,
 and Section \ref{sec:UpperKRR} derives an upper bound for kernel ridge regression.

 Throughout this section, we consider the same data model as in the previous sections: we observe pairs  $(y_i, \bx_i)_{i \in [n]}$, with $(\bx_i)_{i \in [n]} \sim \Unif(\S^{d-1} (\sqrt d))$, and $y_i = f_\star(\bx_i) +\eps_i$,
 $f_{\star}\in L^2(\S^{d-1} (\sqrt d))$ and $\eps_i \sim \normal ( 0 ,\tau^2)$ independently.
 
\subsection{Lower bound for general kernel methods}
\label{sec:KRR}

Consider any regression method of the form 
\begin{align}
\hf_{\lambda} = \arg\min_{f}\left\{\sum_{i=1}^n\ell(y_i,f(\bx_i)) +\lambda\|f\|_H^2\right\}\, ,\label{eq:KernelMethods}
\end{align}
where $\|f\|_H$ is the reproducing kernel Hilbert space (RKHS) norm with respect to the kernel $H$ of the form
\eqref{eq:GeneralKernel}. By the representer theorem \cite{berlinet2011reproducing} there exist coefficients
$\hat a_1,\dots, \hat a_n$ such that 
\begin{align}
\hf_{\lambda}(\bx) = \sum_{i=1}^n \hat a_i \, h_d(\<\bx,\bx_i\>/d)\, .  \label{eq:KRRf}
\end{align}
We are therefore led to define the following data-dependent prediction risk function for kernel methods
\begin{align}
R_{H}(f_{\star},\bX) \equiv \min_{\ba} \E_\bx\Big\{ \Big(f_\star(\bx) - \sum_{i=1}^n a_i  h_d(\< \bx_i, \bx\>/d) \Big)^2 \Big\}.
\end{align}

The next theorem provides a decomposition of this generalization error that is analogous to the one given in Theorem \ref{thm:RF_lower_upper_bound}.$(a)$.
Notice however that the controlling factor is not the number of neurons $N$, but instead the sample size $n$.

\begin{theorem}\label{thm:KR-general-main}
Assume $n \le d^{\ell+1-\delta_d}$ for a fixed integer $\ell$ and any sequence $\delta_d$ such that $\delta_d^2\log d\to\infty$
(in particular, $n\le d^{\ell+1-\delta}$ is sufficient for any fixed $\delta>0$).
Let $\{ f_d \in L^2(\S^{d-1}(\sqrt d))\}_{d \ge 1}$ be a sequence of functions,  $\{\bx_i\}_{i \in [n]} \sim_{iid} \Unif(\S^{d-1}(\sqrt d))$ 
with $y_i = f_d(\bx_i)$.  Assume $  h_n (\< \be_1, \,\cdot\,\>/\sqrt{d}) \in L^2 ( \S^{d-1} (\sqrt{d}))$. Then for any $\eps > 0$, with high probability as $d \to \infty$, we have 
\begin{align}
\left|R_{H}(f_{d},\bX) - R_{H}(\proj_{\le\ell} f_{d},\bX) - \| \proj_{> \ell} f_d \|_{L^2}^2 \right|\le \eps \| f_d \|_{L^2} \| \proj_{> \ell } f_d \|_{L^2}.  
\end{align}
\end{theorem}
\begin{proof}
  This follows immediately from Theorem \ref{thm:RF_lower_upper_bound}.$(a)$. Indeed, setting $\sigma_d(u) = h_d(u/\sqrt{d})$
  and $\bw_i = \bx_i/\sqrt{d}$, we obtain
$R_{H}(f_{d},\bX) =R_{\RF}(f_d,\bW)$, whence the claim follows by applying Eq.~(\ref{eq:RF_lower_bound}).
\end{proof}

\subsection{Upper bound for kernel ridge regression}
\label{sec:UpperKRR}

Kernel ridge regression is one specific way of selecting the coefficients $\hba$ in Eq.~\eqref{eq:KRRf}, namely
by using $\ell(\hy,y) = (\hy-y)^2$ in Eq.~\eqref{eq:KernelMethods}. Solving for the coefficients yields
\[
\hba = (\bH + \lambda \id_n)^{-1} \by,
\]
where the kernel matrix $\bH = (H_{ij})_{ij \in [n]}$ is given by
\[
H_{ij} =  h_d(\<\bx_i, \bx_j\>/d),
\]
and $\by = (y_1, \ldots, y_n)^\sT$. 
The prediction function at location $\bx$ is given by
\[
\hat f_\lambda(\bx) = \by^\sT (\bH + \lambda \id_n)^{-1} \bh(\bx),
\]
where
\[
\bh(\bx) = [ h_d(\< \bx, \bx_1\>/d), \ldots, h_d(\< \bx, \bx_n\>/d)]^\sT. 
\]
The test error of empirical kernel ridge regression is defined as
\[
\begin{aligned}
R_\KR(f_d, \bX, \lambda) \equiv& \E_\bx\Big[ \Big(f_d(\bx) - \by^\sT (\bH + \lambda \id_n)^{-1} \bh(\bx) \Big)^2 \Big]. \\
\end{aligned}
\]
We assume that $\{h_d \}_{d \ge 1}$ are positive-definite kernels, and we consider the associated eigenvalues:
\begin{align}
    \xi_{d, k}(h_d) = \int_{[-\sqrt d , \sqrt d]} h_d\big(x/\sqrt{d}\big) Q_k^{(d)}(\sqrt d x) \tau^1_{d-1}(\de x),\,
\end{align}
where we recall that  $Q_k^{(d)}$ is the $k$-th Gegenbauer polynomial.

\begin{assumption}[Assumption for KRR at level $\ell \in \N$] \label{ass:activation_krr}
Let $\{ h_d \}_{d \ge 1}$ be a sequence of functions $h_d:\reals\to\reals$, such that $H_d (\bx_1,\bx_2) = h_d(\<\bx_1,\bx_2\>/d)$ is a positive semidefinite kernel. 
\begin{itemize}
\item[(a)] $h_d ( \cdot /\sqrt{d}) \in L^2([-\sqrt d, \sqrt d], \tau^1_{d-1})$, where $\tau^1_{d-1}$ is the distribution of $\< \bx, \be\>$ for $\bx \sim \Unif(\S^{d-1}(\sqrt d))$, where $\be = (1, 0, \ldots, 0)^\sT \in \R^d$. 
\item[(b)] There exists a constant $c_{\ell}>0$ such that
 \begin{align}
   \frac{d^{\ell} \min_{k \le \ell} \xi_{d,k} (h_d) }{\sum_{k \ge \ell + 1} \xi_{d,k}(h_d) B(d, k)}\ge c_{\ell}\, .
\end{align}
\end{itemize}
\end{assumption}
\begin{theorem}\label{thm:upper_bound_KRR}
  Assume $\omega_d(d^{\ell} \log d) \le  n \le O_d(d^{\ell +1 - \delta})$ for some integer $\ell$ and $\delta > 0$.
  Let  $\{ f_d \in L^2(\S^{d-1}(\sqrt d))\}_{d \ge 1}$ be a sequence of functions. Let $\{ h_d \}_{d \ge 1}$ be a sequence of kernels
  satisfying Assumption \ref{ass:activation_krr} at level $\ell$. Further define
  \begin{align}
    \lambda_*(d,\ell) := d^{\ell} \min_{k \le \ell} \xi_{d, k}(h_d)\, .
  \end{align}
  If $h_d$ has zero mean (i.e. $\int h_d(\sqrt{d}\<\be_1,\bx\>) \tau_d(\de\bx) = 0$) further assume that $f_d$ is centered
  (i.e. $\int f_d(\bx) \tau_d(\de\bx) = 0$).
  
  Let $\bX = (\bx_i)_{i \in [n]}$ with $(\bx_i)_{i \in [n]} \sim \Unif(\S^{d-1}(\sqrt d))$ independently, and $y_i = f_d(\bx_i) + \eps_i$ and $\eps_i \sim_{iid} \normal(0, \tau^2)$.
  Then for any $\eps > 0$, and any regularization parameter $\lambda\in (0,\lambda_*)$  with high probability we have
  \begin{align}
\vert R_{\KR}(f_d, \bX, \lambda) - \| \proj_{> \ell} f_d \|_{L^2}^2 \vert \le \eps (\| f_d \|_{L^2}^2 + \tau^2). 
\end{align}
\end{theorem}
See Section \ref{sec:proof_KR} for the proof of this theorem. 
  
\begin{remark}
  Assume  $h_d \to h$ as $d\to\infty$, uniformly over
  $[-\delta,\delta]$, together with its derivatives, and further assume $|h_d(x)|\le c_0\exp(c_1x^2/2)$ for some $c_0>0$, $c_1<1$.
  We expect this to be the case for many kernels of interest, and in particular it can be shown to be the case for $h_d^{\RF}$
  and $h_d^{\NT}$ under mild conditions on the activation $\sigma$. 
   Using Rodrigues' formula described in Section \ref{sec:Gegenbauer}, by an application of integration by part followed by dominated convergence, we get
  \begin{align}
    \xi_{d,k}(h_d) = \frac{1}{d^k}\, h^{(k)}(0) +o_d(d^{-k-1})\, ,
  \end{align}
  where $h^{(k)}$ is the $k$-th derivative of $h$. Notice further that $\xi_{d,k}(h_d)\ge 0$ for all $k$ since $h_d$ is positive
  semidefinite by definition.
  Therefore, as long as $h^{(k)}(0)>0$ for all $k\le \ell$, Assumption  \ref{ass:activation_krr} is satisfied, and $\lambda_*(d,\ell)$
  is bounded away from $0$.
\end{remark}

\begin{remark}
  For $h_d = h_d^{\RF}$ and if the activation $\sigma\in L^2(\reals,\gamma)$ is independent of $d$, we have $\xi_{d,k}(h_d) = \mu_k(\sigma)^2 d^{-k}+o_d(d^{-k-1})$,
  and therefore Assumption \ref{ass:activation_krr} is satisfied as soon as $\mu_k(\sigma)\neq 0$  for all $k\le \ell$. 
\end{remark}

Notice that the setting of Theorem \ref{thm:upper_bound_KRR} is the same as in classical nonparametric regression.
However, classical theory typically establishes minimax consistency rates of the form
$\E\{[\hf(\bx)-f_{\star}(\bx)]^2\}\le C(d)\, n^{-2\beta/(2\beta+d)}$ \cite{tsybakov2008introduction,gyorfi2006distribution}.
In order to guarantee a fixed (small)  error, these bounds require $n\ge \exp\{c\, d\}$.  Modern machine learning
typically have $d\ge 100$ and $n$ between $10^4$ and $10^8$, and it is therefore unrealistic to consider $n$ exponential in $d$. 
This regime motivates a new type of question:
\emph{assuming $n\asymp d^{\alpha}$, what is the minimum prediction error that can be achieved?}
This question is addressed by Theorem \ref{thm:upper_bound_KRR}.

\subsection{Separation between  kernel methods and neural networks}

Repeating the same argument of Section \ref{sec:Separation}, we
see that Theorems \ref{thm:KR-general-main} and
\ref{thm:upper_bound_KRR} imply a separation between kernel methods, with rotationally invariant kernels,
and gradient-descent trained neural networks.

Namely, consider again the target function $f_{\star}(\bx)=\sigma(\<\bw_{\star},\bx\>)$, for $\|\bw_{\star}\|_2=1$.
As proven in  \cite{mei2016landscape}, $f_\star$ can be learnt efficiently by  minimizing the following empirical risk
via gradient descent:
\begin{align*}
  \hR_{\NN}(\bw;\bw_\star) := \frac{1}{n}\sum_{i=1}^n\big(y_i-\sigma(\<\bw,\bx_i\>)\big)^2\, .
  \end{align*}
  Namely, if $n\ge C\, d\log d$ samples are used (and under some technical conditions on $\sigma$), gradient descent
  reaches prediction error of order $(d\log d)/n$

  In contrast, Theorems \ref{thm:KR-general-main} and \ref{thm:upper_bound_KRR} imply that, for any integer $\ell$,
  and any $d^{\ell+\delta}\le n\le d^{\ell+1-\delta}$, any kernel method has test error bounded away from zero. Namely
\begin{align}
  R_{H}(\sigma;\bX)  = \|\sigma_{>\ell}\|^2_{L^2(\reals,\gamma)} +o_{d,\P}(1)\, .
  \end{align}
  This test error is achieved by kernel ridge regression. 
  
  \subsection{Near-optimality of interpolators}
  
Let us emphasize some important statistical aspects of Theorem \ref{thm:upper_bound_KRR}. KRR is proved to achieve near optimal prediction error
(matching the lower bound of Theorem \ref{thm:KR-general-main}) \emph{pointwise}, i.e. per given function $f_d$. What is the nature of
the predictor $\hf_{\lambda}$?  Theorems \ref{thm:KR-general-main} and \ref{thm:upper_bound_KRR} imply
that, in $\ell_2$ sense, $\hf_{\lambda}$ must be close to a low-degree approximation of $f_d$, namely $\proj_{\le \ell}f_d$.

Optimal test error is achieved for any $\lambda<\lambda_*$. In particular, by taking $\lambda\to 0$, we obtain an
\emph{interpolator}, i.e. a predictor that interpolates the data $(y_i,\bx_i)$. This remark is made quantitative
in the following bound on the empirical risk 
\begin{align}
\hR_{\KR}(f_d,\bX,\lambda) :=\frac{1}{n}\sum_{i=1}^n(y_i-\hf_{\lambda}(\bx))^2\, .\label{eq:EmpRiskKRR}
\end{align}
\begin{theorem} \label{thm:empirical_risk_interpolator}
   Assume $\omega_d(d^{\ell} \log d) \le  n \le O_d(d^{\ell +1 - \delta})$ for some integer $\ell$ and $\delta > 0$.
  Under the same assumptions of Theorem \ref{thm:upper_bound_KRR}, if $\lambda<\lambda_*$, then
  \begin{align}
    \hR_{\KR}(f_d,\bX,\lambda) \le (1 + o_{d,\P}(1) ) ( \| f_d \|_{L^2}^2 + \tau^2) \left(\frac{\lambda}{\lambda + \kappa_h} \right)^2\, ,
  \end{align}
  where $\kappa_h = \sum_{k \ge \ell + 1} \xi_{d,k}(h_d) B(d, k)$.
\end{theorem}

\begin{proof}[Proof of Theorem \ref{thm:empirical_risk_interpolator}]
Recall that the empirical risk of KRR is given by Eq.~\eqref{eq:EmpRiskKRR},
where $ \hat \boldf_{\lambda} = ( \hat{f}_\lambda ( \bx_1 )  , \ldots , \hat{f}_\lambda ( \bx_n ))$ can be rewritten as 
\[
 \hat \boldf_{\lambda} = \bH (\bH + \lambda \id_n )^{-1} \by.
\]
Therefore,
\[
\begin{aligned}
\hR_{\KR}(f_d,\bX,\lambda) & = \| [\id_n - \bH (\bH + \lambda \id_n )^{-1}] \by \|_2^2/n \\
& = \lambda^2 \| (\bH + \lambda \id_n )^{-1}\by \|_2^2 /n.
\end{aligned}
\]
From the proof of Theorem \ref{thm:upper_bound_KRR}, we have the following lower bound on the eigenvalues $\bH + \lambda \id_n  \succeq (\kappa_h + \lambda + o_{d,\P} (1) ) \id_n $. We deduce that with high probability 
\[
\hR_{\KR}(f_d,\bX,\lambda) \leq (1+o_{d,\P}(1))(\lambda / (\kappa_h+\lambda) )^2 \| \by \|_2^2 / n \leq (1+o_{d,\P}(1)) ( \| f_d \|_{L^2}^2 + \tau^2 ) (\lambda /  (\kappa_h+\lambda) )^2,
\]
where we simply used the law of large numbers $\| \by \|_2^2 / n \to \| f_d \|_{L^2}^2 + \tau^2$.
\end{proof}

\subsection{A conjecture for generalization error of random features model}

Consider random features regression with finite sample size and a finite number of neurons.
We fit data $\{(y_i, \bx_i)\}_{i\le n}$ using ridge regression in the random features ($\RF$) model, with 
(where $\bw_i \sim_{iid} \Unif(\S^{d-1}(1))$)
\begin{align}
\hba(\lambda) = \argmin_{\ba\in\reals^N} \left\{\frac{1}{n}\sum_{j=1}^n \Big( y_j- \sum_{i=1}^N a_i \sigma(\< \bw_i, \bx_j\>) \Big)^2  +  
\frac{N\lambda}{d}\, \| \ba \|_2^2\right\}\, . \label{eq:Ridge}
\end{align}
Under the same data model of the previous sections, we are interested in the test prediction error
\begin{align}\label{eqn:prediction_risk_first_definition}
R_\RF(f_d, \bX, \bW, \lambda) = \E_{\bx}\Big[ \Big(f_d(\bx) - \sum_{i=1}^N \hat a_i(\lambda) \sigma(\< \bw_i, \bx\> ) \Big)^2\Big]\, .
\end{align} 
Theorem \ref{thm:RF_lower_upper_bound} characterized the test error $R_\RF(f_d, \bX, \bW, \lambda)$ in the population limit
$n = \infty$, whereas Theorems \ref{thm:KR-general-main} and \ref{thm:upper_bound_KRR} characterize the same quantity
in the case when $N = \infty$.

What happens when both $n$ and $N$ are finite?
In the proportional regime  $N \propto d$ and $n \propto d$, the precise asymptotics of $R_\RF(f_d, \bX, \bW, \lambda)$
was calculated in \cite{mei2019generalization}. 

What happens beyond the proportional asymptotics? We conjecture that the limiting factor is given by the smallest
of $n$ and $N$. Namely, if $d^{\ell+\delta}\le \min(n,N) \le d^{\ell+1-\delta}$ for some positive $\delta$,
then the prediction error is the same as the one of fitting a degree-$\ell$ polynomial, i.e. 
$R_\RF(f_d, \bX, \bW, \lambda) = \|\proj_{>\ell}f_d\|_{L^2}^2 +\|f_d\|_{L^2}^2 \cdot o_{d, \P}(1)$.
We leave this conjecture to future work.

\section{Further related work}
\label{sec:Related}

Donoho and Johnstone \cite{donoho1989projection} study an
approximation problem analogous to the one we considered in Section \ref{sec:Approximation}, although 
in $d=2$ dimensions. Their problem essentially reduces to determining rates of approximation on the unit circle,
with the technical difference that the $\bw_i$'s are equi-spaced along the circle instead of being random.
As for other references mentioned in Section \ref{sec:Parenthesis}, the lower bounds of \cite{donoho1989projection}
are worst case over differentiable functions.

The limitations of kernel methods in high-dimension are studied by El Karoui in \cite{el2010spectrum}
(see also \cite{el2010information}), which analyzes kernel random matrices of the form $\bH= (h(\<\bx_i,\bx_j\>/d))_{i,j\le n}$.
The analysis of \cite{el2010spectrum} is limited to the proportional asymptotics $n \propto d$.
and establishes that in this regime $\bH$ is well approximated
by the Gram matrix of raw feature vectors plus a diagonal term:
$\bH \approx (h(1)-h'(0))\id_n+h'(0)\bG$, where $\bG= (\<\bx_i,\bx_j\>/d)_{i,j\le n}$.
This result is related to our Theorems \ref{thm:KR-general-main} and \ref{thm:upper_bound_KRR},
which deal with kernel methods. However our results
analyze  general polynomial scalings $n = O_d(d^{\ell+1-\delta})$, while \cite{el2010spectrum} assumes $n =\Theta_d(d)$.
Also \cite{el2010spectrum} analyzes the spectrum of $\bH$ but not the prediction error of kernel methods.
Finally,  a large part of our technical work is devoted to \RF\, and \NT\, models, cf. Theorems
\ref{thm:RF_lower_upper_bound} and \ref{thm:NT_lower_upper_bound}, which are not touched upon by  \cite{el2010spectrum}.

Recent work of Vempala and Wilmes \cite{vempala2018polynomial} analyzes what amounts to an \RF\,  model.
These authors prove that \RF\, can learn a degree-$\ell$ polynomial from $n= d^{O(\ell)}$ samples using $N =d^{O(\ell)}$
neurons, and that at least $d^{\Omega(\ell)}$ queries are needed within the statistical query model.
While related,  our setting is not directly comparable to theirs. Notice further that we obtain a sharper tradeoff, since we
obtain the precise exponents of  $d$.

After the present paper appeared as a preprint, several authors presented important contributions to the same line of work.
In particular, Liang, Rakhlin, and Zhai \cite{liang2019risk} studies kernel ridge regression in $d$ dimension using
$n = O_d(d^\gamma)$ samples. Assuming the target function has bounded RKHS norm, they derive upper and lower bounds
on the rate of convergence of the generalization error. 
This result is related to our Theorem \ref{thm:KR-general-main}.  
The most important difference is that we do not assume that the target function has bounded RKHS norm.
Instead we obtain the precise asymptotics of the  generalization error in a regime in which it is non-vanishing. As illustrated in Section \ref{sec:Numerical}, this asymptotic analysis captures indeed the actual behavior
in practically reasonable settings.

From a technical viewpoint, several of our calculations make use of harmonic analysis over the $d$-dimensional sphere, as it is natural given that $\bx_i$'s are uniform over the sphere. Spherical harmonics expansion appear in related contexts, e.g. in \cite{donoho1989projection,bach2017breaking,vempala2018polynomial}. 

Let us finally mention that an alternative approach to the analysis of two-layers neural networks in the wide limit, 
was developed in \cite{mei2018mean,rotskoff2018neural,sirignano2018mean,chizat2018global,mei2019mean} using mean field theory.
Unlike in the neural tangent approach, the evolution of network weights is described beyond the linear regime in this theory.

\section{Technical background}
\label{sec:Background}

In this section we introduce some notation and technical background which will be useful for the proofs in the next sections.
In particular, we will use decompositions in (hyper-)spherical harmonics on the  $\S^{d-1}(\sqrt{d})$ and in orthogonal polynomials
on the real line. All of the properties listed below are classical: we will however prove a few facts that are slightly less standard. 
We refer the reader to \cite{costas2014spherical,szego1939orthogonal,chihara2011introduction}
for further information on these topics.
As mentioned above, expansions in spherical harmonics were used in the past in the statistics literature, for instance in 
 \cite{donoho1989projection,bach2017breaking}.

\subsection{Functional spaces over the sphere}

For $d \ge 1$, we let $\S^{d-1}(r) = \{\bx \in \R^{d}: \| \bx \|_2 = r\}$ denote the sphere with radius $r$ in $\reals^d$.
We will mostly work with the sphere of radius $\sqrt d$, $\S^{d-1}(\sqrt{d})$ and will denote by $\tau_{d-1}$  the uniform probability measure on $\S^{d-1}(\sqrt d)$. 
All functions in the following are assumed to be elements of $ L^2(\S^{d-1}(\sqrt d) ,\tau_{d-1})$, with scalar product and norm denoted as $\<\,\cdot\,,\,\cdot\,\>_{L^2}$
and $\|\,\cdot\,\|_{L^2}$:
\begin{align}
\<f,g\>_{L^2} \equiv \int_{\S^{d-1}(\sqrt d)} f(\bx) \, g(\bx)\, \tau_{d-1}(\de \bx)\,.
\end{align}

For $\ell\in\integers_{\ge 0}$, let $\tilde{V}_{d,\ell}$ be the space of homogeneous harmonic polynomials of degree $\ell$ on $\reals^d$ (i.e. homogeneous
polynomials $q(\bx)$ satisfying $\Delta q(\bx) = 0$), and denote by $V_{d,\ell}$ the linear space of functions obtained by restricting the polynomials in $\tilde{V}_{d,\ell}$
to $\S^{d-1}(\sqrt d)$. With these definitions, we have the following orthogonal decomposition
\begin{align}
L^2(\S^{d-1}(\sqrt d) ,\tau_{d-1}) = \bigoplus_{\ell=0}^{\infty} V_{d,\ell}\, . \label{eq:SpinDecomposition}
\end{align}
The dimension of each subspace is given by
\begin{align}
\dim(V_{d,\ell}) = B(d, \ell) = \frac{2 \ell + d - 2}{\ell} { \ell + d - 3 \choose \ell - 1} \, .
\end{align}
For each $\ell\in \integers_{\ge 0}$, the spherical harmonics $\{ Y_{\ell, j}^{(d)}\}_{1\le j \in \le B(d, \ell)}$ form an orthonormal basis of $V_{d,\ell}$:
\[
\<Y^{(d)}_{ki}, Y^{(d)}_{sj}\>_{L^2} = \delta_{ij} \delta_{ks}.
\]
Note that our convention is different from the more standard one, that defines the spherical harmonics as functions on $\S^{d-1}(1)$.
It is immediate to pass from one convention to the other by a simple scaling. We will drop the superscript $d$ and write $Y_{\ell, j} = Y_{\ell, j}^{(d)}$ whenever clear from the context.

We denote by $\proj_k$  the orthogonal projections to $V_{d,k}$ in $L^2(\S^{d-1}(\sqrt d),\tau_{d-1})$. This can be written in terms of spherical harmonics as
\begin{align}
\proj_k f(\bx) \equiv& \sum_{l=1}^{B(d, k)} \< f, Y_{kl}\>_{L^2} Y_{kl}(\bx). 
\end{align}
We also define
$\proj_{\le \ell}\equiv \sum_{k =0}^\ell \proj_k$, $\proj_{>\ell} \equiv \id -\proj_{\le \ell} = \sum_{k =\ell+1}^\infty \proj_k$,
and $\proj_{<\ell}\equiv \proj_{\le \ell-1}$, $\proj_{\ge \ell}\equiv \proj_{>\ell-1}$.

\subsection{Gegenbauer polynomials}
\label{sec:Gegenbauer}

The $\ell$-th Gegenbauer polynomial $Q_\ell^{(d)}$ is a polynomial of degree $\ell$. Consistently
with our convention for spherical harmonics, we view $Q_\ell^{(d)}$ as a function $Q_{\ell}^{(d)}: [-d,d]\to \reals$. The set $\{ Q_\ell^{(d)}\}_{\ell\ge 0}$
forms an orthogonal basis on $L^2([-d,d],\tilde\tau^1_{d-1})$, where $\tilde\tau^1_{d-1}$ is the distribution of $\sqrt{d}\<\bx,\be_1\>$ when $\bx\sim \tau_{d-1}$,
satisfying the normalization condition:
\begin{align}
\< Q^{(d)}_k(\sqrt{d}\< \be_1, \cdot\>), Q^{(d)}_j(\sqrt{d}\< \be_1, \cdot\>) \>_{L^2(\S^{d-1}(\sqrt d))} = \frac{1}{B(d,k)}\, \delta_{jk} \, .  \label{eq:GegenbauerNormalization}
\end{align}
In particular, these polynomials are normalized so that  $Q_\ell^{(d)}(d) = 1$. 
As above, we will omit the superscript $d$ when clear from the context.

Gegenbauer polynomials are directly related to spherical harmonics as follows. Fix $\bv\in\S^{d-1}(\sqrt{d})$ and 
consider the subspace of  $V_{\ell}$ formed by all functions that are invariant under rotations in $\reals^d$ that keep $\bv$ unchanged.
It is not hard to see that this subspace has dimension one, and coincides with the span of the function $Q_{\ell}^{(d)}(\<\bv,\,\cdot\,\>)$.

We will use the following properties of Gegenbauer polynomials
\begin{enumerate}
\item For $\bx, \by \in \S^{d-1}(\sqrt d)$
\begin{align}
\< Q_j^{(d)}(\< \bx, \cdot\>), Q_k^{(d)}(\< \by, \cdot\>) \>_{L^2} = \frac{1}{B(d,k)}\delta_{jk}  Q_k^{(d)}(\< \bx, \by\>).  \label{eq:ProductGegenbauer}
\end{align}
\item For $\bx, \by \in \S^{d-1}(\sqrt d)$
\begin{align}
Q_k^{(d)}(\< \bx, \by\> ) = \frac{1}{B(d, k)} \sum_{i =1}^{ B(d, k)} Y_{ki}^{(d)}(\bx) Y_{ki}^{(d)}(\by). \label{eq:GegenbauerHarmonics}
\end{align}
\item Recurrence formula 
\begin{align}
\frac{t}{d}\,  Q_k^{(d)}(t) = \frac{k}{2k + d - 2} Q_{k-1}^{(d)}(t) + \frac{k + d - 2}{2k + d - 2} Q_{k+1}^{(d)}(t). \label{eq:RecursionG}
\end{align}
\item Rodrigues' formula
\begin{align}
Q_k^{(d)}(t) = (-1/2)^k d^k \frac{\Gamma((d - 1)/2)}{\Gamma(k + (d - 1)/2)} \Big( 1 -  \frac{t^2}{d^2} \Big)^{(3-d)/2} \Big( \frac{\de }{\de t}\Big)^k \Big(1 - \frac{t^2}{d^2} \Big)^{k + (d-3)/2}. 
\label{eq:Rogrigues_formula}
\end{align}
\end{enumerate}
Note in particular that property 2 implies that --up to a constant-- $Q_k^{(d)}(\< \bx, \by\> )$ is a representation of the projector onto 
the subspace of degree -$k$ spherical harmonics
\begin{align}
(\proj_k f)(\bx) = B(d,k) \int_{\S^{d-1}(\sqrt{d})} \, Q_k^{(d)}(\< \bx, \by\> )\,  f(\by)\, \tau_{d-1}(\de\by)\, .\label{eq:ProjectorGegenbauer}
\end{align}
For a function $\sigma \in L^2([-\sqrt d, \sqrt d], \tau^1_{d-1})$ (where $\tau^1_{d-1}$ is the distribution of $\< \bx_1, \bx_2 \> / \sqrt d$ when $\bx_1, \bx_2 \sim_{iid} \Unif(\S^{d-1}(\sqrt d))$), denoting its spherical harmonics coefficients $\lambda_{d, k}(\sigma)$ to be 
\begin{align}\label{eqn:technical_lambda_sigma}
\lambda_{d, k}(\sigma) = \int_{[-\sqrt d , \sqrt d]} \sigma(x) Q_k^{(d)}(\sqrt d x) \tau^1_{d-1}(\de x),
\end{align}
then we have the following equation holds in $L^2([-\sqrt d, \sqrt d],\tau^1_{d-1})$ sense
\[
\sigma(x) = \sum_{k = 0}^\infty \lambda_{d, k}(\sigma) B(d, k) Q_k^{(d)}(\sqrt d x). 
\]

To  any rotationally invariant kernel $H_d(\bx_1, \bx_2) = h_d(\< \bx_1, \bx_2\> / d)$,
with $h_d(\sqrt{d}\, \cdot \, ) \in L^2([-\sqrt{d},\sqrt{d}],\tau^1_{d-1})$,
we can associate a self adjoint operator $\cuH_d:L^2(\S^{d-1}(\sqrt{d}))\to L^2(\S^{d-1}(\sqrt{d}))$
via
\begin{align}
\cuH_df(\bx) := \int_{\S^{d-1}(\sqrt{d})} h_d(\<\bx,\bx_1\>/d)\, f(\bx_1) \, \tau_{d-1}(\de \bx_1)\, .
\end{align}
By rotational invariance,   the space $V_{k}$ of homogeneous polynomials of degree $k$ is an eigenspace of
$\cuH_d$, and we will denote the corresponding eigenvalue by $\xi_{d,k}(h_d)$. In other words
$\cuH_df(\bx) := \sum_{k=0}^{\infty} \lambda_{d,k}(h_d) \proj_{k}f$.   The eigenvalues can be computed via
\begin{align}
  \xi_{d, k}(h_d) = \int_{[-\sqrt d , \sqrt d]} h_d\big(x/\sqrt{d}\big) Q_k^{(d)}(\sqrt d x) \tau^1_{d-1}(\de x)\, .
\end{align}

\subsection{Hermite polynomials}
\label{sec:Hermite}

The Hermite polynomials $\{\bbHe_k\}_{k\ge 0}$ form an orthogonal basis of $L^2(\reals,\gamma)$, where $\gamma(\de x) = e^{-x^2/2}\de x/\sqrt{2\pi}$ 
is the standard Gaussian measure, and $\bbHe_k$ has degree $k$. We will follow the classical normalization (here and below, expectation is with respect to
$G\sim\normal(0,1)$):
\begin{align}
\E\big\{\bbHe_j(G) \,\bbHe_k(G)\big\} = k!\, \delta_{jk}\, .
\end{align}
As a consequence, for any function $g\in L^2(\reals,\gamma)$, we have the decomposition
\begin{align}\label{eqn:sigma_He_decomposition}
g(x) = \sum_{k=0}^{\infty}\frac{\mu_k(g)}{k!}\, \bbHe_k(x)\, ,\;\;\;\;\;\; \mu_k(g) \equiv \E\big\{g(G)\, \bbHe_k(G)\}\, .
\end{align}

The Hermite polynomials can be obtained as high-dimensional limits of the Gegenbauer polynomials introduced in the previous section. Indeed, the Gegenbauer polynomials (up to a $\sqrt d$ scaling in domain) are constructed by Gram-Schmidt orthogonalization of the monomials $\{x^k\}_{k\ge 0}$ with respect to the measure 
$\tilde\tau^1_{d-1}$, while Hermite polynomial are obtained by Gram-Schmidt orthogonalization with respect to $\gamma$. Since $\tilde\tau^1_{d-1}\Rightarrow \gamma$
(here $\Rightarrow$ denotes weak convergence),
it is immediate to show that, for any fixed integer $k$, 
\begin{align}
\lim_{d \to \infty} \Coeff\{ Q_k^{(d)}( \sqrt d x) \, B(d, k)^{1/2} \} = \Coeff\left\{ \frac{1}{(k!)^{1/2}}\,\bbHe_k(x) \right\}\, .\label{eq:Gegen-to-Hermite}
\end{align}
Here and below, for $P$ a polynomial, $\Coeff\{ P(x) \}$ is  the vector of the coefficients of $P$. As a consequence,
for any fixed integer $k$, we have
\begin{align}\label{eqn:mu_lambda_relationship}
\mu_k(\sigma) = \lim_{d \to \infty} \lambda_{d,k}(\sigma) (B(d, k)k!)^{1/2}, 
\end{align}
where $\mu_k(\sigma)$ and $\lambda_{d,k}(\sigma)$ are given in Eq. (\ref{eqn:sigma_He_decomposition}) and (\ref{eqn:technical_lambda_sigma}).

\subsection{Notations}
Throughout the proofs, $O_d(\, \cdot \, )$  (resp. $o_d (\, \cdot \,)$) denotes the standard big-O (resp. little-o) notation, where the subscript $d$ emphasizes the asymptotic variable. We denote $O_{d,\P} (\, \cdot \,)$ (resp. $o_{d,\P} (\, \cdot \,)$) the big-O (resp. little-o) in probability notation: $h_1 (d) = O_{d,\P} ( h_2(d) )$ if for any $\eps > 0$, there exists $C_\eps > 0 $ and $d_\eps \in \Z_{>0}$, such that
\[
\begin{aligned}
\P ( |h_1 (d) / h_2 (d) | > C_{\eps}  ) \le \eps, \qquad \forall d \ge d_{\eps},
\end{aligned}
\]
and respectively: $h_1 (d) = o_{d,\P} ( h_2(d) )$, if $h_1 (d) / h_2 (d)$ converges to $0$ in probability.

We will occasionally hide logarithmic factors  using the  $\Tilde O_d (\, \cdot\, )$ notation (resp. $\Tilde o_d (\, \cdot \, )$): $h_1(d) = \tilde O_d(h_2(d))$ if there exists a constant $C$ 
such that $h_1(d) \le C(\log d)^C h_2(d)$. Similarly, we will denote $\Tilde O_{d,\P} (\, \cdot\, )$ (resp. $\Tilde o_{d,\P} (\, \cdot \, )$) when considering the big-O in probability notation up to a logarithmic factor.

\section{Proof of Theorem \ref{thm:RF_lower_upper_bound}.(a): \RF\, model lower bound}
\label{sec:proof_RFK_lower}

\subsection{Proof of Theorem \ref{thm:RF_lower_upper_bound}.(a): Outline}
\label{sec:OutlineRF}

Recall that $(\bw_i)_{i \in [N]} \sim \Unif(\S^{d-1})$ independently. We define $\btheta_i = \sqrt d \cdot \bw_i$ for $i \in [N]$, so that $(\btheta_i)_{i \in [N]} \sim \Unif(\S^{d-1}(\sqrt d))$ independently. Let $\bW = (\bw_1, \ldots, \bw_N)$, and $\bTheta = (\btheta_1, \ldots, \btheta_N)$. We denote $\E_\btheta$ to be the expectation operator with respect to $\btheta \sim \Unif(\S^{d-1}(\sqrt d))$, $\E_\bx$ to be the expectation operator with respect to $\bx \sim \Unif(\S^{d-1}(\sqrt d))$, and $\E_\bw$ to be the expectation operator with respect to $\bw \sim \Unif(\S^{d-1}(1))$.  

Define the random vectors $\bV = (V_1, \ldots, V_N)^\sT$, $\bV_{\le \ell} = (V_{1, \le \ell}, \ldots, V_{N, \le \ell})^\sT$, $\bV_{> \ell} = (V_{1, >\ell}, \ldots, V_{N, >\ell})^\sT$, with
\begin{align}
V_{i, \le \ell} \equiv& \E_{\bx}[[\proj_{\le \ell} f_d](\bx) \sigma_d(\< \btheta_i, \bx\>/\sqrt d )],\\
V_{i, > \ell} \equiv& \E_{\bx}[[\proj_{> \ell} f_d](\bx) \sigma_d(\< \btheta_i, \bx\> /\sqrt d)],\\
V_i \equiv& \E_{\bx}[f_d(\bx) \sigma_d (\< \btheta_i, \bx\>/\sqrt d)] = V_{i, \le \ell} + V_{i, > \ell}. 
\end{align}
Define the random matrix $\bU = (U_{ij})_{i, j \in [N]}$, with 
\begin{align}
U_{ij} = \E_{\bx}[\sigma_d (\< \bx, \btheta_i\>/\sqrt d) \sigma_d (\< \bx, \btheta_j\> /\sqrt d)]. \label{eq:KernelMatrix}
\end{align}
In what follows, we write $R_{\RF}(f_d) = R_{\RF}(f_d,\bW) =  R_{\RF}(f_d,\bTheta/\sqrt{d})$ for the random features risk, omitting the dependence on the weights $\bW = \bTheta/\sqrt{d}$.
By the definition and a simple calculation, we have 
\[
\begin{aligned}
R_{\RF}(f_d) =& \min_{\ba \in \R^N} \Big\{ \E_{\bx}[f_d(\bx)^2] - 2 \< \ba, \bV \> + \< \ba, \bU \ba\> \Big\} = \E_{\bx}[f_d(\bx)^2] - \bV^\sT \bU^{-1} \bV,\\
R_{\RF}(\proj_{\le \ell} f_d) =& \min_{\ba \in \R^N} \Big\{ \E_{\bx}[\proj_{\le \ell} f_d(\bx)^2] - 2 \< \ba, \bV_{\le \ell} \> + \< \ba, \bU \ba\> \Big\} = \E_{\bx}[\proj_{\le \ell} f_d(\bx)^2] - \bV_{\le \ell}^\sT \bU^{-1} \bV_{\le \ell}. 
\end{aligned}
\]

By orthogonality, we have
\[
\E_{\bx}[f_d(\bx)^2] = \E_{\bx}[[\proj_{\le \ell}f_d](\bx)^2] + \E_{\bx}[[\proj_{> \ell} f_d](\bx)^2], 
\]
which gives
\begin{equation}\label{eqn:decomposition_risk}
\begin{aligned}
& \Big\vert R_{\RF}(f_d) - R_{\RF}(\proj_{\le \ell} f_d) - \E_{\bx}[[\proj_{> \ell} f_d](\bx)^2] \Big\vert \\
=& \Big\vert \bV_{\le \ell}^\sT \bU^{-1} \bV_{\le \ell} - \bV^\sT \bU^{-1} \bV \Big\vert = \Big\vert \bV_{\le \ell}^\sT \bU^{-1} \bV_{\le \ell} - (\bV_{\le \ell} + \bV_{> \ell})^\sT \bU^{-1} (\bV_{\le \ell} + \bV_{> \ell}) \Big\vert\\
=& \Big\vert 2 \bV^\sT \bU^{-1} \bV_{>\ell} - \bV_{>\ell}^\sT \bU^{-1} \bV_{>\ell} \Big\vert \le 2 \| \bU^{-1/2} \bV_{> \ell} \|_2 \| \bU^{-1/2} \bV \|_{2} +  \| \bU^{-1} \|_{\op} \| \bV_{> \ell}\|_2^2\\
\le&  2 \| \bU^{-1/2} \|_{\op} \| \bV_{> \ell} \|_2 \| f_d \|_{L^2}+  \| \bU^{-1} \|_{\op} \| \bV_{> \ell}\|_2^2,
\end{aligned}
\end{equation}
where the last inequality used the fact that
\[
0 \le R_{\RF}(f_d) = \| f_d \|_{L^2}^2 - \bV^\sT \bU^{-1} \bV, 
\]
so that
\[
\| \bU^{-1/2} \bV \|_2^2 = \bV^\sT \bU^{-1} \bV \le \| f_d \|_{L^2}^2. 
\]

We claim that we have 
\begin{align}
 \| \bU^{-1/2} \|_{\op} \| \bV_{> \ell} \|_2  =& o_{d,\P} (\| \proj_{> \ell} f_d \|_{L^2} ), \label{eqn:bound_U_V_RFK}
\end{align}
This is achieved by the Proposition \ref{prop:expected_V_RFK} and \ref{prop:kernel_lower_bound_RFK} stated below. 

We will denote below by $\lambda_k(\sigma_d)$, $k\ge 0$, the coefficients of $\sigma_d$ in the basis of Gegenbauer polynomials.
Explicitly, since $\sigma_d(\< \be, \cdot\>) \in L^2(\S^{d-1}(\sqrt d))$, we can expand $\sigma_d$ as
\begin{align}
\sigma_d (x_1) = \sum_{k=0}^\infty B(d, k) \lambda_k(\sigma_d ) Q_k(\sqrt d\, x_1), \label{eq:SigmaGegen1}
\end{align}
where
\begin{align}
\lambda_k(\sigma_d ) = \< \sigma_d (\< \be, \, \cdot\,\>), Q_k(\sqrt d \< \be, \,\cdot\,\>) \>_{L^2}. \label{eq:SigmaGegen2}
\end{align}

\begin{proposition}[Expected norm of $\bV$]\label{prop:expected_V_RFK}
Let $\{ \sigma_d \}_{d\ge 1}$ be a sequence of activation functions with $\sigma_d ( \< \be , \cdot \>) \in L^2 ( \S^{d-1} (\sqrt{d}))$. Define $\cE_{\ge \ell}$ by
\[
\cE_{\ge \ell} \equiv  \E_{\btheta}[\< \proj_{\ge \ell} f_\star, \sigma_d (\<\btheta, \cdot\>/\sqrt d) \>_{L^2}^2].
\]
Then
\[
\cE_{\ge \ell} \le \Big[ \max_{k \geq \ell} \lambda_k (\sigma_d)^2 \Big] \cdot \| \proj_{\ge \ell} f_\star \|_{L^2}^2 \, .
\]
\end{proposition}

\begin{proposition}[Lower bound on the kernel matrix]\label{prop:kernel_lower_bound_RFK}
Assume $N \le d^{\ell+1}/e^{A_d\sqrt{\log d}}$ for a fixed integer $\ell$ and any $A_d\to\infty$
(in particular, $N\le d^{\ell+1-\delta}$ is sufficient for any fixed $\delta>0$). Let  $(\btheta_i)_{i \in [N]} \sim \Unif(\S^{d-1}(\sqrt d))$ independently, and $\{ \sigma_d \}_{d\ge 1}$ be a sequence of activation functions with $\sigma_d ( \< \be , \cdot \>) \in L^2 ( \S^{d-1} (\sqrt{d}))$. Let $\bU \in \R^{N \times N}$ be the  kernel matrix defined by Eq.~\eqref{eq:KernelMatrix}.
Then for any $\eps \in ( 0, 1)$, 
\[
\begin{aligned}
\lambda_{\min}(\bU) \ge& (1- \eps) \Big[ \sum_{k = \ell +1}^\infty \lambda_k (\sigma_d)^2 \cdot B(d,k) \Big], \\
\end{aligned}
\]
with high probability as $d \to \infty$. 
\end{proposition}

The proof of Proposition \ref{prop:kernel_lower_bound_RFK} relies on the following tight bound on the operator norm of the Gegenbauer polynomials of the Gram matrix:

\begin{proposition}[Bound on the Gram matrix]\label{prop:Delta_bound}
Let $N \le d^{k}/e^{A_d\sqrt{\log d}}$ for a fixed integer $k$ and any $A_d\to\infty$. Let $(\btheta_i)_{i \in [N]} \sim \Unif(\S^{d-1}(\sqrt d))$
independently, and $Q_k^{(d)}$ be the $k$'th Gegenbauer polynomial
with domain $[-d, d]$. Consider the random matrix $\bW =
(\bW_{ij})_{i, j \in [N]} \in \R^{N \times N}$, with $\bW_{ij} =
Q_k^{(d)}(\< \btheta_i, \btheta_j\>)$. Then we have 
\[
\lim_{d, N \to \infty} \E[\| \bW - \id_d \|_{\op}] = 0. 
\]
\end{proposition}

The proofs of these three propositions are provided in the next sections. Proposition \ref{prop:expected_V_RFK}  implies
\[
\begin{aligned}
\E[\| \bV_{> \ell} \|_2^2] = N \cE_{\ge \ell + 1} \le& N \cdot \Big[ \max_{k \geq \ell+1 } \lambda_k (\sigma_d)^2 \Big]  \| \proj_{\ge \ell + 1} f_d \|_{L^2}^2.\\
\end{aligned}
\]
From Proposition \ref{prop:kernel_lower_bound_RFK}, we have with high probability
\[
\| \bU^{-1} \|_{\op} \le 2 \Big[ \sum_{k = \ell + 1}^{\infty} \lambda_k( \sigma_d)^2 \cdot B(d, k) \Big]^{-1}.
\]
Then by Markov inequality, we have with high probability 
\[
\begin{aligned}
& \| \bU^{-1} \|_{\op} \| \bV_{> \ell} \|_2^2 / \| \proj_{\ge \ell + 1} f_d \|_{L^2}^2\\
 \le& 2  N \Big[\max_{k \ge \ell + 1} \lambda_k( \sigma_d)^2\Big] \cdot \Big[ \sum_{k = \ell + 1}^{\infty} \lambda_k( \sigma_d)^2 \cdot B(d, k) \Big]^{-1}\\
 \leq & 2 N \max_{k \ge \ell + 1} B(d, k)^{-1} .
\end{aligned}
\]
Equation (\ref{eqn:bound_U_V_RFK}) follows by noting that $B(d,k)$ is non-decreasing in $k$ (see Lemma \ref{lem:non_decreasing_N} below) and $B(d, \ell + 1) = \Theta_d(d^{\ell + 1})$, and recalling $N = o_d(d^{\ell+1})$. Combining with Eq. (\ref{eqn:decomposition_risk}), the theorem holds. 

\begin{lemma}\label{lem:non_decreasing_N}
The number $B(d, k)$ of independent degree-$k$ spherical harmonics on $\S^{d-1}$ is non-decreasing in $k$ for any fixed $d \ge 2$. 
\end{lemma}

\begin{proof}[Proof of Lemma \ref{lem:non_decreasing_N}]
By \cite[Section 4.1]{costas2014spherical}, we have 
\[
B(d, k) = K(d-1, k) + K(d-1, k-1). 
\]
and 
\[
K(d, k) = \sum_{j = 0}^k K(d-1, j),
\]
where $K(d-1, j) = {d-2+j \choose j}$ is non-negative for $d \ge 2$. This immediately shows that $B(d, k)$ is non-decreasing in $k$. 
\end{proof}

\subsection{Proof of Proposition \ref{prop:expected_V_RFK}}

The quantity $\cE_{\ge \ell}$ can be rewritten as
\[
\begin{aligned}
\cE_{\ge \ell} \equiv&  \E_{\btheta}[\< \proj_{\ge \ell} f_\star, \sigma_d (\<\btheta, \cdot\>/\sqrt d) \>_{L^2}^2] \\
=&  \sum_{s, t \ge \ell} \E_{\btheta}[\E_\bx[\proj_{s} f_\star(\bx) \sigma_d (\<\btheta, \bx\> / \sqrt d) \>] \E_\bx[\proj_{t} f_\star(\bx) \sigma_d (\<\btheta, \bx\> / \sqrt d) \>]]. 
\end{aligned}
\]
First we calculate $\E_\bx[\proj_{k} f_\star(\bx) \sigma_d (\<\btheta, \bx\> / \sqrt d) \>]$. Note the spherical harmonics expansion of $\proj_k f_\star$ gives
\[
\proj_{k} f_\star(\bx) = \sum_{l = 1}^{B(d, k)} \lambda_{kl}(f_\star) Y_{kl}(\bx), 
\]
and the Gegenbauer expansion of $\sigma_d$ gives
\[
\sigma_d (\<\btheta, \bx\>/\sqrt{d}) = \sum_{u = 0}^\infty \lambda_u(\sigma_d) B(d, u) Q_u(\< \btheta, \bx\>).
\]
By the fact that 
\[
\E_\bx[Q_u(\< \btheta, \bx\>) Y_{kl}(\bx)] = \frac{1}{B(d, u)} \sum_{s = 1}^{B(d, u)} Y_{us}( \btheta) \E_\bx\Big[  Y_{us}(\bx) Y_{kl}(\bx) \Big] = \frac{1}{B(d, u)} Y_{ul}(\btheta) \delta_{uk} ,
\]
we have 
\[
\begin{aligned}
\E_\bx[\proj_{k} f_\star(\bx) \sigma_d (\<\btheta, \bx\> / \sqrt d) \>] & = \sum_{l = 1}^{B(d, k)} \lambda_{kl}(f_\star) \lambda_k (\sigma_d)B(d,k) \E_\bx[Q_k(\< \btheta, \bx\>) Y_{kl}(\bx)]\\
&  = \sum_{l = 1}^{B(d, k)} \lambda_{kl}(f_\star) \lambda_k (\sigma_d) Y_{kl}(\btheta). 
\end{aligned}
\]
We deduce that
\[
\begin{aligned}
\cE_{\ge \ell} = & \sum_{s, t \ge \ell} \lambda_s (\sigma_d) \lambda_{t} (\sigma_d) \sum_{l = 1}^{B(d, s)}  \sum_{u = 1}^{B(d, t)} \lambda_{sl}(f_\star) \lambda_{tu}(f_\star) \delta_{st} \delta_{lu} \\
 \leq &\Big[ \max_{k \geq \ell} \lambda_k (\sigma_d)^2\Big] \cdot \sum_{k \ge \ell}\sum_{l = 1}^{B(d, k)} \lambda_{kl}(f_\star)^2 \\
 = & \Big[ \max_{k \geq \ell} \lambda_k (\sigma_d)^2\Big] \cdot \| \proj_{\ge \ell} f_\star \|_{L^2}^2. 
\end{aligned}
\]
This proves the proposition.

\subsection{Proof of Proposition \ref{prop:kernel_lower_bound_RFK}}

Recall the expansion of $\sigma_d$ in terms of Gegenbauer
polynomials, see Eqs.~\eqref{eq:SigmaGegen1} and \eqref{eq:SigmaGegen2}.
From the properties of Gegenbauer polynomials, we have
\[
\begin{aligned}
\E_\bx [ \sigma_d (\<  \btheta, \bx \> / \sqrt{d}) \sigma_d (\< \btheta ' , \bx \> /\sqrt{d} ) ] & = \sum_{k= 0}^\infty \lambda_{k} (\sigma_d)^2 B(d,k)^2 \E_\bx  [ Q_k (\<  \btheta, \bx \> ) Q_k (\<  \btheta ', \bx \> )] \\
& = \sum_{k= 0}^\infty \lambda_{k} (\sigma_d)^2 B(d,k) Q_k (\<  \btheta, \btheta' \> ).
\end{aligned}
\]
We can therefore decompose $\bU$:
\[
\bU = \sum_{k= 0}^\infty \lambda_{k} (\sigma_d)^2 B(d,k) \cdot \bW_k,
\]
where $\bW_k = (W_{k, ij})_{i, j \in [N]}$ with $W_{k, ij} = Q_{k}(\< \btheta_i, \btheta_j\>)$. 

Define
\[
\begin{aligned}
\hat \bU \equiv& \sum_{k= 0}^\ell \lambda_{k} (\sigma_d)^2 B(d,k) \cdot \bW_k,\\
\bar \bU \equiv& \sum_{k=\ell + 1}^{\infty} \lambda_k( \sigma_d)^2 \cdot B(d, k) \cdot \bW_k. 
\end{aligned}
\]
Note that 
\[
(\hat \bU)_{ij} = \E_\bx[\hat \sigma(\< \btheta_i, \bx\> / \sqrt d) \hat \sigma(\< \btheta_j, \bx\>/ \sqrt d)], 
\]
where $\hat \sigma$ is given by 
\[
\hat \sigma ( x ) = \sum_{k = 0}^\ell \lambda_{k}(\sigma_d) B(d, k) Q_k(\sqrt d x).
\]
As a result, we have $\hat \bU \succeq 0$, and hence
\[
\bU = \hat \bU + \bar \bU \succeq \bar \bU. 
\]

In the following, we give a lower bound for $\bar \bU$. Note we have 
\begin{equation}\label{eqn:operator_norm_bound_U_minus_sum}
\Big\| \bar \bU - \Big[\sum_{k=\ell + 1}^{\infty} \lambda_k( \sigma_d)^2 \cdot B(d, k)\Big] \id_N \Big\|_{\op} \le \Big[ \sum_{k = \ell + 1}^\infty \lambda_k(\sigma_d)^2 \cdot B(d, k) \Big] \cdot \Big[ \sup_{k \ge \ell + 1} \| \bW_k - \id_N \|_{\op} \Big]. 
\end{equation}
By Proposition \ref{prop:Delta_bound}, we have 
\begin{equation}\label{eqn:W_first_2l_terms}
\sup_{\ell + 1 \le k \le 2 \ell + 2} \| \bW_k - \id_N \|_{\op} = o_{d, \P}(1).
\end{equation}
Further we have
\[
\begin{aligned}
&\E\Big[\sup_{k \ge 2 \ell + 3} \| \bW_k - \id_N \|_{\op}^2 \Big] \le   \E\Big[\sum_{k \ge 2 \ell + 3} \| \bW_k - \id_N \|_F^2 \Big] \\
 =&  N(N-1) \sum_{k \ge 2 \ell + 3} \E [ Q_k ( \< \btheta , \btheta ' \> )^2] = N(N-1) \sum_{k \ge 2 \ell + 3} B(d, k)^{-1}.
\end{aligned}
\]
For $d$ sufficiently large, there exists $C>0$  such that for any $ p \geq m \equiv 2 \ell+3$:
\[
\begin{aligned}
\frac{B(d,m)}{B(d,p)} = \prod_{k = m}^{p-1} \frac{(2k+d-2)}{(2k +d)}\cdot \frac{(k+1)}{(k+d-2)}  \leq  \prod_{k = m}^{p-1} \frac{1}{1 + (d-3)/(k+1)} \leq \prod_{k = m}^{p-1} e^{- \frac{m+1}{d-2 + m}\cdot \frac{d - 2 }{k+1}} \leq \frac{C}{p^2}\, . 
\end{aligned}
\]
Hence, there exists constant $C'$, such that for large $d$, we have
\[
\sum_{k \ge 2 \ell + 3} B(d, k)^{-1} \le C' \cdot B(d, 2 \ell + 3)^{-1}. 
\]
Recalling that $B(d, 2\ell + 3) = \Theta_d(d^{2\ell + 3})$, and $N = o_d(d^{\ell+1})$, we deduce
\begin{equation}\label{eqn:W_after_2l_terms}
\E\Big[\sup_{k \ge 2 \ell + 3} \| \bW_k - \id_N \|_{\op}^2 \Big] = o_{d}(1).
\end{equation}
Combining Eq. (\ref{eqn:W_first_2l_terms}) and (\ref{eqn:W_after_2l_terms}) we get
\begin{equation}\label{eqn:W_all_terms}
\sup_{k \ge  \ell + 1} \| \bW_k - \id_N \|_{\op} = o_{d, \P}(1). 
\end{equation}
Plug Eq. (\ref{eqn:W_all_terms}) into Eq. (\ref{eqn:operator_norm_bound_U_minus_sum}), we get with high probability
\[
\bar \bU \succeq (1 - \eps)\Big[ \sum_{k = \ell + 1}^{\infty} \lambda_k(\sigma_d)^2 \cdot B(d, k) \Big] \id_N.
\]
Hence the proposition follows.

\subsection{Proof of Proposition \ref{prop:Delta_bound}}

\noindent
{\bf Step 1. Bounding operator norm by moments.}

We define $\bDelta = \bW - \id_d$. Then we have
\[
\begin{aligned}
\bDelta = \begin{cases}
0, & i = j, \\
Q_k^{(d)}(\< \btheta_i, \btheta_j\>), & i \neq j.
\end{cases}
\end{aligned}
\]
For any sequence of integers $p=p(d)$, we have 
\begin{align}
\E[\|\bDelta\|_{\op}]\le \E[\Trace(\bDelta^{2p})^{1/(2p)}]\le \E[\Trace(\bDelta^{2p})]^{1/(2p)}\label{eq:MomentBound}
\end{align}
To prove the proposition, it suffices to show that for any sequence $A_d \to \infty$, we have
\begin{equation}
\lim_{d, N \to \infty, N = O_d(d^{k}  e^{- A_d \sqrt{ \log d} } )} \E[\Trace( \bDelta^{2p} ) ]^{1/(2p)} = 0. 
\end{equation}
In the following, we calculate $\E[\Trace(\bDelta^{2p})]$. 
We have
\[
\begin{aligned}
\E[\Trace(\bDelta^{2p})] = \sum_{\bi = (i_1, \ldots, i_{2p}) \in [N]^{2p}} \E[\Delta_{i_1 i_2} \Delta_{i_2 i_3} \ldots \Delta_{i_{2p} i_1}]. 
\end{aligned}
\]
To calculate this quantity, we will apply repeatedly the following identity, which is an immediate consequence of  Eq.~\eqref{eq:ProductGegenbauer}.
For any $i_1, i_2, i_3$ distinct, we have 
\[
\E_{\btheta_{i_2}}[\Delta_{i_1 i_2} \Delta_{i_2 i_3}] = \frac{1}{B(d,k)}\Delta_{i_1 i_3}. 
\]
Throughout the proof, we will denote by $C,C',C''$ constants that may depend on $k$ but not on $p,d,N$. The value of these constants is allowed to change from line to line.

\noindent
{\bf Step 2. The induced graph and equivalence of index sequences. }

For any index sequence $\bi = (i_1, i_2, \ldots, i_{2p}) \in [N]^{2p}$, we defined an undirected multigraph $G_\bi = (V_\bi, E_\bi)$ associated to index sequence $\bi$. The vertex set $V_\bi$ is the set of distinct elements in $i_1, \ldots, i_{2p}$. The edge set $E_{\bi}$ is formed as follows: for any $j \in [2 p]$  we add an edge between $i_j$ and $i_{j+1}$ (with convention $2 p + 1 \equiv 1$). Notice that this could be a self-edge, or a repeated edge:  $G_\bi = (V_\bi, E_\bi)$  will be --in general-- a multigraph. We denote $v(\bi) = \vert V_\bi \vert$ to be the number of vertices of $G_\bi$, and $e(\bi) = \vert E_\bi \vert$ to be the number of edges (counting  multiplicities). In particular,  $e(\bi) = k$ for $\bi \in [N]^k$. We define
\[
\cT_\star(p) = \{\bi \in [N]^{2p}: G_\bi \text{ does not have self edge}\}. 
\]

For any two index sequences $\bi_1, \bi_2$, we say they are equivalent $\bi_1 \asymp \bi_2$, if  the two graphs $G_{\bi_1}$ and $G_{\bi_2}$
are isomorphic, i.e. there exists an edge-preserving bijection of their vertices (ignoring vertex labels). We denote the equivalent class of $\bi$ to be 
\[
\cC(\bi) = \{ \bj: \bj \asymp \bi\}.
\]
We define the quotient set $\cQ(p)$ by
\[
\cQ(p) = \{ \cC(\bi): \bi \in [N]^{2p} \}. 
\]

For any integer $k \ge 2$ and $\bi = (i_1, \ldots, i_k) \in [N]^k$, we define
\[
M_{\bi} = \E[\Delta_{i_1 i_2} \cdots \Delta_{i_k i_1}].
\]

\begin{lemma}\label{lem:equivalent_class}
The following properties holds for all sufficiently large $N$ and $d$:
\begin{itemize}
\item[$(a)$] For any equivalent index sequences $\bi = (i_1, \ldots, i_{2p}) \asymp \bj = (j_1, \ldots, j_{2p})$, we have $M_{\bi} = M_{\bj}$. 
\item[$(b)$] For any index sequence $\bi \in [N]^{2p} \setminus \cT_\star(p)$, we have $M_{\bi} = 0$. 
\item[$(c)$] For any index sequence $\bi \in \cT_\star(p)$, the degree of any vertex in $G_\bi$ must be even. 
\item[$(d)$] The number of equivalent classes $\vert \cQ(p) \vert \le (2p)^{2p}$. 
\item[$(e)$] Recall that $v(\bi) = \vert V_\bi\vert$ denotes the number  of distinct elements in $\bi$. Then, for any 
$\bi\in [N]^{2p}$, the number of elements in the corresponding equivalence class satisfies $\vert \cC(\bi)\vert \le v(\bi)^{v(\bi)} \cdot N^{v(\bi)}\le p^p N^{v(\bi)}$. 
\end{itemize}
\end{lemma}

\begin{proof}
Properties $(a)$, $(b)$ and $(c)$ are straightforward. Note that $v(\bi) \le 2p$ for any $\bi\in [N]^{2p}$. For property $(d)$, notice that to each  distinct equivalence class
we can associate, in an injective manner, a string of length $2p$ over an alphabet of size $2p$ (simply follow the elements in $\bi$ in order, and replace the 
labels by some canonical ones, e.g. $\{1,2,3,\dots\}$ in order of appearance). 
Therefore the number of classes  is bounded as
\[
\vert \cQ \vert \le (2p)^{2p}
\]

For property $(e)$, we need to bound the number of elements in $\cC(\bi)$ for representative $\bi$ with degree $v(\bi)$. Define a mapping $\psi: \cC(\bi) \to [N]^{v(\bi)}$ as follows.
For $\bi\in [N]^{2p}$, $\psi(\bi)$ is a vector of the distinct elements in $\bi$, listed  in increasing order.
For any $\bk \in [N]^{v(\bi)}$, the pre-image $\psi^{-1}(\bk)$ contains at most $v(\bi)! \le v(\bi)^{v(\bi)}$ elements. As a result, we have 
\[
\vert \cC(\bi) \vert \le \sum_{\bk \in [N]^{v(\bi)}} \vert \psi^{-1}(\bk) \vert \le p^p \cdot N^{v(\bi)}. 
\]
This proves property $(e)$. 
\end{proof}

In view of property $(a)$ in the last lemma, given an equivalence class $\cC=\cC(\bi)$, we will write $M_{\cC} = M_{\bi}$ for the corresponding value common to the equivalence class $\cC$.

\noindent
{\bf Step 3. The skeletonization process. }

For  multi-graph $G$, we say that one of its vertices is \emph{redundant}, if it has degree 2. 
For any index sequence $\bi \in \cT_\star(p) \subset [N]^{2p}$ (i.e. such that $G_\bi$ does not have self-edges), we denote by $r(\bi) \in \N_+$ to be the redundancy of $\bi$, and by
 $\sk(\bi)$ to be the skeleton of $\bi$, both defined  by the following skeletonization process. Let $\bi_0 = \bi \in [N]^{2p}$. 
For any integer $s\ge 0$, if $G_{\bi_s}$ has no redundant vertices then stop and set $\sk(\bi)= \bi_s$.
Otherwise, select a redundant vertex $\bi_s(\ell)$ arbitrarily (the $\ell$-th element of $\bi_s$). If 
$\bi_s(\ell-1) \neq \bi_s(\ell+1)$, then remove $\bi_s(\ell)$ from the graph (and from the sequence), together with its adjacent edges, and connect $\bi_s(\ell-1)$ and
$\bi_s(\ell+1)$ with an edge, and denote $\bi_{s+1}$ to be the resulting index sequence, i.e., $\bi_{s+1} = (\bi_s(1), \ldots, \bi_s(\ell - 1), \bi_s(\ell + 2), \ldots, \bi_s(\endd))$. If $\bi_s(\ell-1) = \bi_s(\ell+1)$, then remove $\bi_s(\ell)$ from the graph (and from the sequence), together with its adjacent edges, and denote $\bi_{s+1}$ to be the resulting index sequence, i.e., $\bi_{s+1} = (\bi_s(1), \ldots, \bi_s(\ell - 1), \bi_s(\ell + 1), \bi_s(\ell + 2), \ldots, \bi_s(\endd))$. (Here $\ell+1$, and $\ell-1$ have to be interpreted modulo $\vert \bi_s \vert$, the length of $\bi_s$.) The redundancy of  $\bi$, denoted by $r(\bi)$,  is the number of vertices removed during the skeletonization process. 

It is easy to see that the outcome of this process is independent of the order in which we select vertices.
\begin{example}
For illustration, we give two examples of skeletonization processes:
\begin{itemize}
\item  Let $\bi = (1, 2, 1, 3, 4, 3)$, and set $\bi_0 = \bi$. First notice that $\{2, 4\}$ are redundant vertices and we can remove them in arbitrary order to get  
$\bi_2 = (1, 3)$. Then notice that $3$ is redundant whence we get $\bi_3 = \{1\}$. Hence we have $r(\bi) = 3$, and $\sk(\bi) = (1)$. 
\item Consider the skeletonization process of $\bj = (1, 2, 3, 2, 4, 3)$. Take $\bj_0 = \bj$. First notice that $\{1, 4\}$ are redundant vertices and can be removed in arbitrary 
order to get  $\bj_2 = (2, 3, 2, 3)$. We see that there is no further redundant vertex in $G_{\bj_1}$, so that $r(\bj) = 2$, and $\sk(\bj) = \bj_1 = (2, 3, 2, 3)$.
\end{itemize}
\end{example}

\begin{lemma}\label{lem:skeleton}
For the above skeletonization process, the following properties hold:
\begin{itemize}
\item[$(a)$] If $\bi \asymp \bj \in [N]^p$, then $\sk(\bi) \asymp \sk(\bj)$. That is, the skeletons of equivalent index sequences are equivalent. 
\item[$(b)$] For any $\bi = (i_1, \ldots, i_k) \in [N]^k$, define
\[
M_{\bi} = \begin{cases} \E[\Delta_{i_1 i_2} \cdots \Delta_{i_k i_1}] & k \ge 2, \\
1&  k =1\, .
\end{cases}
\]
Then we have
\begin{align*}
M_{\bi} = \frac{M_{\sk(\bi)} }{ B(d, k)^{r(\bi)}}. 
\end{align*}
\item[$(c)$] For any $\bi \in \cT_\star(p) \subset [N]^{2p}$, its skeleton is either formed by a single element, or an index sequence whose graph has the property that 
every vertex has degree greater or equal to $4$. 
\end{itemize}
\end{lemma}
\begin{proof}
Property $(a)$ holds by the definition of equivalence which is graph isomorphism. Property $(b)$ used the fact that, if $i \neq j_1$ and $i \neq j_2$, we have 
\[
\E_{\btheta_i}[Q_k^{(d)}(\< \btheta_{j_1}, \btheta_i\>) Q_k^{(d)}(\<\btheta_{j_2}, \btheta_i\>)] = \frac{1}{B(d, k)} Q_k^{(d)}(\< \btheta_{j_1}, \btheta_{j_2}\>),
\]
so that deleting a redundant vertex will contribute a $1/B(d, k)$ factor. 

To show property $(c)$, note that any intermediate index sequence $\bi_s$ in the skeletonization process is such that $G_{\bi_s}$ only has even degree vertices, is connected,
and has no self-edges (by induction). Hence, $G_{\sk(\bi)}$ only has even degree vertices, is connected, and has no self-edges. Note that $G_{\sk(\bi)}$ 
cannot have degree-2 vertices, 
and has at least one vertex (because the last vertex is not removed). Therefore, as long as $\sk(\bi)$ contains at least two vertices, $G_{\sk(\bi)}$ can only contain vertices with degree greater or equal to $4$. 
\end{proof}

Given an index sequence $\bi \in \cT_\star(p)  \subset [N]^{2p}$, we say
$\bi$ is of type 1, if $\sk(\bi)$ contains only one index. We say  $\bi$ is of type 2 if $\sk(\bi)$ has more than one index (so that by Lemma 
\ref{lem:skeleton}, $G_{\sk(\bi)}$ can only contain vertices with degree greater or equal to $4$). Denote the class of type 1 index sequence (respectively type 2 index sequence) 
by $\cT_1(p)$ (respectively $\cT_2(p)$). 
We also denote by $\tcT_a(p)$, $a\in\{1,2\}$ the set of equivalence
classes of sequences in $\cT_a(p)$. This definition makes sense since the equivalence class of the skeleton of a sequence only depends on the equivalence class of the sequence itself.

\noindent
{\bf Step 4. Type 1 index sequences. }

Recall that $v(\bi)$ is the number of vertices in $G_\bi$, and  $e(\bi)$ is the number of edges in $G_\bi$ (which coincides with the length of $\bi$). 
We consider $\bi \in \cT_1(p)$. Since for $\bi \in \cT_1(p)$, every edge of $G_\bi$ must be at most a double edge. Indeed, if $(u_1,u_2)$ had multiplicity larger than $2$ in $G_{\bi}$,
neither $u_1$ nor $u_2$ could be deleted during the skeletonization process, contradicting the assumption that $\sk(\bi)$ contains a single vertex. 
Therefore, we must have $\min_{\bi \in \cT_1} v(\bi) = p + 1$. According the Lemma \ref{lem:skeleton}.$(b)$, for every $\bi \in \cT_1(p)$, we have 
\[
M_\bi = 1/B(d, k)^{v(\bi)-1}. 
\]
Note by Lemma \ref{lem:equivalent_class}.$(e)$, the number of elements in the equivalence class of $\bi$ is $\vert \cC(\bi) \vert \le p^p \cdot N^{v(\bi)}$. Hence we get 
\begin{equation}
\max_{\bi \in \cT_1(p)} \big[\vert \cC(\bi) \vert \vert M_{\bi}\vert \big] \le \sup_{\bi \in \cT_1(p)}  \big[ p^p N^{v(\bi)} / B(d, k)^{v(\bi) - 1}\big] = p^p N^{p + 1} / B(d, k)^{p}\, .
\end{equation}
Therefore
\begin{align}
& \sum_{\bi\in\cT_1(p)} M_{\bi} = \sum_{\cC \in \tcT_1(p)} \vert\cC \vert\,  \vert M_{\cC}\vert\\
\le& |\cQ(p)|p^p \frac{N^{p + 1}}{ B(d, k)^{p}} \le (Cp)^{3p}\frac{N^{p+1}}{d^{kp}}\, . \label{eqn:bound_T1}
\end{align}
where in the last step we used Lemma \ref{lem:equivalent_class} and the fact that $B(d,k)\ge C_0 d^k$ for some $C_0>0$.

\noindent
{\bf Step 5. Type 2 index sequences. }

We have the following simple lemma bounding $M_\bi$. This bound is useful when  $\bi$ is a skeleton.
\begin{lemma}\label{lem:M_estimate}
There exists constants $C$ and $d_0$ depending uniquely on $k$ such that, for any $d\ge d_0(k)$, and any index sequence $\bi \in [N]^m$ with  
$2\le m\le d/(4k)$, we have 
\[
\vert M_\bi \vert \le \left(C m^k\cdot d^{- k }\right)^{m/2}\, .
\]
\end{lemma}

\begin{proof}
By Holder's inequality, we have 
\[
\vert M_\bi \vert =\vert \E[\Delta_{i_1, i_2} \Delta_{i_2 i_3} \cdots \Delta_{i_m i_1}] \vert \le \prod_{s \in [m]} \E[\vert \Delta_{i_s i_{s+1}}\vert^{m}]^{1/{m}} \le \prod_{s \in [m]} \E[ \Delta_{i_s i_{s+1}}^{2m}]^{1/{2m}}. 
\]
The lemma following by the claim that (for $d\ge d_0(k)$)
\[
\E[ \Delta_{i j}^{2m}] = \E[ Q_k^{(d)}(\<\btheta_{i}, \btheta_j\>)^{2m}] = \left(C\, m^k \cdot d^{- k }\right)^m\, .
\]

In the following, we will write $\Coeff\{q(x);x^{\ell}\}$ for the coefficient of $x^{\ell}$ in the polynomial $q(x)$.
To show the above claim, recall that we have, for any $\ell$,
\[
\lim_{d \to \infty} \Coeff\{ Q_k^{(d)}(\sqrt d x) B(d, k)^{1/2} ;x^{\ell}\} = \Coeff\{ \bbHe_k(x) / \sqrt{k!} ;x^{\ell}\}\, .
\]
Therefore there exists a constant $C_0$ such that for all $d$ large enough
\[
B(d, k)^{1/2}\max_{\ell\le k}\left|\Coeff\{ Q_k^{(d)}(\sqrt d x) ;x^{\ell} \}\right| \le C_0\, .
\]
As a consequence, for any integer $m$, we have
\begin{align}
\max_{\ell\le km}\left|\Coeff\{ Q_k^{(d)}(\sqrt d x)^m ;x^{\ell} \}\right|&\le k^m\max_{\ell\le k}\left|\Coeff\{ Q_k^{(d)}(\sqrt d x) ;x^{\ell} \}\right|^m\\
& \le k^mC_0^m \, B(d,k)^{-m/2}\, .\label{eq:BoundCoefficients}
\end{align}
Define the random variable $G_d = \< \btheta_i, \btheta_j\> / \sqrt d$ for $\btheta_i, \btheta_j \sim \Unif(\S^{d-1}(\sqrt d))$. The probability distribution of 
$G_d$ is given by $\tau^1_{d-1}$ given in Eq.~\eqref{eq:taud-def} below. Hence defining $A_d\equiv \Gamma(d-1)/(2^{d-2}\sqrt{d} \,\Gamma((d-1)/2)^2)$, we have
(since $A_d\le 1$ for all $d$ large enough)
\begin{align*}
\E \{G_d^{2\ell}\}&= A_d\, \int_{[-\sqrt{d},\sqrt{d}]} \left(1-\frac{x^2}{d}\right)^{\frac{d-3}{2}}\, x^{2\ell}\de x \le  \int_{\reals} e^{-\frac{d-3}{2d}x^2}\,  x^{2\ell}\de x \\
&\le \sqrt{2 \pi} \left(1-\frac{3}{d}\right)^{-\ell- 1/2}\E\{G^{2\ell}\}\, ,
\end{align*}
where $G\sim\normal(0,1)$. Therefore, for all $\ell\le d/2$,
\begin{align}
\E \{G_d^{2\ell}\} \le 100\, \frac{(2\ell)!}{\ell!2^{\ell}}\le 100\, \ell^{\ell}\, . \label{eq:BoundMoments}
\end{align}
Combining the above two upper bounds \eqref{eq:BoundCoefficients} and \eqref{eq:BoundMoments}, we have 
\begin{align*}
\E[Q_k^{(d)}(\< \btheta_i, \btheta_j\>)^{2m}] &= \sum_{j=0}^{2km}\Coeff\big\{Q_k^{(d)}(\< \btheta_i, \btheta_j\>)^{2m};x^j\big\}\, \E\{G_d^j\}\\
& \le \sum_{\ell=0}^{km}\Coeff\big\{Q_k^{(d)}(\< \btheta_i, \btheta_j\>)^{2m};x^{2\ell}\big\}\, \E\{G_d^{2\ell}\}\\
& \le C_1^m B(d,k)^{-m}\sum_{\ell=0}^{km}\ell^{\ell}\le 2C_2^m B(d,k)^{-m} (km)^{km}\le C_3^m m^{km} B(d,k)^{-m}\, .
\end{align*}
By noting that $B(d, k) \ge C_0 d^k$ for some $C_0>0$, this proves the claim. 
\end{proof}

Suppose $\bi \in \cT_2(p)$, and denote $v(\bi)$ to be the number of vertices in $G_\bi$. We have, for a sequence $p = o_d (d) $ 
\begin{align*}
|M_\bi| &\stackrel{(1)}{=} \frac{|M_{\sk(\bi)}|}{B(d, k)^{r(\bi)}}\\
&\stackrel{(2)}{\le}  \left(\frac{Ce(\sk(\bi))}{d}\right)^{k \cdot e( \sk(\bi))/2} (C'd)^{- r(\bi) k} \\
& \stackrel{(3)}{\le}   \left(\frac{Cp}{d}\right)^{k \cdot e(\sk(\bi))/2} (C'd)^{- r(\bi) k}   \\
& \stackrel{(4)}{\le}   \left(\frac{Cp}{d}\right)^{k \cdot v(\sk(\bi))} (C'd)^{- r(\bi) k}   \\
&\stackrel{(5)}{\le} C^{v(\bi)}p^{k \cdot v(\sk(\bi))}  d^{- (v(\sk(\bi)) + r(\bi)) \cdot k}\\
&\stackrel{(6)}{\le}(Cp)^{k\cdot v(\bi)}  d^{- v(\bi) k}\, .
\end{align*}
Here $(1)$ holds by Lemma \ref{lem:skeleton}.$(b)$; $(2)$ by Lemma \ref{lem:M_estimate}, and the fact that $\sk(\bi) \in [N]^{e(\sk(\bi))}$, together by $B(d,k)\ge C_0 d^k$;
$(3)$ because $e(\sk(\bi))\le 2p$; $(4)$ by Lemma \ref{lem:skeleton}.$(c)$, implying  that for $\bi \in \cT_2(p)$,  each vertex of $G_{\sk(\bi)}$ has degree greater or equal to $4$, so that $v(\sk(\bi)) \le e(\sk(\bi))/2$ (notice that for $d\ge d_0(k)$ we can assume $Cp/d<1$). 
Finally, $(5)$ follows since $r(\bi), v(\sk(\bi))\le v(\bi)$, and $(6)$   the definition of $r(\bi)$ implying $r(\bi) = v(\bi) - v(\sk(\bi))$. 

Note by Lemma \ref{lem:equivalent_class}.$(e)$, the number of elements in equivalent class $\vert \cC(\bi) \vert \le p^{v(\bi)} \cdot N^{v(\bi)}$. 
Since $v(\bi)$ depends only on the equivalence class of $\bi$, we will write, with a slight abuse of notation $v(\bi) = v(\cC(\bi))$. 
Notice that the number of equivalence classes with $v(\cC) = v$ is upper bounded by the number multi-graphs with $v$ vertices and  $2p$ edges, which is
at most $v^{4p}$.
Hence we get 
\begin{align}
\sum_{\bi\in\cT_2(p)} M_{\bi} &\le\sum_{\cC\in\tcT_2(p)}\vert \cC \vert \vert M_{\cC}\vert \\
&\le \sum_{\cC\in\tcT_2(p)}  (Cp)^{(k+1) v(\cC)} \left(\frac{N}{d^k}\right)^{v(\cC)} \\
& \le \sum_{v=2}^{2p} v^{4p} \left(\frac{CNp^{k+1}}{d^k}\right)^{v}. 
\end{align}
Define $\eps = CNp^{k+1} /d^k$. We will assume hereafter that $p$ is selected such that
\begin{align}
2p\le -\log\left(\frac{CNp^{k+1} }{d^k}\right)\,. \label{eq:p-condition}
\end{align}
By calculus and condition (\ref{eq:p-condition}), the function $F(v) = v^{4p}\eps^v$ is maximized over $v\in [2,2p]$ at $v=2$, whence
\begin{align}
\sum_{\bi\in\cT_2(p)} M_{\bi} &\le 2p\, F(2) \le C^p \left(\frac{N}{d^k}\right)^{2}\, .\label{eqn:bound_T2}
\end{align}

\noindent
{\bf Step 6. Concluding the proof. }

Using  Eqs. (\ref{eqn:bound_T1}) and (\ref{eqn:bound_T2}), we have, for any $p = o_d (d)$ satisfying Eq.~\eqref{eq:p-condition}, we have
\begin{align}
\E[\Trace(\bDelta^{2p})] & = \sum_{\bi = (i_1, \ldots, i_{2p}) \in [N]^{2p}} M_\bi = \sum_{\bi \in \cT_1(p)} M_{\bi}+\sum_{\bi\in\cT_2(p)}  M_{\bi}\\
& \le  (Cp)^{3p}\frac{N^{p+1}}{d^{kp}} +  C^p \left(\frac{N}{d^k}\right)^{2}\, .
\end{align}
Form Eq.~\eqref{eq:MomentBound}, we obtain
\begin{align}
\E[\|\bDelta\|_{\op}] \le C \left\{ p^{3/2} N^{1/(2p)}\sqrt{\frac{N}{d^{k}}}  + \left(\frac{N}{d^{k}}\right)^{1/p} \right\}. 
\end{align}
Finally setting $N = d^k e^{-2A\sqrt{\log d}}$ and $p = (k/A)\sqrt{\log d}$, this yields
\begin{align}
\E[\|\bDelta\|_{\op}] \le C \left\{ e^{-\frac{A}{4}\sqrt{\log d}} +e^{-2A^2/k}\right\}\,.
\end{align}
Therefore, as long as $A\to \infty$, we have $\E[\|\bDelta\|_{\op}]\to 0$. It is immediate to check that the above choice of $p$
satisfies the required conditions $p = o_d (d)$ and Eq.~\eqref{eq:p-condition} for all $d$ large enough.

\section{Proof of Theorem \ref{thm:RF_lower_upper_bound}.(b): \RF\, model upper bound}\label{sec:proof_RFK_upper}

Recall that $(\bw_i)_{i \in [N]} \sim \Unif(\S^{d-1})$ independently. We define $\btheta_i = \sqrt d \cdot \bw_i$ for $i \in [N]$, so that $(\btheta_i)_{i \in [N]} \sim \Unif(\S^{d-1}(\sqrt d))$ independently. Let $\bW = (\bw_1, \ldots, \bw_N)$, and $\bTheta = (\btheta_1, \ldots, \btheta_N)$. We denote $\E_\btheta$ to be the expectation operator with respect to $\btheta \sim \Unif(\S^{d-1}(\sqrt d))$, $\E_\bx$ to be the expectation operator with respect to $\bx \sim \Unif(\S^{d-1}(\sqrt d))$, and $\E_\bw$ to be the expectation operator with respect to $\bw \sim \Unif(\S^{d-1}(1))$.  

Without loss of generality, assume that $\lbrace f_d \rbrace_{d \geq 0 }$ are polynomials of degree at most $\ell$, i.e.\,$f_d = \proj_{\leq \ell } f_d$. We denote the expansion of $\sigma_d$ in terms of Gegenbauer polynomials by (for $\btheta, \bx \in \S^{d-1}(\sqrt d)$)
\[
\sigma_d ( \< \btheta, \bx \> /\sqrt{d} )  = \sum_{m=0}^\infty \lambda_{m} (\sigma_d) B(d,m) Q^{(d)}_m (\< \btheta , \bx \> ),
\]
where
\[
\lambda_{m} ( \sigma_d) = \< \sigma_d( \< \be , \cdot \>) , Q_m^{(d)} ( \sqrt{d} \< \be , \cdot \> ) \>_{L^2(\S^{d-1}(\sqrt d))}.
\]

Denote $\mathcal{L} = L^2 ( \S^{d-1} ( \sqrt{d}) \rightarrow \R)$. We introduce the operator $\T : \cL \to \cL$, such that for any $g \in \cL$
\[
\T g ( \btheta ) = \< \sigma_d ( \< \btheta , \cdot \>/ \sqrt{d}), g \>_{L^2} = \E_{\bx} [  \sigma_d ( \< \btheta , \bx \>/ \sqrt{d}) g (\bx)].
\]
In particular, for any $k\in \N$ and $1 \leq u \leq B(d,k)$, we have
\begin{equation}\label{eq:diag_T}
\T Y^{(d)}_{ku} (\btheta) = \sum_{m=0}^\infty \lambda_{m} (\sigma_d) B(d,m) \E_\bx [ Q^{(d)}_m (\< \btheta , \bx \> ) Y^{(d)}_{ku} ( \bx ) ] = \lambda_{k} (\sigma_d) Y^{(d)}_{ku} ( \btheta). 
\end{equation}
It is easy to check that $\T^*$ (the adjoint operator) has the same expression as $\T$ with $\bx$ and $\btheta$ swapped. We define the operator $\K: \cL \to \cL$ as $\K \equiv \T \T^*$. For any $g \in \cL$, we have
\[
\K g ( \btheta) = \E_{\btheta'} [ K (\btheta , \btheta ' ) g ( \btheta ' ) ],
\]
where 
\[
K( \btheta , \btheta ' )  = \E_{\bx} [ \sigma_d ( \< \btheta , \bx \> /\sqrt{d}) \sigma_d ( \< \btheta '  , \bx \> /\sqrt{d})  ].
\]

We will restrict ourselves to the subspace $V_{d,\leq\ell}$ of polynomials of degree less or equal to $\ell$. We have for $0 \leq k \leq \ell$ and $1 \leq u \leq B(d,k)$, 
\begin{equation}\label{eq:diag_K}
\K Y^{(d)}_{ku} = \lambda_{k} (\sigma_d)^2 Y^{(d)}_{ku}. 
\end{equation}
Hence $\lbrace Y^{(d)}_{ku} \rbrace_{ 0 \leq k \leq \ell, 1 \leq u \leq B(d,k)}$ is an orthogonal basis that diagonalizes $\K$ on $V_{d,\leq\ell}$. By Assumption \ref{ass:activation_lower_upper_RF_v2}.(b), we deduce that $\K$ is a bijection from $V_{d,\leq\ell}$ to itself for $d$ sufficiently large. In particular, its restricted inverse $\K^{-1}\vert_{V_{d, \leq\ell}}$ is well defined.

Consider $\hat{f}_{\RF} (\bx ; \bTheta, \ba ) = \sum_{i = 1}^N a_i \sigma_d (\<\btheta_i , \bx \> / \sqrt d)$. We can expand the risk achieved at parameter $\ba$ as
\[
\begin{aligned}
\E_\bx [ (f_d (\bx) - \hat{f}_{\RF} (\bx ; \bTheta, \ba) )^2 ] = & \| f_d \|^2_{L^2} - 2 \sum_{i=1}^N a_i  \< \sigma_d ( \< \btheta_i , \cdot \>/ \sqrt{d}), f \>_{L^2} \\
& +  \sum_{i,j=1}^N a_i a_j \< \sigma_d( \< \btheta_i , \cdot \>/\sqrt{d}), \sigma_d( \< \btheta_j , \cdot \>/\sqrt{d}) \>_{L^2}.
\end{aligned}
\]
Let us define $\alpha ( \btheta) \equiv (\K^{-1} \T f_d) (\btheta)$ and choose $a_i = N^{-1} \alpha ( \btheta_i)$. We consider the expectation over $\bTheta$ of the \RF\,risk:
\[
\begin{aligned}
\E_{\bTheta} [ R_{\RF}(f_d , \bTheta / \sqrt d) ] =& \E_{\bTheta} \Big[ \inf_{\ba \in \R^N} \E_{\bx} [(f_d (\bx) - \hat f_{\RF} (\bx; \bTheta, \ba))^2] \Big]\\
\leq & \E_{\bTheta, \bx} \Big[(f_d (\bx) - \hat f_{\RF} (\bx; \bTheta, \ba) )^2 \Big] \Big\vert_{a_i = N^{-1} \alpha(\btheta_i)}\\
=& \| f_d \|^2_{L^2} - 2 \< \K^{-1} \T f_d ,\T f_d \>_{L^2}+ \< \K^{-1} \T f_d , \K [\K^{-1} \T f_d]\> \\
&+ \frac{1}{N} \big[  \E_{\btheta} [K(\btheta,\btheta) ( \K^{-1} \T f_d ( \btheta ))^2] - \< \K^{-1} \T f_d , \K [\K^{-1}\T f_d]\>\big].
\end{aligned}
\]
It is easy to check that $\T^* \K^{-1} \T \vert_{V_{d,\leq\ell}} = \id \vert_{V_{d,\leq\ell}}$. Hence
\[
\begin{aligned}
\E_{\bTheta} [ R_{\RF}(f_d , \bTheta / \sqrt d) ]  \leq & \| f_d \|^2_{L^2} - 2 \| f_d \|^2_{L^2} + \| f_d \|^2_{L^2} + \frac{1}{N} \Big[ \sup_{\btheta} \vert K (\btheta , \btheta) \vert \Big]  \E_{\btheta} [ ( \K^{-1} \T f_d ( \btheta ))^2]\\
= &  \frac{ \E_\bx[\sigma_d(\< \be, \bx\>)^2]}{N} \| \K^{-1} \T f_d  \|_{L^2}^2.
\end{aligned}
\]
Recall the decomposition of $f_d$ in terms of spherical harmonics (and note we assumed $f_d$ is a degree $\ell$ polynomial)
\[
f_d ( \bx )  = \sum_{k=0}^\ell \sum_{u = 1}^{B(d,k)} \lambda^{(d)}_{ku} (f_d) Y^{(d)}_{ku} (\bx ),
\]
and the equations \eqref{eq:diag_T} and \eqref{eq:diag_K}, we get
\[
\| \K^{-1} \T f_d  \|_{L^2}^2 = \sum_{k=0}^\ell  \frac{1}{\lambda_{k} (\sigma_d)^2 }  \| \proj_k f_d \|^2_{L^2}.
\]
As a result, we deduce that
\[
\E_{\bTheta} [ R_{\RF}(f_d , \bTheta / \sqrt d) ]  \leq \| \sigma_d \|_{L^2}^2  \Big[  \sum_{k = 0}^\ell 1/ (N \lambda_{k}(\sigma_d)^2) \Big] \| f_d \|^2_{L^2}.
\]
Hence, by Assumption \ref{ass:activation_lower_upper_RF_v2}.(b), and from the assumption that $N = \omega_d ( d^\ell )$, we deduce that the risk $R_{\RF}(f_d , \bTheta / \sqrt d) / \| f_d \|_{L^2}^2$ converges in $L^1$ to $0$, and therefore in probability.

\section{Proof of Theorem \ref{thm:NT_lower_upper_bound}.(a): \NT\, model lower bound}
\label{sec:proof_NTK_lower}

\subsection{Preliminaries}

We begin with some notations and simple remarks. 
\begin{lemma}\label{lemma:square_integrable}
Assume $\sigma$ is an activation function with $\sigma(u)^2\le c_0 \, \exp(c_1\, u^2/2)$ for some constants $c_0 > 0$ and $c_1<1$.
Then 
\begin{enumerate}
\item[$(a)$] $\E_{G \sim \normal(0, 1)}[\sigma(G)^2] < \infty$. 
\item[$(b)$] Let $\|\bw\|_2=1$. Then there exists $d_0=d_0(c_1)$ such that, for $\bx \sim \Unif(\S^{d-1}(\sqrt  d))$,
\begin{align}
\sup_{d \ge d_0} \E_{\bx}[\sigma(\<\bw,\bx\>)^2] < \infty\, .
\end{align}
\item[$(c)$] Let $\|\bw\|_2=1$. Then there exists a coupling of $G\sim\normal(0,1)$ and $\bx\sim \Unif(\S^{d-1}(\sqrt  d))$ such that
\begin{align}
\lim_{d \to \infty} \E_{\bx, G} [(\sigma(\<\bw,\bx\>) - \sigma(G))^2] =& 0.
\end{align}
\end{enumerate}
\end{lemma}
\begin{proof}
Claim 1 is obvious.

For claim 2, note that the probability distribution of $\<\bw,\bx\>$ when $\bx \sim \Unif(\S^{d-1}(\sqrt  d))$ is given by
\begin{align}
\tau^1_{d-1}(\de x) &= C_d\, \left(1-\frac{x^2}{d}\right)^{\frac{d-3}{2}}\bfone_{x\in [- \sqrt d, \sqrt d]}\de x\, ,\label{eq:taud-def}\\
C_d & = \frac{\Gamma(d-1)}{2^{d-2}\sqrt{d} \,\Gamma((d-1)/2)^2}\, . \
\end{align}
A simple calculation shows that $C_d\to (2\pi)^{-1/2}$ as $d\to\infty$, and hence $\sup_d C_d\le \overline{C}<\infty$. 
Therefore
\begin{align}
\E_{\bx}[\sigma(\<\bw,\bx\>)^2]  &\le \overline{C} \int_{- \sqrt d}^{\sqrt d} \left(1-\frac{x^2}{d}\right)^{\frac{d-3}{2}}\, e^{c_1x^2/2}\de x\\
&\le \overline{C} \int_{\reals} e^{-\frac{d-3}{2d}x^2}\, e^{c_1x^2/2}\de x\le C'\, ,
\end{align}
where the last inequality holds provided $d\ge d_0 = 10/(1-c_1)$.

Finally, for point 3, without loss of generality we will take $\bw=\be_1$, so that $\<\bw,\bx\> = x_1$.
By the same argument  given above (and since both $G$ and $x_1$ have densities bounded uniformly in $d$), for any $M>0$ 
we can choose $\sigma_M$ bounded continuous so that
for any $d$, 
\begin{align}
\E_{\bx, G} [(\sigma(x_1) - \sigma(G))^2] \le \E_{\bx, G} [(\sigma_M(x_1) - \sigma_M(G))^2]+\frac{1}{M}\, .
\end{align}
It is therefore sufficient to prove the claim for $\sigma_M$. Letting $\bxi\sim\normal(0,\id_{d-1})$, independent of $G$,
we construct the coupling via
\begin{align}
x_1 = \frac{G\sqrt{d}}{\sqrt{G^2+\|\bxi\|_2^2}}\, ,\;\;\; \;\;\bx' = \frac{\bxi\sqrt{d}}{\sqrt{G^2+\|\bxi\|_2^2}}\, ,
\end{align}
where we set $\bx = (x_1,\bx')$. We thus have $x_1\to G$ almost surely, and the claim follows by weak convergence. 
\end{proof}

We denote the Hermite decomposition of $\sigma$ by
\begin{align}
\sigma(x) = \sum_{k=0}^\infty \frac{\mu_k(\sigma)}{ k!} \bbHe_k(x)\, ,\;\;\;\;\;
\mu_k(\sigma) \equiv \E_{G \sim \normal(0, 1)}[\sigma(G) \bbHe_k(G) ].
\end{align}

We state separately the assumptions of Theorem \ref{thm:NT_lower_upper_bound}.(a) for future reference.
\begin{assumption}[Integrability condition]\label{ass:NT-square_integrable}
The activation function $\sigma$ is weakly differentiable with weak derivative $\sigma'$. 
There exist constants $c_0$, $c_1$, with $c_0 > 0$ and $c_1<1$ such that, for all $u\in\reals$, $\sigma'(u)^2\le c_0\, \exp(c_1u^2/2)$.
\end{assumption}

\begin{assumption}[Level-$\ell$ non-trivial Hermite components]\label{ass:NT-non_trivial_components}
Recall that $\mu_k(h) \equiv \E_{G\sim\normal(0,1)}[h(G)\bbHe_k(G)]$ denote the $k$-th coefficient of the Hermite expansion of 
$h\in L_2(\reals,\gamma)$ (with $\gamma$ the standard Gaussian measure).

Then there exists $k_1,k_2\ge 2\ell+7$ such that $\mu_{k_1}(\sigma') ,\mu_{k_2}(\sigma')\neq 0$ and
\begin{align}
\frac{\mu_{k_1}(x^2\sigma')}{\mu_{k_1}(\sigma')}\neq \frac{\mu_{k_2}(x^2\sigma')}{\mu_{k_2}(\sigma')} \, .
\end{align}
\end{assumption}
It is also useful to notice that the Hermite coefficients of $x^2\sigma'(x)$
can be computed from the ones of $\sigma'(x)$ using the relation $\mu_k (x^2 \sigma ' )  = \mu_{k+2} (\sigma') + [1+2k] \mu_k(\sigma') + k(k-1) \mu_{k-2} (\sigma') $.

\subsection{Proof of Theorem \ref{thm:NT_lower_upper_bound}.(a): Outline}

The proof for the \NT\, model follows the same scheme as for the \RF\, case. However, several steps are 
technically more challenging. We will follow the same notations introduced in Section \ref{sec:OutlineRF}. In particular
$\E_{\bx}, \E_{\bw}, \E_{\btheta}$ will denote, respectively, expectation with respect to $\bx\sim\Unif(\S^{d-1}(\sqrt{d}))$,
$\bw\sim\Unif(\S^{d-1}(1))$, $\btheta\sim\Unif(\S^{d-1}(\sqrt{d}))$.

We define  the random vector $\bV = (\bV_1, \ldots, \bV_N)^\sT\in\reals^{Nd}$, where, for each $j\le N$, $\bV_j\in\reals^d$, and analogously
$\bV_{\le \ell + 1} = (\bV_{1, \le \ell + 1}, \ldots, \bV_{N, \le \ell + 1})^\sT\in\reals^{Nd}$, $\bV_{> \ell + 1} = (\bV_{1, > \ell + 1}, \ldots, \bV_{N, > \ell + 1})^\sT\in\reals^{Nd}$, as follows
\[
\begin{aligned}
\bV_{i, \le \ell + 1} =& \E_{\bx}[[\proj_{\le \ell + 1} f_d](\bx) \sigma'(\< \btheta_i, \bx\>/\sqrt d) \bx],\\
\bV_{i, > \ell + 1} =& \E_{\bx}[[\proj_{> \ell + 1} f_d](\bx) \sigma'(\< \btheta_i, \bx\>/\sqrt d) \bx],\\
\bV_i =& \E_{\bx}[f_d(\bx) \sigma'(\< \btheta_i, \bx\>/\sqrt d) \bx] = \bV_{i, \le \ell + 1} + \bV_{i, > \ell + 1}. \\
\end{aligned}
\]
We define the random matrix $\bU = (\bU_{ij})_{i, j \in [N]}\in\reals^{Nd\times Nd}$, where, for each $i,j\le N$, $\bU_{ij}\in\reals^{d\times d}$, 
is given by
\begin{align}
\bU_{ij} = \E_{\bx}[\sigma'(\< \bx, \btheta_i\>/\sqrt d) \sigma'(\< \bx, \btheta_j\>/\sqrt d) \bx \bx^\sT]. 
\label{eq:NT-Kernel}
\end{align}
Proceeding as for the \RF\, model, we obtain 
\[
\begin{aligned}
& \Big\vert R_{\NT}(f_d) - R_{\NT}(\proj_{\le \ell + 1} f_d) - \| \proj_{> \ell + 1} f_d \|^2_{L^2} \Big\vert \\
\le& 2  \| f_d \|_{L^2} \| \bU^{-1/2} \|_{\op} \| \bV_{> \ell + 1} \|_2 +  \| \bU^{-1} \|_{\op} \| \bV_{>\ell + 1}\|_2^2.
\end{aligned}
\]

We claim that we have
\begin{align}
\| \bV_{> \ell + 1} \|_2 / \| \proj_{> \ell + 1} f_d \|_{L^2} =& o_{d,\P}(1), \label{eqn:bound_V_NTK}\\
\| \bU^{-1} \|_{\op} =& O_{d,\P} (1),\label{eqn:bound_inverse_U_NTK}
\end{align}
This is achieved in the following two propositions.
\begin{proposition}[Expected norm of $\bV$]\label{prop:expected_V_NTK}
Let $\sigma$ be an activation function satisfying Assumption \ref{ass:NT-square_integrable}. Define
\[
\begin{aligned}
\cE_{\ge \ell} \equiv& \E_{\btheta}[ \< \E_\bx[ \proj_{\ge \ell} f_\star(\bx) \sigma'(\<\btheta, \bx\>/\sqrt d) \bx], \E_\bx[ \proj_{\ge \ell} f_\star(\bx) \sigma'(\<\btheta, \bx\>/\sqrt d) \bx] \>]\\
=& \E_{\bx, \bx' } [\proj_{\ge \ell} f_\star(\bx) \proj_{\ge \ell} f_\star(\bx') \E_\btheta[\sigma'(\< \btheta, \bx\>/\sqrt d) \sigma'(\< \btheta, \bx'\>/\sqrt d) \< \bx, \bx'\>] ].
\end{aligned}
\]
where expectation is with respect to $\bx, \bx' \sim_{i.i.d.} \Unif(\S^{d-1}(\sqrt d))$.
Then there exists a constant $C$ (depending only on the constants in Assumption \ref{ass:NT-square_integrable}) such that, for any $\ell \ge 1$ and $d \ge 6$, 
\[
\cE_{\ge \ell} \le \frac{Cd}{B(d,\ell)}\, \| \proj_{\ge \ell} f_\star \|_{L^2(\S^{d-1}(\sqrt d))}^2 \, .
\]
\end{proposition}

\begin{proposition}[Lower bound on the kernel matrix]\label{prop:kernel_lower_bound_NTK}
Let $N = o_d(d^{\ell+1})$ for some $\ell \in \integers_{>0}$, and  $(\btheta_i)_{i \in [N]} \sim \Unif(\S^{d-1}(\sqrt d))$ independently. Let $\sigma$ be an activation that satisfies
 Assumption \ref{ass:NT-square_integrable} and Assumption \ref{ass:NT-non_trivial_components}. 
Let $\bU \in \R^{Nd \times Nd}$ be the  kernel matrix with $i,j$ block $\bU_{ij}\in\reals^{d\times d}$ defined by Eq.~\eqref{eq:NT-Kernel}. Then there exists a constant $\eps > 0$ that depends on the activation function $\sigma$, such that 
\[
\begin{aligned}
\lambda_{\min}(\bU) \ge& \eps \\
\end{aligned}
\]
with high probability as $d \to \infty$. 
\end{proposition}

These two propositions will be proven in the next sections. Proposition \ref{prop:expected_V_NTK} shows that 
\[
\begin{aligned}
\E[\| \bV_{> \ell + 1}\|_2^2] \le& \frac{C\,N d}{ B(d, \ell + 2)} \| \proj_{> \ell + 1} f_d \|_2^2 .
\end{aligned}
\]
Note $B(d, \ell + 2) = \Theta_d(d^{\ell + 2})$, and $N = o_d(d^{\ell+1})$. By Markov inequality, we have Eq. (\ref{eqn:bound_V_NTK}).  Equation (\ref{eqn:bound_inverse_U_NTK}) follows simply by Proposition \ref{prop:kernel_lower_bound_NTK}. This proves the theorem.

\subsection{Proof of Proposition \ref{prop:expected_V_NTK}} 

We denote the Gegenbauer decomposition of $\sigma'(\< \be, \cdot\>)$ by
\[
\sigma'( \< \be, \bx\>) = \sum_{k=0}^\infty B(d, k) \lambda_k(\sigma) Q_k(\sqrt d\< \be, \bx\>), 
\]
where
\[
\lambda_k(\sigma') = \< \sigma'(\< \be, \cdot\>), Q_k(\sqrt d \< \be, \cdot\>) \>_{L^2}. 
\]
By Lemma \ref{lemma:square_integrable}, applied to function $\sigma'$ (instead of $\sigma$), under Assumption \ref{ass:NT-square_integrable},
we have $\| \sigma'(\< \be, \cdot\>) \|_{L^2}^2 \le C$ (for $C$ a constant independent of $d$). We therefore have (recalling the 
normalization of the Gegenbauer polynomials in Eq.~\eqref{eq:GegenbauerNormalization})
\begin{equation}\label{eqn:bound_of_lambda_sigma_NTK}
\sum_{k=0}^\infty \lambda_k(\sigma')^2  B(d, k) = \| \sigma'(\< \be, \cdot\>) \|_{L^2}^2 \le C. 
\end{equation}
We define the \NT\, kernel by
\[
H(\bx, \bx') = \E_{\btheta}\big[\sigma'(\< \btheta, \bx\>/\sqrt d) \sigma'(\< \btheta, \bx'\>/\sqrt d)\big]\< \bx, \bx'\>. 
\]
Then
\begin{equation}\label{eqn:expression_Kernel_NTK}
\begin{aligned}
H(\bx, \bx') =& \E_{\btheta} \Big[ \sum_{k=0}^\infty B(d, k) \lambda_k(\sigma') Q_k(\< \btheta, \bx\>) \sum_{l=0}^\infty B(d, l) \lambda_l(\sigma') Q_l(\< \btheta, \bx'\>) \Big]\< \bx, \bx'\>\\
=& \sum_{k=0}^\infty B(d, k)^2 \lambda_k(\sigma')^2 \E_{\btheta} \Big[Q_k( \< \btheta, \bx\> ) Q_k(\< \btheta, \bx'\>) \Big] \< \bx, \bx'\> \\
=& \sum_{k=0}^\infty d \cdot B(d, k) \lambda_k(\sigma')^2 Q_k(\< \bx, \bx'\>) \< \bx, \bx'\>/d\, ,
\end{aligned}
\end{equation}
where in the last step we used Eq.~\eqref{eq:ProductGegenbauer}.
By the recurrence relationship for Gegenbauer polynomials \eqref{eq:RecursionG},  we have 
\[
\frac{t}{d}\, Q_k(t) =  s_{d, k} Q_{k-1}(t) + t_{d, k} Q_{k+1}(t),
\]
where 
\[
\begin{aligned}
s_{d, k} =& \frac{k}{2k + d - 2}, \\
t_{d, k} =& \frac{k + d - 2}{2k + d - 2}.
\end{aligned}
\]
We use the convention that $t_{d, -1} = 0$. This gives
\begin{equation}\label{eqn:bound_s_t}
\sup_{d \ge 6, k\ge 0} [s_{d, k+1} + t_{d, k-1}] = \sup_{d \ge 6, k\ge 0}\Big[ \frac{k + 1}{2k + d} + \frac{k + d - 3}{2k + d - 4} \Big] \le 2.
\end{equation}
Hence we get
\[
\begin{aligned}
H(\bx, \bx') =& \sum_{k=0}^\infty d \cdot B(d, k) \lambda_k(\sigma')^2 Q_k(\< \bx, \bx'\>) \< \bx, \bx'\>/d\\
=&\sum_{k=0}^\infty d \cdot B(d, k) \lambda_k(\sigma')^2 [s_{d, k} Q_{k-1}(\< \bx, \bx'\> ) + t_{d, k} Q_{k+1}(\< \bx, \bx'\> )]\\
=& \sum_{k=0}^\infty  \Gamma_{d, k} Q_k(\< \bx, \bx'\>),
\end{aligned}
\]
where
\[
\begin{aligned}
\Gamma_{d, k} =& d \cdot [t_{d, k-1} \lambda_{k-1}(\sigma')^2 B(d, k-1) + s_{d, k+1} \lambda_{k+1}(\sigma')^2 B(d, k+1)] \le 2 d C\, .
\end{aligned}
\]
The last inequality follows by Eqs. (\ref{eqn:bound_of_lambda_sigma_NTK}) and (\ref{eqn:bound_s_t}). 

We define
\[
\begin{aligned}
\cE_k \equiv&\, \E_{\btheta}[ \< \E_\bx[ \proj_{k} f_\star(\bx) \sigma'(\<\btheta, \bx\>/\sqrt d) \bx], \E_\bx[ \proj_{k} f_\star(\bx) \sigma'(\<\btheta, \bx\>/\sqrt d) \bx] \>]\\
=& \E_{\bx, \bx'}[ [\proj_{k} f_\star](\bx) H(\bx, \bx') [\proj_k f_\star](\bx')]. 
\end{aligned}
\]
Using the fact that the kernel $H$ preserve the decomposition \eqref{eq:SpinDecomposition}, we have 
\[
\cE_{\ge \ell} = \sum_{k \ge \ell} \cE_k. 
\]
Note by Eq. (\ref{eqn:expression_Kernel_NTK}), we have  (as always, expectations are with respect to $\bx,\bx'\sim  \Unif(\S^{d-1}(\sqrt d))$ independently)
\[
\begin{aligned}
\cE_k=& \E_{\bx, \bx'}[ [\proj_{k} f_\star](\bx) H(\bx, \bx') [\proj_k f_\star](\bx')]\\
=& \E_{\bx, \bx'}\Big[ \sum_{l=1}^{B(d,k)} \lambda_{kl}(f_\star) Y_{kl}(\bx)  \Gamma_{d, k}  Q_k(\< \bx, \bx'\> ) \sum_{s=1}^{B(d,k)} \lambda_{ks}(f_\star) Y_{ks}(\bx') \Big]\\
=& \Gamma_{d, k}  \sum_{l=1}^{B(d,k)}\sum_{s=1}^{B(d,k)} \lambda_{kl}(f_\star) \lambda_{ks}(f_\star) \E_{\bx, \bx' }\Big[  Y_{kl}(\bx)   Q_k(\< \bx, \bx'\>)   Y_{ks}(\bx') \Big]\\
=& \frac{\Gamma_{d, k}}{B(d, k)} \times  \sum_{l=1}^{B(d,k)}\sum_{s=1}^{B(d,k)} \lambda_{kl}(f_\star) \lambda_{ks}(f_\star) \delta_{ls}\\
=& \frac{\Gamma_{d, k}}{B(d, k)} \times \| \proj_k f_\star \|_{L^2}^2 \le \frac{2 C d}{ B(d, k)} \cdot \| \proj_k f_\star \|_{L^2}^2. 
\end{aligned}
\]
where the fourth equality used the fact that $\E_{\bx, \bx' }[  Y_{kl}(\bx)   Q_k(\< \bx, \bx'\>)   Y_{ks}(\bx')] = \delta_{ls} / B(d, k)$. 

Hence we have 
\begin{align*}
\cE_{\ge \ell} = \sum_{k= \ell }^\infty \cE_k \le \frac{2 d C}{B(d, \ell)} \cdot \| \proj_{\ge \ell} f_\star \|_{L^2}^2, 
\end{align*}
where we used the fact that $B(d, k)$ is non-decreasing in $k$ given by Lemma \ref{lem:non_decreasing_N}. This concludes the proof.

\subsection{Proof of Proposition \ref{prop:kernel_lower_bound_NTK}}

\subsubsection{Auxiliary lemmas}

In the proof of this proposition, we will need the following lemmas. 

\begin{lemma}\label{lem:gegenbauer_derivatives}
Let $\psi: \R \to \R$ be a function such that $\psi ( \< \be , \cdot \>) \in L^2 (\S^{d-1} (\sqrt{d}) )$ and $\psi ( \< \be , \cdot \>) \< \be, \cdot\> \in L^2 (\S^{d-1} (\sqrt{d}) )$. Let $\lbrace \lambda_{k,d} (\psi) \rbrace_{k = 0 }^\infty$ be the coefficients of its expansion in terms of the $d$-th order Gegenbauer polynomials
\[
\psi (x) = \sum_{k \ge 0} \lambda_{k,d} (\psi) B(d,k) Q_k^{(d)} (\sqrt{d} x), \qquad 
\lambda_{k,d} (\psi) = \E_{\bx \sim \Unif(\S^{d-1} (\sqrt{d}) ) } [\psi(x_1) Q^{(d)}_k (\sqrt{d} x_1 ) ].
\]
Then we can write 
\[
x \psi (x) = \sum_{k \ge 0} \lambda^{(1)}_{k,d} (\psi) B(d,k) Q_k^{(d)} (\sqrt{d} x),
\]
with the new coefficients given by
\[
\lambda^{(1)}_{0,d} (\psi)  = \sqrt{d}  \lambda_{1,d} (\psi) , \qquad \lambda^{(1)}_{k,d} (\psi)  = \sqrt{d}\frac{k+d-2}{2k+d-2} \lambda_{k+1,d} (\psi)  + \sqrt{d} \frac{k}{2k+d-2} \lambda_{k-1,d} (\psi) .
\]
\end{lemma}

\begin{proof}
We recall the following two formulas for $k\ge 1$ (see Section \ref{sec:Gegenbauer}):
\[
\begin{aligned}
\frac{x}{d} Q^{(d)}_{k} (x) & = \frac{k}{2k +d - 2} Q^{(d)}_{k-1} (x) + \frac{k+d-2}{2k+d-2} Q^{(d)}_{k+1} (x), \\
B(d,k) & = \frac{2k+d-2}{k} \binom{k+d-3}{k-1}. 
\end{aligned}
\]
Furthermore, we have $Q^{(d)}_0 (x) = 1$, $Q^{(d)}_1 (x)  = x/d$ and therefore therefore 
$x Q^{(d)}_0 (x) = d Q^{(d)}_1 (x)$. We insert these expressions in the expansion of the function $\psi$
\[
\begin{aligned}
x \psi (x) =&  \sum_{k \ge 0} \lambda_{k,d} (\psi) B(d,k) x Q_k^{(d)} (\sqrt{d} x  )\\
= & \lambda_{0,d} (\psi) B(d,0) \sqrt{d} Q^{(d)}_1 ( \sqrt{d} x) + \sum_{k = 0}^\infty \lambda_{k+1,d} (\psi) \frac{k+1}{2k+d} B(d,k+1) \sqrt{d} Q^{(d)}_k (\sqrt{d} x) \\
&+  \sum_{k = 2}^\infty \lambda_{k-1,d} (\psi) \frac{k+d-3}{2k+d-4} B(d,k-1) \sqrt{d} Q^{(d)}_k (\sqrt{d} x) \\
= & \sum_{k = 0}^\infty \lambda^{(1)}_{k,d} (\psi) B(d,k)  Q^{(d)}_k (x).
\end{aligned}
\]
Matching the coefficients of the expansion yields
\[
\begin{aligned}
 \lambda^{(1)}_{0,d} (\psi)  =&  \sqrt{d} \lambda_{1,d} (\psi)  \frac{1}{d} \frac{B(d,1)}{B(d,0) } =  \sqrt{d} \lambda_{1,d} (\psi) \\
 \lambda^{(1)}_{k,d} (\psi)  = & \sqrt{d} \frac{k+d-2}{2k+d-2}  \lambda_{k+1,d} (\psi)  + \sqrt{d} \frac{k}{2k+d-2}  \lambda_{k-1,d} (\psi) .
\end{aligned}
\]
\end{proof}

Similarly, we can write the decomposition of $x^2 \psi(x)$ to be
\[
x^2 \psi (x) = \sum_{k \ge 0} \lambda^{(2)}_{k,d} (\psi) B(d,k) Q_k^{(d)} (\sqrt{d} x ),
\]
where the coefficients are given by the same relation as in the above lemma
\[
\lambda^{(2)}_{0,d} (\psi)  = \sqrt{d} \lambda^{(1)}_{1,d} (\psi) , \qquad \lambda^{(2)}_{k,d} (\psi)  = \sqrt{d}\frac{k+d-2}{2k+d-2} \lambda^{(1)}_{k+1,d} (\psi)  + \sqrt{d} \frac{k}{2k+d-2} \lambda^{(1)}_{k-1,d} (\psi) .
\]

\begin{lemma}\label{lem:decomposition_kernel_NTK}
Let $\bu : \S^{d-1} (\sqrt{d} ) \times \S^{d-1} (\sqrt{d} )  \to \R^{d \times d}$ be a matrix-valued function defined by
\[
\bu (\btheta_1 , \btheta_2 ) = \E_{\bx}[\sigma'(\<\btheta_1, \bx\>/\sqrt{d}) \sigma'(\<\btheta_2, \bx\>/\sqrt{d}) \bx \bx^\sT].
\]
Then there exist functions $u_1,u_2,u_3 : [-1,1] \to \R$ such that
\[
\bu (\btheta_1 , \btheta_2 ) = u_1(\< \btheta_1, \btheta_2\>/d) \id_d + u_2(\< \btheta_1, \btheta_2\>/d) [\btheta_1 \btheta_2^\sT + \btheta_2 \btheta_1^\sT] + u_3(\< \btheta_1, \btheta_2\>/d)[\btheta_1 \btheta_1^\sT + \btheta_2 \btheta_2^\sT].
\]
\end{lemma}

\begin{proof}

\noindent
{\bf Case 1: $\btheta_1 \neq \btheta_2$.} 

We first consider the case $\btheta_1 \neq \btheta_2$. We will denote $\gamma = \< \btheta_1 , \btheta_2 \>/d < 1$ for convenience. Given any three functions $u_1 , u_2 , u_3: (-1, 1) \to \R$, we define
\[
\Tilde \bu ( \btheta_1 , \btheta_2 ) = u_1(\< \btheta_1, \btheta_2\>/d) \id_d + u_2(\< \btheta_1, \btheta_2\>/d) [\btheta_1 \btheta_2^\sT + \btheta_2 \btheta_1^\sT] + u_3(\< \btheta_1, \btheta_2\>/d)[\btheta_1 \btheta_1^\sT + \btheta_2 \btheta_2^\sT].
\]
Let us rotate $\bu$ and $\Tilde \bu$ such that $\btheta_1 = (\sqrt{d},0,\ldots , 0) $ and $\btheta_2 = (\gamma \sqrt{d} , \sqrt{1 - \gamma^2} \sqrt{d} , 0 , \ldots , 0 )$. We can rewrite
\[
\begin{aligned}
\bu (\btheta_1 , \btheta_2) = \begin{bmatrix} 
\bu_{1:2,1:2} & \bzero \\
\bzero & \E_{\bx}[\sigma'(x_1) \sigma'(\gamma x_1 + \sqrt{1 - \gamma^2} x_2 ) x_3^2] \id_{d-2}
\end{bmatrix},
\end{aligned}
\]
where 
\[
\bu_{1:2,1:2} = \begin{bmatrix} 
 \E_{\bx}[\sigma'(x_1) \sigma'(\gamma x_1 + \sqrt{1 - \gamma^2} x_2 ) x_1^2] & \E_{\bx}[\sigma'(x_1) \sigma'(\gamma x_1 + \sqrt{1 - \gamma^2} x_2 ) x_1 x_2] \\
\E_{\bx}[\sigma'(x_1) \sigma'(\gamma x_1 + \sqrt{1 - \gamma^2} x_2 ) x_1 x_2] & \E_{\bx}[\sigma'(x_1) \sigma'(\gamma x_1 + \sqrt{1 - \gamma^2} x_2 ) x_2^2] 
\end{bmatrix} .
\]
Similarly, we can write  
\[
\begin{aligned}
\Tilde \bu (\btheta_1 , \btheta_2) = \begin{bmatrix}
\Tilde \bu_{1:2,1:2} & \bzero \\
\bzero & u_1 (\gamma) \id_{d-2}
\end{bmatrix},
\end{aligned}
\]
where 
\[
\Tilde \bu_{1:2,1:2} = \begin{bmatrix} 
u_1(\gamma) +2\gamma d u_2(\gamma)  +(1+\gamma^2) d u_3(\gamma)  & \sqrt{1 - \gamma^2} d u_2(\gamma)  +\gamma \sqrt{1-\gamma^2} d u_3 (\gamma) \\
\sqrt{1 - \gamma^2} d u_2(\gamma)  +\gamma \sqrt{1-\gamma^2} d u_3 (\gamma) & u_1(\gamma)  +(1 - \gamma^2) d u_3(\gamma) 
\end{bmatrix} .
\]
We check in both cases that:
\[
\begin{aligned}
\bu_{11} - \bu_{22} - 2 \frac{\gamma}{ \sqrt{1- \gamma^2}} \bu_{12}  = &  \E_{\bx}[\sigma'(x_1) \sigma'(\gamma x_1 + \sqrt{1 - \gamma^2} x_2 ) \lbrace x_1^2 - (\gamma x_1 + \sqrt{1 - \gamma^2} x_2 )^2\rbrace ] /(1 - \gamma^2) = 0, \\
\Tilde \bu_{11} - \Tilde \bu_{22} - 2 \frac{\gamma}{ \sqrt{1- \gamma^2}} \Tilde \bu_{12}  = & 0.
\end{aligned}
\]
We conclude that $\bu$ and  $\Tilde \bu$ are equal if and only if
\[
\begin{aligned}
\Trace ( \bu (\btheta_1 , \btheta_2 ) ) = & \Trace ( \Tilde \bu (\btheta_1 , \btheta_2 ) ) =  d u_1 (\gamma) + 2 \gamma d u_2 (\gamma) + 2 d u_3 (\gamma ), \\
\< \btheta_1 , \bu (\btheta_1 , \btheta_2) \btheta_2 \> = &  \< \btheta_1 , \Tilde \bu (\btheta_1 , \btheta_2) \btheta_2 \>  = \gamma d u_1 (\gamma) + (1 + \gamma^2) d^2 u_2 (\gamma) + 2 d^2 \gamma u_3 (\gamma ),\\
\< \btheta_1 , \bu (\btheta_1 , \btheta_2)  \btheta_1 \> = & \< \btheta_1 , \Tilde \bu (\btheta_1 , \btheta_2)  \btheta_1 \>  =  d u_1 (\gamma) + 2 d^2 \gamma u_2 (\gamma) + d^2 (1 + \gamma^2)   u_3 (\gamma ).
\end{aligned}
\]
We can therefore choose for $\gamma < 1$
\begin{align}
\begin{bmatrix} 
u_1 ( \gamma)  \\
u_2 ( \gamma)  \\
u_3 ( \gamma)  
\end{bmatrix} = d^{-1} \begin{bmatrix} 
1 & 2\gamma & 2 \\
\gamma & d(1+ \gamma^2) & 2 d\gamma  \\
1 & 2d\gamma & d( 1 + \gamma^2 )
\end{bmatrix}^{-1} \times \begin{bmatrix} 
\Trace ( \bu (\btheta_1 , \btheta_2 ) ) \\
\< \btheta_1 , \bu (\btheta_1 , \btheta_2) \btheta_2 \> \\
\< \btheta_1 , \bu (\btheta_1 , \btheta_2)  \btheta_1 \> 
\end{bmatrix} . \label{eq:KernelFormulaeDiff}
\end{align}

\noindent
{\bf Case 2: $\btheta_1 = \btheta_2$.} 

Similarly, for some fixed $\alpha$ and $\beta$, we define 
\[
\Tilde \bu ( \btheta_1 , \btheta_1 ) = \alpha \id_d + \beta \btheta_1 \btheta_1^\sT.
\]
We can show that the matrices $\bu$ and $\Tilde \bu $ are equal if and only if
\begin{align}
\begin{bmatrix} 
\alpha \\
\beta
\end{bmatrix} = d^{-1} \begin{bmatrix} 
1 & 1 \\
1 & d  
\end{bmatrix}^{-1} \times \begin{bmatrix} 
\Trace ( \bu ( \btheta_1 , \btheta_1 ) ) \\
\< \btheta_1 ,\bu ( \btheta_1 , \btheta_1 ) \btheta_1 \> 
\end{bmatrix} . \label{eq:KernelFormulaeEqual}
\end{align}
We can therefore fix $u_1 (1 ) = \alpha$ and $u_2(1) + u_3(1) = \beta/2$.
\end{proof}

\begin{lemma}\label{lem:convergence_of_coefficients}
Let $\sigma$ be an activation function such that $\sigma (u) \leq c_0 \exp(c_1 u^2)$ for some constants $c_0,c_1$, with $c_1 <1$. Let the Hermite and Gegenbauer decompositions of $\sigma$ be
\begin{align}
\sigma(x) &= \sum_{k \ge 0} \frac{\mu_k(\sigma)}{k!}\, \bbHe_k(x),\\
\sigma(x) &= \sum_{k \ge 0} \lambda_{k, d}(\sigma) B(d, k) Q_k^{(d)}(\sqrt d x)\, .
\end{align}
Then we have for any fixed $k$,
\[
\lim_{d \to \infty} \lambda_{d, k}^2(\sigma) B(d, k) = \frac{\mu_k(\sigma)^2}{ k!}\, .
\]
\end{lemma}
\begin{proof}
Recall the correspondence \eqref{eq:Gegen-to-Hermite} between Gegenbauer and Hermite polynomials.
Note for any monomial $m_k(x) = x^k$, by Lemma \ref{lemma:square_integrable}.$(c)$, we have
\[
\begin{aligned}
&\lim_{d \to \infty} \vert \E_{\bx \sim \Unif(\S^{d-1}(\sqrt d))}[\sigma(x_1) m_k(x_1)] - \E_G[\sigma(G) m_k(G)] \vert\\
\le& \lim_{d \to \infty} \E[\sigma(G)^2]^{1/2} \E[(x_1^k - G^k)^2]^{1/2} + \E[(\sigma(G) - \sigma(x_1))^2]^{1/2} \E[x^{2k}]^{1/2} = 0. 
\end{aligned} 
\]
This gives for any fixed $k$, we have
\[
\begin{aligned}
&\lim_{d \to \infty} \lambda_{k, d}(\sigma) [B(d, k) k!]^{1/2} = \lim_{d \to \infty} \E_{\bx}[ \sigma(x_1) Q_k^{(d)}(\sqrt d x_1)] [B(d, k) k!]^{1/2} \\
=& \lim_{d \to \infty} \E_{\bx}[\sigma(x_1) \bbHe_k(x_1)] =\E_{G}[\sigma(G) \bbHe_k(G)] = \mu_k(\sigma).
\end{aligned}
\]
This proves the lemma. 
\end{proof}

\begin{lemma}\label{lem:gegenbauer_coefficients}
For any fixed $k$, let $Q_k^{(d)}(x)$ be the $k$-th Gegenbauer polynomial. We expand 
\[
Q_k^{(d)}(x) = \sum_{s = 0}^k p_{k, s}^{(d)} x^s. 
\]
Then we have 
\[
p_{k, s}^{(d)} = O_d(d^{-k/2 - s/2}). 
\]
\end{lemma}
\begin{proof}
Using the correspondence \eqref{eq:Gegen-to-Hermite} between Gegenbauer and Hermite polynomials
we have 
\[
[p_{k, s}^{(d)} \times (\sqrt{d})^s] \times B(d, k)^{1/2} = O_d(1).
\]
This gives
\[
p_{k, s}^{(d)} = O_d(1/ [ d^{s/2}\cdot B(d, k)^{1/2}]) = O_d(d^{- s/2 - k/2}). 
\]
This proves the lemma. 
\end{proof}

\begin{lemma}\label{lem:random_matrix_bound}
Let $N = o_d(d^{\ell+1})$ for a fixed integer $\ell$. Let $(\bw_i)_{i \in [N]} \sim \Unif(\S^{d-1})$ independently. Denote a matrix $\bDelta^{(k)} = (\Delta_{ij}^{(k)})_{i, j \in [N]}$ with 
\[
\Delta_{ij}^{(k)} = \begin{cases}
0, & ~~~~ i = j,\\
\< \bw_i, \bw_j\>^k, &~~~~ i \neq j.
\end{cases}
\]
Then as $d \to \infty$, we have
\[
\| \bDelta^{(k)} \|_{\op} = o_{d,\P} ((\log d)^{k/2} d^{\ell+1 - k/2}). 
\]
\end{lemma}

\begin{proof}
Let us consider  $\bw \sim \Unif(\S^{d-1})$, and $w_1$ its first coordinate.
We have $ w_1$ which has density $f(x) = (C_d\sqrt{d})  (1-x^2)^{(d-3)/2}$ on $[-1,+1]$, cf. Eq.~\eqref{eq:taud-def}:
\[
\P ( | w_1 | > t ) =2 (C_d\sqrt{d}) \int_t^1 (1-x^2)^{(d-3)/2} \de x \le  2 (C_d\sqrt{d})  (1 - t^2)^{(d-3)/2} \le 2\sqrt{d} e^{- (d-3) t^2/2 } ,
\]
where the last inequality holds for all $d$ large enough, since $C_d\to(2\pi)^{-1/2}$ as $d\to\infty$.
Hence, we have:
\[
\P\Big( \max_{i \neq j} \vert \< \bw_i , \bw_j \> \vert > t \Big) \le N^2 \P ( \vert \< \be_1 , \bw \> \vert > t ) \le C \exp \lbrace  (2 \ell+3) \log (d) - (d-3) t^2/2 \rbrace.
\]
Taking $t = O ( \log (d)^{1/2} d^{-1/2}) $, we get
\[
\max_{i \neq j} \vert  \< \bw_i , \bw_j \> \vert = O_{d,\P} ( \log (d)^{1/2} d^{-1/2}) .
\]
Using the following bound:
\[
\| \bDelta^{(k)} \|_{\op} \le N \| \bDelta^{(k)} \|_{\max}  = N \Big(\max_{i \neq j} \vert \< \bw_i , \bw_j \> \vert \Big)^k =o_{d,\P} ( (\log d)^{k/2} d^{\ell+1 - k/2}),
\]
which concludes the proof.
\end{proof}

\subsubsection{Proof of Proposition \ref{prop:kernel_lower_bound_NTK}}

\noindent
{\bf Step 1. Construction of the activation function $\hat \sigma$.}

By  Assumption \ref{ass:NT-square_integrable} and Lemma \ref{lemma:square_integrable} (applied to $\sigma'$ instead of $\sigma$),
 we have $\sigma' ( \< \be , \cdot \> ) \in L^2 (\S^{d-1} (\sqrt{d}))$ and we consider its expansion in terms of Gegenbauer polynomials
(as always, expectation is taken with respect to $\bx \sim  \Unif(\S^{d-1} (\sqrt{d}))$ with $x_1=\<\bx,\be_1\>$):
\[
\sigma' (x ) = \sum_{k=0}^\infty \lambda_{k,d} (\sigma ' ) B(d,k) Q_k (\sqrt{d}x), \qquad
\lambda_{k,d} (\sigma ' ) = \E_{\bx} [ \sigma' (x_1) Q_k (\sqrt{d} x_1) ].
\]
Let  $k_2 > k_1 \ge 2 \ell + 7$ be two indices that satisfy the conditions of Assumption \ref{ass:NT-non_trivial_components}. Using the Gegenbauer coefficients of $\sigma'$, we define $\hat \sigma':[-d,d]\to\reals$ by
\begin{equation}\label{eqn:hat_sigma_definition}
\hat \sigma' (x ) = \sum_{k\neq k_1, k_2} \lambda_{k,d} (\sigma ' ) B(d,k) Q_k (\sqrt{d} x) + \sum_{t = 1,2} (1 -\delta_t) \lambda_{k_t,d} (\sigma ' ) B(d,k_t) Q_{k_t} (\sqrt{d} x),
\end{equation}
for some $\delta_1, \delta_2$ that we will fix later (with $|\delta_t|\le 1$).

\noindent
{\bf Step 2. The functions $\bu, \hat \bu$ and $\bar \bu$.}

Let $\bu$ and $\hat \bu$ be the matrix-valued functions associated respectively to $\sigma'$ and $\hat \sigma'$
\begin{align}
\bu ( \btheta_1 , \btheta_2 ) & = \E_{\bx}[\sigma'(\<\btheta_1, \bx\>/\sqrt{d}) \sigma'(\<\btheta_2, \bx\>/\sqrt{d}) \bx \bx^\sT]\, , \label{eq:bu} \\
\hat \bu ( \btheta_1 , \btheta_2 ) & = \E_{\bx}[\hat \sigma'(\<\btheta_1, \bx\>/\sqrt{d}) \hat \sigma'(\<\btheta_2, \bx\>/\sqrt{d}) \bx \bx^\sT]\, .\label{eq:hbu}
\end{align}
From Lemma \ref{lem:decomposition_kernel_NTK}, there exists functions $u_1, u_2, u_3$ and $\hat u_1, \hat u_2, \hat u_3$, such that 
\[
\begin{aligned}
\bu ( \btheta_1 , \btheta_2 ) = &  u_1(\< \btheta_1, \btheta_2\>/d) \id_d + u_2(\< \btheta_1, \btheta_2\>/d) [\btheta_1 \btheta_2^\sT + \btheta_2 \btheta_1^\sT] + u_3(\< \btheta_1, \btheta_2\>/d)[\btheta_1 \btheta_1^\sT + \btheta_2 \btheta_2^\sT], \\
\hat \bu ( \btheta_1 , \btheta_2 ) = &  \hat u_1(\< \btheta_1, \btheta_2\>/d) \id_d + \hat u_2(\< \btheta_1, \btheta_2\>/d) [\btheta_1 \btheta_2^\sT + \btheta_2 \btheta_1^\sT] + \hat u_3(\< \btheta_1, \btheta_2\>/d)[\btheta_1 \btheta_1^\sT + \btheta_2 \btheta_2^\sT]. 
\end{aligned}
\]
We define $\bar \bu = \bu - \hat \bu $. Then we can write
\begin{equation}\label{eqn:bar_u_expression}
\bar \bu ( \btheta_1 , \btheta_2 ) =  \bar u_1(\< \btheta_1, \btheta_2\>/d) \id_d + \bar u_2(\< \btheta_1, \btheta_2\>/d) [\btheta_1 \btheta_2^\sT +\btheta_2 \btheta_1^\sT] + \bar u_3(\< \btheta_1, \btheta_2\>/d)[\btheta_1 \btheta_1^\sT + \btheta_2 \btheta_2^\sT],
\end{equation}
where $\bar u_k = u_k - \hat u_k$ for $k =1,2,3$.

\noindent
{\bf Step 3. Construction of the kernel matrices.}

Let $\bU, \hat \bU , \bar \bU \in \R^{Nd \times Nd}$ with $i,j$-th block (for $i, j \in [N]$) given by
\begin{align}
\bU_{ij} & = \bu ( \btheta_i , \btheta_j )\, , \label{eq:bU}\\
\hat \bU_{ij} & = \hat \bu ( \btheta_i , \btheta_j )\, , \label{eq:hbU}\\
\bar \bU_{ij} & = \bar \bu ( \btheta_i , \btheta_j ) =  \bu ( \btheta_i , \btheta_j ) - \hat \bu ( \btheta_i , \btheta_j )\, . \label{eqn:bar_U_expression}
\end{align}
Note that we have $\bU = \hat \bU + \bar \bU$. By Eq. (\ref{eq:hbU}) and (\ref{eq:hbu}), it is easy to see that $\hat \bU \succeq 0$. Then we have $\bU \succeq \bar \bU$. In the following, we would like to lower bound matrix $\bar \bU$. 

We decompose $\bar \bU  $ as
\[
\bar \bU = \bD + \bDelta,
\]
where $\bD \in \R^{dN \times dN}$ is a block-diagonal matrix, with
\begin{equation}\label{eqn:expression_D_matrix}
\bD = \diag(\bar \bU_{11} , \ldots , \bar \bU_{NN}),
\end{equation}
and $\bDelta \in \R^{dN \times dN}$ is formed by blocks $\bDelta_{ij}\in\reals^{d\times d}$ for $i, j \in [n]$,
defined by
\begin{equation}\label{eqn:expression_Delta_matrix}
\bDelta_{ij} = \begin{cases}
0, & ~~~~ i = j,\\
\bar \bU_{ij}, &~~~~ i \neq j.
\end{cases}
\end{equation}
In the rest of the proof, we will prove that $ \| \bDelta \|_{\op} = o_{d,\P} (1)$ and for $\eps$ small enough $\bD \succeq \eps \id_{Nd}$ with high probability.

\noindent
{\bf Step 4. Prove that $ \| \bDelta \|_{\op} = o_{d,\P} (1)$.}

Denoting $\gamma_{ij} = \< \btheta_i , \btheta_j \>/d < 1$, we get, from Eq.~\eqref{eq:KernelFormulaeDiff},
\begin{align}
\begin{bmatrix} 
\bar u_1 ( \gamma_{ij})  \\
\bar u_2 ( \gamma_{ij})  \\
\bar u_3 ( \gamma_{ij})  
\end{bmatrix} = \begin{bmatrix} 
u_1 ( \gamma_{ij}) - \hat u_1 ( \gamma_{ij})  \\
u_2 ( \gamma_{ij}) - \hat u_2 ( \gamma_{ij})  \\
u_3 ( \gamma_{ij}) - \hat u_3 ( \gamma_{ij})  
\end{bmatrix} 
&= d^{-1} \bM_{ij} \begin{bmatrix} 
\Trace ( \bar \bU_{ij} ) \\
\<\btheta_i , \bar  \bU_{ij} \btheta_j \> \\
\< \btheta_i , \bar \bU_{ij} \btheta_i \> 
\end{bmatrix}\, ,
\label{eq:u_bar_coeff}\\
\bM_{ij} & \equiv  \begin{bmatrix} 
1 & 2\gamma_{ij} & 2 \\
\gamma_{ij} & d(1+ \gamma_{ij}^2) & 2 d\gamma_{ij}  \\
1 & 2d\gamma_{ij} & d( 1 + \gamma_{ij}^2 )
\end{bmatrix}^{-1}\, . \label{eq:u_bar_coeff-Mmatrix}
\end{align}
Using the notations of Lemma \ref{lem:gegenbauer_derivatives}, we get
\[
\begin{aligned}
\Trace ( \bU_{ij} ) = &  \E_{\bx}[\sigma'(\<\btheta_i, \bx\>/\sqrt{d}) \sigma'(\<\btheta_j, \bx\>/\sqrt{d}) \|\bx\|_2^2 ]  \\
=& \sum_{k=0}^\infty d \cdot \lambda_{k,d} (\sigma ' )^2 B(d,k) Q_k (\< \btheta_i , \btheta_j \>), \\
\< \btheta_i , \bU_{ij} \btheta_j \> = & \E_{\bx}[\sigma'(\<\btheta_i, \bx\> / \sqrt{d}) \<\btheta_i, \bx\> \sigma'(\<\btheta_j, \bx\> /\sqrt{d} ) \<\btheta_j, \bx\> ] \\
=& \sum_{k=0}^\infty  d\cdot \lambda_{k,d}^{(1)} (\sigma ' )^2 B(d,k) Q_k (\< \btheta_i , \btheta_j \>), \\
\< \btheta_i , \bU_{ij} \btheta_i \> = & \E_{\bx}[\sigma'(\<\btheta_i, \bx\>/\sqrt{d}) \<\btheta_i, \bx\>^2 \sigma'(\<\btheta_j, \bx\>/\sqrt{d} )  ] \\
=& \sum_{k=0}^\infty  d\cdot \lambda_{k,d}^{(2)} (\sigma ' ) \lambda_{k,d} (\sigma ' ) B(d,k) Q_k (\< \btheta_i , \btheta_j \>).
\end{aligned}
\]
We get similar expressions for $\hat \bU_{ij} $ with $\lambda_{k,d} (\sigma' ) $ replaced by $\lambda_{k,d} (\hat \sigma' ) $. Because we defined $\sigma'$ and $\hat \sigma'$ by only modifying the $k_1$-th and $k_2$-th coefficients, we get 
\begin{equation}
\begin{aligned}
\Trace ( \bar \bU_{ij} ) = \Trace ( \bU_{ij}  - \hat \bU_{ij} ) = & \sum_{t = 1, 2} d [\lambda_{k_t,d}(\sigma ')^2 - (1 - \delta_t)^2 \lambda_{k_t,d}(\sigma ')^2 ] B(d,k_t) Q_k (\gamma_{ij}).
\end{aligned}
\label{eq:trace_ubar}
\end{equation}
Recalling that $\lambda_{k,d}^{(1)}$ only depend on $\lambda_{k-1,d}$ and $\lambda_{k+1,d}$ (Lemma \ref{lem:gegenbauer_derivatives}), we get
\begin{equation}
\begin{aligned}
 \< \btheta_i ,  \bar \bU_{ij}  \btheta_j \>
= &  d \sum_{k \in \lbrace k_1 -1 , k_1 +1 \rbrace\cup \lbrace k_2 -1 , k_2 +1 \rbrace} [\lambda_{k,d}^{(1)} (\sigma ')^2 - \lambda_{k,d}^{(1)} (\hat \sigma ')^2 ] B(d,k) Q_k (\<\btheta_{i},\btheta_j\>),\\
 \< \btheta_i ,  \bar  \bU_{ij}  \btheta_j \>
= & d \sum_{k \in \lbrace k_1 -2, k_1 , k_1 +2 \rbrace\cup \lbrace k_2 -2 , k_2 , k_2 +2 \rbrace}[\lambda_{k,d}^{(2)} (\sigma ')\lambda_{k,d} (\sigma ') - \lambda_{k,d}^{(2)} (\hat \sigma ')\lambda_{k,d} (\hat \sigma ') ] B(d,k) Q_k (\<\btheta_{i},\btheta_j\>).
\end{aligned}
\label{eq:quadratic_ubar}
\end{equation}
By Assumption \ref{ass:NT-square_integrable} and the convergence in Lemma \ref{lem:convergence_of_coefficients}, for any fixed $k$,
\begin{equation}
\lim_{d \rightarrow \infty} \lambda_{k,d} (\sigma ' ) [B(d,k) k!]^{1/2} = \mu_k (\sigma ').
\label{eq:lambda_limit}
\end{equation}
Using the expression of $B(d,k)$ we get
\begin{equation}
\begin{aligned}
\lim_{d \rightarrow \infty} \lambda_{k,d}^{(1)} (\sigma ' ) [B(d,k) k!]^{1/2} & = \mu_{k+1} (\sigma') + k \mu_{k-1} (\sigma') = \mu_k( x \sigma ' ),\\
\label{eq:lambdaone_limit}
\lim_{d \rightarrow \infty} \lambda_{k,d}^{(2)} (\sigma ' ) [B(d,k) k!]^{1/2} & = \mu_{k+2} + (2k +1 ) \mu_{k} (\sigma') + k(k-1) \mu_{k-2} (\sigma') = \mu_k( x^2 \sigma ' ).
\end{aligned}
\end{equation}

From Lemma \ref{lem:gegenbauer_coefficients}, we recall that the coefficients of the $k$-th Gegenbauer polynomial $Q_k^{(d)} ( x ) = \sum_{s=0}^k p^{(d)}_{k,s} x^s$
satisfy
\begin{align}
p^{(d)}_{k, s} = O_d (d^{-k/2 -s/2})\, .
\end{align}
Furthermore, we have shown in Lemma \ref{lem:random_matrix_bound} that $\max_{i \neq j} |\< \btheta_i , \btheta_j \> | = O_{d,\P} (\sqrt{d \log d} )$. 
We deduce that
\begin{equation}
\max_{i \neq j} \vert Q_k^{(d)} ( \<\btheta_{i},\btheta_j\>) \vert = \Tilde O_{d,\P} ( d^{-k/2}).
\label{eq:sup_gegenbauer_polynomials}
\end{equation}
Plugging the estimates \eqref{eq:lambda_limit}, \eqref{eq:lambdaone_limit} and \eqref{eq:sup_gegenbauer_polynomials} into
Eqs.~\eqref{eq:trace_ubar} and \eqref{eq:quadratic_ubar}, we obtain that
\begin{equation}
\begin{aligned}
\max_{i \neq j} \Big\lbrace \big|\Trace (\bar \bU_{ij}  )\big|  , \;\big|\< \btheta_i ,  \bar \bU_{ij}  \btheta_j \>\big|   ,\;
\big|\< \btheta_i ,  \bar \bU_{ij} \btheta_i \> \big| \Big\rbrace = \Tilde O_{d,\P} ( d^{2 - k_1/2}).
\label{eq:boundUbar}
\end{aligned} 
\end{equation}
From Eq.~\eqref{eq:u_bar_coeff-Mmatrix}, using the fact that $\max_{i\neq j}|\gamma_{ij}| = O_{d,\P} (\sqrt{(\log d)/d})$ and Cramer's rule for matrix inversion, it 
is easy to see that 
\begin{align}\label{eqn:M_matrix_bound}
\max_{i\neq j}\max_{l,k \in [3]}\big|(\bM_{ij})_{lk}\big| = O_{d,\P} (1)\, .
\end{align}
We deduce from \eqref{eq:boundUbar} \eqref{eq:u_bar_coeff} and \eqref{eqn:M_matrix_bound} that
\begin{equation}\label{eqn:bound_bar_u}
\max_{i \neq j } \max_{k = 1,2,3} \vert \bar u_k (\gamma_{ij}) \vert = \Tilde O_{d,\P} ( d^{1-k_1/2}). 
\end{equation}
As a result, combining Eq. (\ref{eqn:bound_bar_u}) with Eq. (\ref{eqn:bar_U_expression}) and (\ref{eqn:bar_u_expression}), we get
\[
\begin{aligned}
 \max_{i\neq j}\| \bar \bU_{ij} \|_{F}^2 =& \max_{i\neq j} \| ( \bar u_1(\gamma_{ij} ) \id_d + \bar u_2(\gamma_{ij}) [\btheta_i \btheta_j^\sT +\btheta_j \btheta_i^\sT] + \bar u_3(\gamma_{ij})[\btheta_i \btheta_i^\sT + \btheta_j \btheta_j^\sT]) \|_F^2 \\
 \le & \max_{i \neq j } \max_{k = 1,2,3} \vert \bar u_k (\gamma_{ij}) \vert^2 \Big[ \max_{i\neq j} [ \| \id_d \|_F^2 + \| \btheta_i \btheta_j^\sT +\btheta_j \btheta_i^\sT \|_F^2 + \| \btheta_i \btheta_i^\sT + \btheta_j \btheta_j^\sT \|_F^2 ] \Big] = \Tilde O_{d,\P} (d^{4 - k_1}).
 \end{aligned}
\]
By the expression of $\bDelta$ given by \eqref{eqn:expression_Delta_matrix}, we conclude that
\[
\| \bDelta \|_{\op}^2 \le \| \bDelta \|^2_{F} = \sum_{i,j = 1, i \neq j}^N \| \bar \bU_{ij}   \|_{F}^2 = \Tilde o_{d,\P} ( d^{2\ell+6 - k_1 }) .
\]
Since $k_1 \ge 2\ell +7$, we deduce that $ \| \bDelta \|_{\op} = o_{d,\P} (1)$.

\noindent
{\bf Step 5. Proving that $\bD \succeq \eps \id_{Nd}$.}~

By Lemma \ref{lem:decomposition_kernel_NTK}, we can express $\bar \bU_{ii}$ by
\[
\bar \bU_{ii} = \alpha \id_d + \beta \btheta_i \btheta_i^\sT
\]
with $\alpha$, $\beta$ independent of $i$, and given by Eq.~\eqref{eq:KernelFormulaeEqual}, namely
\begin{equation}
\begin{bmatrix} 
\alpha \\
\beta
\end{bmatrix} = [d(d-1)]^{-1} \begin{bmatrix} 
d & - 1 \\
- 1 & 1  
\end{bmatrix} \times \begin{bmatrix} 
\Trace ( \bar \bU_{ii}) \\
\< \btheta_i , \bar \bU_{ii} \btheta_i \> 
\end{bmatrix} .
\label{eq:u_bar_cst}
\end{equation}
(Notice that $\Trace ( \bar \bU_{ii}) $ and $\< \btheta_i , \bar \bU_{ii}  \btheta_i \>$
 are independent of $i$ by construction, cf. Eqs.~\eqref{eq:bu}, \eqref{eq:hbu} and \eqref{eq:bU}, \eqref{eq:hbU}.) 
By the definition of $\bD$ given in Eq. (\ref{eqn:expression_D_matrix}), We deduce that:
\begin{align}
\lambda_{\min} ( \bD) &= \min ( \alpha , \alpha + \beta d ) = \min \left\{ \frac{1}{d-1}\Trace ( \bar \bU_{ii} )- \frac{1}{d(d-1)}\< \btheta_i , \bar \bU_{ii} \btheta_i \> ,  \frac{1}{d}\< \btheta_i , \bar \bU_{ii} \btheta_i \> \right\}.
\end{align}
We claim that, under the assumptions of Proposition \ref{prop:kernel_lower_bound_NTK}, and denoting $\bdelta=(\delta_1,\delta_2)$ (where $\delta_1, \delta_2$ first appears in the definition of $\hat \sigma$ in Eq. (\ref{eqn:hat_sigma_definition}), and till now $\delta_1, \delta_2$ are still not determined)
\begin{align}
\lim_{d\to\infty}\frac{1}{d-1}\Trace(\bar \bU_{ii}) & = F_1(\bdelta)\, ,\label{eq:ClaimF1}\\
\lim_{d\to\infty}\frac{1}{d}\< \btheta_i , \bar \bU_{ii} \btheta_i \>& = F_2(\bdelta)\, ,\label{eq:ClaimF2}
\end{align}
where $F_1(\bzero) = F_2(\bzero)=0$ and $\nabla F_1(\bzero), \nabla F_2(\bzero)\neq \bzero$, $\det (\nabla F_1(\bzero) ,\nabla F_2(\bzero)) \neq 0$. 
Before proving this claim, let us show that it allows to finish the proof of Proposition \ref{prop:kernel_lower_bound_NTK}.
Since $\det (\nabla F_1(\bzero) ,\nabla F_2(\bzero)) \neq 0$, there exists a unit-norm vector $\bv$, such that $\<\bv,\nabla F_1(\bzero)\>>0$, and $\<\bv,\nabla F_2(\bzero)\>>0$.
Now we choose $\delta_1, \delta_2$ (first appears in the definition of $\hat \sigma$ in Eq. (\ref{eqn:hat_sigma_definition})): we set $\bdelta= (\delta_1, \delta_2) = \delta_0\bv$ with some $\delta_0>0$ small enough. This yields $F_1(\bdelta)>0$, $F_2(\bdelta)>0$. Define $\eps=\min(F_1(\bdelta),F_2(\bdelta))/2$,
we have 
\begin{align}
\liminf_{d\to\infty}\lambda_{\min} ( \bD) \ge \min(F_1(\bdelta),F_2(\bdelta)) =2 \eps\, ,
\end{align}
and therefore, with high probability,
\begin{align}
\bU =  \hat \bU + \bD +\bDelta \succeq \bzero + \eps \id_{Nd} - \| \bDelta \|_{\op} \id_{Nd} \succeq \frac{\eps}{2}\id_{Nd} \, .
\end{align}

We are left with the task of proving that the limits in Eqs.~\eqref{eq:ClaimF1}, \eqref{eq:ClaimF2} exist, with the desired properties. 
Using Eqs.~\eqref{eq:trace_ubar} and  \eqref{eq:quadratic_ubar}, we get:
\begin{align}
\frac{1}{d-1}\Trace ( \bar \bU_{ii} ) & = \frac{d}{d-1}\sum_{t \in\{1,2\}} [ \lambda_{k_t,d} (\sigma ' )^2 - (1 - \delta_t)^2 \lambda_{k_t,d} (\sigma ' )^2 ] B(k_t,d), \\
\frac{1}{d}\< \btheta_i , \bar \bU_{ii} \btheta_i \>   & =\sum _{k \in \lbrace k_1 -1 , k_1 +1 \rbrace\cup \lbrace k_2 -1 , k_2 +1 \rbrace} 
[ \lambda_{k,d}^{(1)} (\sigma ' )^2 - \lambda_{k,d}^{(1)} (\hat \sigma ' )^2 ] B(k,d)\, .
\end{align}
Using Eq.~\eqref{eq:lambdaone_limit}, we get that the limits \eqref{eq:ClaimF1}, \eqref{eq:ClaimF2} exist. Further, letting $\mu_k \equiv \mu_k(\sigma')$,
we have
\begin{align}
F_1(\bdelta) & = \sum_{t \in\{ 1,2\}}\delta_t (2 - \delta_t) \frac{\mu_{k_t}^2 }{k_t !}\, ,
\end{align}
while, for $k_2\neq k_1+2$
\begin{align*}
F_2(\bdelta) & =\sum_{t \in\{ 1,2\}}\left\{\frac{1}{(k_t-1)!} \Big[(\mu_{k_t}+(k_t-1)\mu_{k_t-2})^2-((1-\delta_t)\mu_{k_t}+(k_t-1)\mu_{k_t-2})^2\Big]\right.\\
&\phantom{AAAAA}+\left. \frac{1}{(k_t+1)!}\Big[(\mu_{k_t+2}+(k_t+1)\mu_{k_t})^2-(\mu_{k_t+2}+(1-\delta_t)(k_t+1)\mu_{k_t})^2\Big]
\right\}\, ,
\end{align*}
while, for $k_2= k_1+2$
\begin{align*}
F_2(\bdelta)  = &\frac{1}{(k_1-1)!} \Big[(\mu_{k_1}+(k_1-1)\mu_{k_1-2})^2-((1-\delta_1)\mu_{k_1}+(k_1-1)\mu_{k_1-2})^2\Big]\\
&+\frac{1}{(k_1+1)!} \Big[(\mu_{k_1+2}+(k_1+1)\mu_{k_1})^2-((1-\delta_2)\mu_{k_1+2}+(1-\delta_1)(k_1+1)\mu_{k_1})^2\Big]\\
&+\frac{1}{(k_2+1)!}\Big[(\mu_{k_2+2}+(k_2+1)\mu_{k_2})^2-(\mu_{k_2+2}+(1-\delta_2)(k_2+1)\mu_{k_2})^2\Big]\, .
\end{align*}
It is easy to check $F_1(\bzero) = F_2(\bzero) =0$, and to compute the gradients,
using the identity $\mu_k( x^2 \sigma ' )=\mu_{k+2}(\sigma') + (2k +1 ) \mu_{k} (\sigma') + k(k-1) \mu_{k-2} (\sigma')$, we get
\begin{align}
\nabla F_1(\bzero)& =\left(\frac{2\mu_{k_1}(\sigma')^2}{k_1!}; \; \frac{2\mu_{k_2}(\sigma')^2}{k_2!}\right)\, ,\\
\nabla F_2(\bzero) & =\left(\frac{2\mu_{k_1}(\sigma')\mu_{k_1}(x^2\sigma')}{k_1!}; \; \frac{2\mu_{k_2}(\sigma')\mu_{k_2}(x^2\sigma')}{k_2!}\right)\, .
\end{align}
Under Assumption \ref{ass:NT-non_trivial_components}, we have $\nabla F_1(\bzero), \nabla F_2(\bzero)\neq \bzero$ and 
$\det (\nabla F_1(\bzero) ,\nabla F_2(\bzero)) \neq 0$ completing the proof.

\section{Proof of Theorem \ref{thm:NT_lower_upper_bound}.(b): \NT\, model upper bound}\label{sec:proof_NTK_upper}

The proof for the \NT\, model follows the same scheme as for the \RF\, case. However, several steps are 
technically more challenging. We will follow the same notations introduced in Section \ref{sec:OutlineRF}. In particular
$\E_{\bx}, \E_{\bw}, \E_{\btheta}$ will denote, respectively, expectation with respect to $\bx\sim\Unif(\S^{d-1}(\sqrt{d}))$,
$\bw\sim\Unif(\S^{d-1}(1))$, $\btheta\sim\Unif(\S^{d-1}(\sqrt{d}))$.

Let us assume that $\lbrace f_d \rbrace$ are polynomials of degree at most $\ell + 1$, i.e.\,$f_d = \proj_{\leq \ell + 1 } f_d$. 

Denote $\mathcal{L} = L^2 ( \S^{d-1} ( \sqrt{d}) \rightarrow \R)$ and $\mathcal{L}_d = L^2 ( \S^{d-1} ( \sqrt{d}) \rightarrow \R^d )$. We introduce the operator $\T : \cL \to \cL_d$, such that for any $g \in \cL$,
\[
\T g ( \btheta ) = \E_{\bx} [ \bx \sigma'  ( \< \btheta , \bx \> / \sqrt{d}  ) g ( \bx) ].
\]
It easy to check that the adjoint operator $\T^* : \cL_d \to \cL$ verifies for any $\bg \in \cL_d$,
\[
\T^* \bg ( \bx) = \E_{\btheta} [  \sigma'  ( \< \btheta , \bx \> / \sqrt{d}  ) \< \bx, \bg ( \btheta)\> ].
\]
We define the operator $\K : \cL_d \to \cL_d$ as $\K \equiv \T \T^*$. For $\bg \in \cL_d$, we can write 
\[
\K \bg ( \btheta ) = \E_{\btheta'} [ \bK ( \btheta , \btheta ' ) \bg ( \btheta' ) ],
\]
where 
\[
\bK ( \btheta , \btheta ' ) = \E_{\bx} [ \bx \bx^\sT \sigma'  ( \< \btheta , \bx \> / \sqrt{d}  )\sigma'  ( \< \btheta' , \bx \> / \sqrt{d}  )].
\]
Furthermore, we define $\bbH: \cL \to \cL$ as $\bbH \equiv \T^* \T$. For $g \in \cL$, we can write 
\[
\bbH g ( \bx ) = \E_{\bx'} [ H ( \bx , \bx ' ) g ( \bx' ) ],
\]
where 
\[
\begin{aligned}
H ( \bx , \bx ' )= & \E_{\btheta} [ \sigma'  ( \< \btheta , \bx \> / \sqrt{d}  )\sigma'  ( \< \btheta, \bx' \> / \sqrt{d}  )] \< \bx , \bx ' \>\\
=& \sum_{m\geq 0 } \Gamma_{d,m} Q^{(d)}_m ( \< \bx , \bx ' \>),
\end{aligned}
\]
and $\Gamma_{d,m}$ can be computed using the Gegenbauer recursion formula Eq. (\ref{eq:RecursionG}),
\[
\Gamma_{d,m} = d \cdot [ t_{d,m-1} \lambda_{d,m-1} ( \sigma ' )^2 B(d , m-1) + s_{d,m+1} \lambda_{d,m+1} ( \sigma ' )^2 B(d , m+1)],
\]
with 
\[
s_{d,m} = \frac{m}{2m+d-2}, \qquad t_{d,m} = \frac{m+d-2}{2m+d-2}.
\]
In particular, it is easy to check that 
\[
\bbH Y^{(d)}_{ku} ( \bx ) = \frac{\Gamma_{d,k}}{B(d,k)} Y^{(d)}_{ku} ( \bx ). 
\]

We consider the subspace of $\cL_d$ corresponding to $\T ( V_{d,\leq\ell+1} )$, the image of $ V_{d,\leq\ell+1}$ by operator $\T$. One can check that $\lbrace \T Y^{(d)}_{ku} \rbrace_{ 0 \leq k \leq \ell+1, 1 \leq u \leq B(d,k)}$ is an orthogonal basis of this subspace. Furthermore
\begin{equation}\label{eq:diag_vectK}
\K \T Y^{(d)}_{ku} = \T ( \T^* \T Y^{(d)}_{ku} ) = \frac{\Gamma_{d,k}}{B(d,k)} \T Y^{(d)}_{ku}.
\end{equation} 
Hence this basis diagonalizes $\K$. By Eq. (\ref{eqn:mu_lambda_relationship}), we have 
\begin{equation}\label{eqn:Gamma_limit}
\begin{aligned}
\lim_{d \to \infty}( \Gamma_{d,m} / d ) =&  \lim_{d \to \infty} [ t_{d,m-1}  \lambda_{d,m-1} ( \sigma ' )^2 B(d , m-1)] + \lim_{d \to \infty}  [ s_{d,m+1}  \lambda_{d, m+1} ( \sigma ' )^2 B(d , m+1)]\\
=& \mu_{m-1}(\sigma')^2 / (m-1)! =  \mu_{m}(\sigma)^2 / (m-1)!. 
\end{aligned}
\end{equation}
By Assumption \ref{ass:activation_lower_upper_NT_v2}.(b), we have $\Gamma_{d,k} \neq 0$ for any $k \leq \ell+1$ when $d$ is sufficiently large. Hence, the restricted inverse $\K^{-1} \vert_{\T ( V_{d,\leq\ell+1} )}$ is well defined for $d$ sufficiently large.

Consider $\hat{f}_\NT (\bx ; \bTheta, \ba ) = \sum_{i = 1}^N \< \ba_i , \bx \> \sigma' (\<\btheta_i , \bx \>)$. We can expand the risk at parameter $\ba$ as
\[
\begin{aligned}
\E_\bx [ (f (\bx) - \hat{f}_\NT (\bx) )^2 ] = & \| f_d \|^2_{L^2} -2 \sum_{i=1}^N \E_{\bx} [ \<\ba_i , \bx \> \sigma ' ( \< \btheta_i , \bx \> / \sqrt{d}  ) f (\bx) ] \\
& +  \sum_{i,j=1}^N \E_{\bx} [ \<\ba_i , \bx \> \<\ba_j , \bx \> \sigma ' ( \< \btheta_i , \bx \> / \sqrt{d}  ) \sigma ' ( \< \btheta_j , \bx \> / \sqrt{d}  ) ].
\end{aligned}
\]
Let us define $\balpha ( \btheta) \equiv \K^{-1} \T f_d (\btheta)$ and choose $\ba_i = N^{-1} \balpha ( \btheta_i)$. We consider the expectation over $\bTheta$ of the \NT\,risk:
\[
\begin{aligned}
\E_{\bTheta} [ R_{\NT}(f_d , \bTheta / \sqrt d) ] \leq & \E_{\bTheta, \bx} \Big[(f_d (\bx) - \hat f_\NT (\bx; \bTheta, \ba ))^2 \Big] \Big\vert_{\ba_i = \balpha(\btheta_i) / N}\\
=& \| f_d \|^2_{L^2} - 2 \< \K^{-1} \T f_d ,\T f_d \>_{L^2}+ \<\K^{-1} \T f_d , \K [\K^{-1} \T f_d]\> \\
&+ \frac{1}{N} \big[  \E_{\btheta} [\<  \K^{-1} \T f_d ( \btheta ), \bK(\btheta,\btheta)  \K^{-1} \T f_d ( \btheta )\>] - \<\K^{-1} \T f_d , \K [\K^{-1} \T f_d]\>\big].
\end{aligned}
\]
It is easy to check that $\T^* \K^{-1} \T \vert_{V_{d,\leq\ell+1}} = \id_{V_{d,\leq\ell+1}}$. By Lemma \ref{lem:decomposition_kernel_NTK}, we have $\bK ( \btheta , \btheta ) = \alpha_d \id + \beta_d \btheta \btheta^\sT$ with 
\[
\begin{bmatrix} 
\alpha_d \\
\beta_d
\end{bmatrix} = [(d-1)]^{-1} \begin{bmatrix} 
d & - 1 \\
- 1 & 1  
\end{bmatrix} \times \begin{bmatrix} 
\| \sigma ' ( \< \be , \cdot \> ) \|_{L^2(\S^{d-1})}^2 \\
\|  \< \be , \cdot \> \sigma ' ( \< \be , \cdot \> ) \|_{L^2(\S^{d-1})}^2
\end{bmatrix} .
\]
From Assumption \ref{ass:activation_lower_upper_NT_v2}.(a) and Lemma \ref{lemma:square_integrable}.(b) applied to $\sigma' ( \< \be , \cdot \> )$ and $\<\be , \cdot \> \sigma ' ( \< \be , \cdot \> )$, we get $\alpha_d = O_d(1)$ and $\beta_d = O_d (d^{-1})$. We deduce that the operator norm verifies $\| \bK ( \btheta , \btheta ) \|_{\op} = \alpha_d + \beta_d \| \btheta \|_2^2 =  O_d (1)$.

Hence, there exists a constant $C>0$ such that 
\[
\E_{\bTheta} [ R_{\NT}(f_d , \bTheta / \sqrt d) ] \leq  \frac{C}{N } \| \K^{-1} \T f_d \|^2_{L^2}.
\]
Using the decomposition of $f_d$ in terms of harmonic polynomials (note we assumed $f_d$ is a degree $\ell+1$ polynomial) and Eq. \eqref{eq:diag_vectK}, we have
\[
\<  \K^{-1} \T f_d , \K^{-1} \T f_d \> = \sum_{k=0}^{\ell + 1} \frac{B(d,k)}{\Gamma_{d,k}} \| \proj_k f_d \|_{L^2}^2.
\]
By Eq. (\ref{eqn:Gamma_limit}), for any fixed $k \leq \ell+1$, we have $\Gamma_{d,k} = \Omega_d (d)$. Hence we get
\[
\E_{\bTheta} [ R_{\NT}(f_d , \bTheta / \sqrt d) ] \leq O_d \Big( \frac{d^{\ell +1}}{d \cdot N} \Big) \cdot \| f_d \|_{L^2}^2.
\]
Hence, from the assumption that $N = \omega_d ( d^\ell )$, we deduce that $R_{\NT}(f_d , \bTheta / \sqrt d)/\| f_d \|_{L^2}^2$ converges in $L^1$ to $0$, and therefore in probability.

\section{Proof of Theorem \ref{thm:upper_bound_KRR}: risk for \KR}\label{sec:proof_KR}

\subsection{Proof of Theorem \ref{thm:upper_bound_KRR}}

\noindent
{\bf Step 1. Rewrite the $\by$, $\bE$, $\bH$, $\bM$ matrices. }

The test error of empirical kernel ridge regression gives
\[
\begin{aligned}
R_{\KR}(f_d, \bX, \lambda) \equiv& \E_\bx\Big[ \Big( f_d(\bx) - \by^\sT (\bH + \lambda \id_n)^{-1} \bh(\bx) \Big)^2 \Big]\\
=& \E_\bx [ f_d(\bx)^2] - 2 \by^\sT (\bH + \lambda \id_n)^{-1} \bE + \by^\sT (\bH + \lambda \id_n)^{-1} \bM (\bH + \lambda \id_n)^{-1} \by,
\end{aligned}
\]
where $\bE = (E_1, \ldots, E_n)^\sT$, $\bM = (M_{ij})_{ij \in [n]}$ and $\bH = (H_{ij})_{ij \in [n]}$ with
\[
\begin{aligned}
E_i =& \E_\bx[f_d(\bx) h_d(\< \bx, \bx_i\>/d)], \\
M_{ij} =& \E_{\bx}[h_d(\< \bx_i, \bx\>/d) h_d(\< \bx_j, \bx\>/d)], \\
H_{ij} = & h_d ( \< \bx_i , \bx_j \> /d ).
\end{aligned}
\]

Let $B = \sum_{k =0}^\ell B(d, k)$. Define
\[
\begin{aligned}
\bD_k =&  \xi_k(h_d) \id_{B(d, k)}, \\
\bY_k =& (Y_{kl}(\bx_i))_{i \in [n], l \in [B(d, k)]} \in \R^{n \times B(d, k)},\\
\blambda_k =& ( \lambda_{k1}(f_d), \ldots, \lambda_{k B(d, k)}(f_d))^\sT \in \R^{B(d, k)}, \\
\bD_{\le \ell} =& \diag(\xi_0(h_d) \id_{B(d, 0)}, \ldots, \xi_\ell(h_d) \id_{B(d, \ell)}) \in \R^{B \times B}\\
\bY_{\le \ell} =& (\bY_0, \ldots, \bY_\ell) \in \R^{n \times B}, \\
\blambda_{\le \ell} =& (\blambda_0^\sT, \ldots, \blambda_\ell^\sT)^\sT \in \R^{B}. 
\end{aligned}
\]
Let the spherical harmonics decomposition of $f_{d}$ be 
\[
f_{d}(\bx) = \sum_{k = 0}^\infty \sum_{l = 1}^{B(d, k)} \lambda_{kl}(f_d) Y_{kl}(\bx) , 
\]
and the Gegenbauer decomposition of $h_{d}$ be 
\[
h_d(\< \bx, \by\> / d) = \sum_{k=0}^\infty \xi_k(h_d) B(d, k) Q_k^{(d)}(\< \bx, \by\>).
\]

We decompose the vectors and matrices $\boldf$, $\bE$, $\bH$, and $\bM$ in terms of spherical harmonics
\[
\begin{aligned}
\boldf =& \bY_{\le \ell} \blambda_{\le \ell} + \sum_{k = \ell + 1}^\infty \bY_k \blambda_k, \\
\bE =&\bY_{\le \ell} \bD_{\le \ell} \blambda_{\le \ell} + \sum_{k = \ell + 1}^\infty \bY_k \bD_k \blambda_k ,\\
\bH =& \bY_{\le \ell} \bD_{\le \ell} \bY_{\le \ell}^\sT + \sum_{k=\ell + 1}^\infty \bY_k \bD_k \bY_k^\sT, \\
\bM =& \bY_{\le \ell} \bD_{\le \ell}^2 \bY_{\le \ell}^\sT + \sum_{k=\ell + 1}^\infty \bY_k \bD_k^2 \bY_k^\sT. 
\end{aligned}
\]
By Proposition \ref{prop:Delta_bound} and Eq.~\eqref{eqn:W_all_terms}, the kernel $\bH$ and $\bM$ can be rewritten as  
\[
\begin{aligned}
\bH =& \bY_{\le \ell} \bD_{\le \ell} \bY_{\le \ell}^\sT + \kappa_h (\id_n + \bDelta_h), \\
\bM =& \bY_{\le \ell} \bD_{\le \ell}^2 \bY_{\le \ell}^\sT + \kappa_u (\id_n + \bDelta_u), 
\end{aligned}
\]
where
\[
\begin{aligned}
\kappa_h =& \sum_{k \ge \ell + 1} \xi_k(h_d) B(d, k),\\
\kappa_u =& \sum_{k \ge \ell + 1} \xi_k(h_d)^2 B(d, k),
\end{aligned}
\]
and 
\[
\max\{ \| \bDelta_u \|_{\op}, \| \bDelta_h \|_{\op} \} = o_{d, \P}(1).
\]

\noindent
{\bf Step 2. Decompose the risk}

Recalling $\by = \boldf + \beps$, we decompose the risk as follows
\[
\begin{aligned}
R_{\KR}(f_d, \bX, \lambda) =& \| f_d \|_{L^2}^2  - 2 T_1 + T_2 +  T_3 - 2 T_4 + 2 T_5. 
\end{aligned}
\]
where
\[
\begin{aligned}
T_1 =& \boldf^\sT (\bH + \lambda \id_n)^{-1} \bE, \\
T_2 =&  \boldf^\sT (\bH + \lambda \id_n)^{-1} \bM (\bH + \lambda \id_n)^{-1} \boldf, \\
T_3 =& \beps^\sT (\bH + \lambda \id_n)^{-1} \bM (\bH + \lambda \id_n)^{-1} \beps,\\ 
T_4 =& \beps^\sT (\bH + \lambda \id_n)^{-1} \bE,\\
T_5 =& \beps^\sT (\bH + \lambda \id_n)^{-1} \bM (\bH + \lambda \id_n)^{-1} \boldf. 
\end{aligned}
\]
Further, we denote $\boldf_{\le \ell}$, $\boldf_{> \ell}$, $\bE_{\le \ell}$, and $\bE_{> \ell}$, 
\[
\begin{aligned}
\boldf_{\le \ell} =&  \bY_{\le \ell} \blambda_{\le \ell},& \bE_{\le \ell} =&  \bY_{\le \ell} \bD_{\le \ell} \blambda_{\le \ell}, \\
\boldf_{> \ell} =&  \sum_{k = \ell + 1}^\infty \bY_k \blambda_k,& \bE_{> \ell} =&  \sum_{k = \ell + 1}^\infty \bY_k \bD_k \blambda_k. 
\end{aligned}
\]

\noindent
{\bf Step 3. Term $T_2$}

Note we have 
\[
T_2 = T_{21}  + T_{22} + T_{23}, 
\]
where
\[
\begin{aligned}
T_{21} =&\boldf_{\le \ell}^\sT (\bH + \lambda \id_n)^{-1} \bM (\bH + \lambda \id_n)^{-1} \boldf_{\le \ell},\\
T_{22} =& 2 \boldf_{\le \ell}^\sT (\bH + \lambda \id_n)^{-1} \bM (\bH + \lambda \id_n)^{-1} \boldf_{> \ell},\\
T_{23} =&\boldf_{> \ell}^\sT (\bH + \lambda \id_n)^{-1} \bM (\bH + \lambda \id_n)^{-1} \boldf_{> \ell}.\\
\end{aligned}
\]
By Lemma \ref{lem:key_H_U_H_bound}, we have 
\begin{equation}\label{eqn:H_U_H_bound}
\| n (\bH + \lambda \id_n)^{-1} \bM (\bH + \lambda \id_n)^{-1} - \bY_{\le \ell} \bY_{\le \ell}^\sT / n\|_{\op} = o_{d, \P}(1),
\end{equation}
hence
\[
\begin{aligned}
T_{21} =& \blambda_{\le \ell} \bY_{\le \ell}^\sT (\bH + \lambda \id_n)^{-1} \bM (\bH + \lambda \id_n)^{-1} \bY_{\le \ell} \blambda_{\le \ell} \\
=& \blambda_{\le \ell}^\sT \bY_{\le \ell}^\sT \bY_{\le \ell} \bY_{\le \ell}^\sT \bY_{\le \ell} \blambda_{\le \ell} / n^2 + [\| \bY_{\le \ell} \blambda_{\le \ell} \|_2^2  / n] \cdot o_{d, \P}(1). 
\end{aligned}
\]
By Lemma \ref{lem:concentration_YY}, we have (with $\| \bDelta \|_2 = o_{d, \P}(1)$)
\[
\begin{aligned}
\blambda_{\le \ell}^\sT \bY_{\le \ell}^\sT \bY_{\le \ell} \bY_{\le \ell}^\sT\bY_{\le \ell} \blambda_{\le \ell} / n^2 =& \blambda_{\le \ell}^\sT(\id_B + \bDelta)^2 \blambda_{\le \ell}  = \| \blambda_{\le \ell} \|_2^2 (1 + o_{d, \P}(1)). 
\end{aligned}
\]
Moreover, we have 
\[
\| \bY_{\le \ell} \blambda_{\le \ell} \|_2^2 / n = \blambda_{\le \ell}^\sT(\id_B + \bDelta) \blambda_{\le \ell}  = \| \blambda_{\le \ell} \|_2^2 (1 + o_{d, \P}(1)). 
\]
As a result, we have
\begin{align}\label{eqn:term_R21}
T_{21} =& \| \blambda_{\le \ell} \|_{2}^2 (1 + o_{d, \P}(1)) = \| \proj_{\le \ell} f_d \|_{L^2}^2 (1 + o_{d, \P}(1)). 
\end{align}

By Eq. (\ref{eqn:H_U_H_bound}) again, we have 
\[
\begin{aligned}
T_{23} =& \Big(\sum_{k\ge \ell+1} \blambda_k^\sT \bY_k^\sT \Big) (\bH + \lambda \id_n)^{-1} \bM (\bH + \lambda \id_n)^{-1}\Big(\sum_{k\ge \ell+1} \bY_k \blambda_k \Big)\\
=&  \Big(\sum_{k\ge \ell+1} \blambda_k^\sT \bY_k^\sT \Big) \bY_{\leq \ell} \bY_{\leq \ell}^\sT \Big(\sum_{k\ge \ell+1} \bY_k \blambda_k \Big) / n^2 +  \Big[\Big\|\sum_{k\ge \ell+1} \bY_k \blambda_k \Big\|_2^2 / n\Big] \cdot o_{d, \P}(1).
\end{aligned}
\]
By Lemma \ref{lem:YYYY_expectation}, we have 
\[
\begin{aligned}
\E\Big[\Big(\sum_{k\ge \ell+1} \blambda_k^\sT \bY_k^\sT \Big) \bY_{\le \ell} \bY_{\le \ell}^\sT \Big(\sum_{k\ge \ell+1} \bY_k \blambda_k \Big) \Big] /n^2 =& \sum_{u, v \ge \ell+1} \blambda_u^\sT\{ \E [(  \bY_u^\sT  \bY_{\le \ell} \bY_{\le \ell}^\sT  \bY_v ) ] /n^2 \} \blambda_v \\
=& \frac{B}{n} \sum_{k = \ell + 1}^\infty \| \blambda_k \|_2^2. \\
\end{aligned}
\]
Moreover
\[
\E \Big[\Big\|\sum_{k\ge \ell+1} \bY_k \blambda_k \Big\|_2^2 / n\Big] = \sum_{k = \ell + 1}^\infty \| \blambda_k \|_2^2 = \| \proj_{> \ell} f_d \|_{L^2}^2. 
\]
This gives
\begin{align}\label{eqn:term_R23}
T_{23} = o_{d, \P}(1) \cdot \| \proj_{> \ell} f_d \|_{L^2}^2. 
\end{align}
Using Cauchy Schwarz inequality for $T_{22}$, we get
\begin{align}\label{eqn:term_R22}
T_{22} \le 2 (T_{21} T_{23})^{1/2} = o_{d, \P}( 1) \cdot \| \proj_{\le \ell} f_d \|_{L^2}  \| \proj_{> \ell} f_d \|_{L^2} .  
\end{align}
As a result, combining Eqs. (\ref{eqn:term_R21}), (\ref{eqn:term_R22}) and (\ref{eqn:term_R23}), we have 
\begin{equation}\label{eqn:KRR_term_T2}
T_2 =  \| \proj_{\le \ell} f_d \|_{L^2}^2 + o_{d, \P}(1) \cdot \| f_d \|_{L^2}^2. 
\end{equation}

\noindent
{\bf Step 4. Term $T_{1}$. }

Note we have 
\[
T_1 = T_{11} + T_{12} + T_{13}, 
\]
where
\[
\begin{aligned}
T_{11} =& \boldf_{\le \ell}^\sT (\bH + \lambda \id_n)^{-1} \bE_{\le \ell}, \\
T_{12} =& \boldf_{> \ell}^\sT (\bH + \lambda \id_n)^{-1} \bE_{\le \ell}, \\
T_{13} =& \boldf^\sT (\bH + \lambda \id_n)^{-1} \bE_{> \ell}. \\
\end{aligned}
\]
By Lemma \ref{lem:lem_for_error_bound_R11}, we have 
\[
\| \bY_{\le \ell}^\sT( \bH + \lambda \id_n)^{-1}\bY_{\le \ell} \bD_{\le \ell} - \id_B \|_{\op} = o_{d, \P}(1). 
\]
so that 
\begin{equation}\label{eqn:term_R11}
T_{11} = \blambda_{\le \ell}^\sT  \bY_{\le \ell}^\sT( \bH + \lambda \id_n)^{-1}\bY_{\le \ell} \bD_{\le \ell} \blambda_{\le \ell} = \| \blambda_{\le \ell} \|_2^2 (1 + o_{d, \P}(1)) = \| \proj_{\le \ell} f_d \|_2^2 (1 + o_{d, \P}(1)). 
\end{equation}

Using Cauchy Schwarz inequality for $T_{12}$, and by the expression of $\bM = \bY_{\le \ell} \bD_{\le \ell}^2 \bY_{\le \ell}^\sT + \kappa_u(\id_n + \bDelta_u)$ with $\| \bDelta_u \|_{\op} = o_{d, \P}(1)$, we get with high probability
\begin{equation}\label{eqn:term_R12}
\begin{aligned}
\vert T_{12} \vert =& \Big\vert \sum_{k=\ell + 1}^\infty \blambda_k^\sT \bY_k^\sT (\bH + \lambda \id_n)^{-1} \bY_{\le \ell} \bD_{\le \ell}  \blambda_{\le \ell} \Big\vert \\
\le& \Big\| \sum_{k=\ell + 1}^\infty \blambda_k^\sT \bY_k^\sT (\bH + \lambda \id_n)^{-1} \bY_{\le \ell} \bD_{\le \ell} \Big\|_2 \| \blambda_{\le \ell} \|_2 \\
=& \Big[ \Big( \sum_{k=\ell + 1}^\infty \blambda_k^\sT \bY_k^\sT \Big) (\bH + \lambda \id_n)^{-1} \bY_{\le \ell} \bD_{\le \ell}^2 \bY_{\le \ell}^\sT (\bH + \lambda \id_n)^{-1} \Big( \sum_{k=\ell + 1}^\infty \blambda_k^\sT \bY_k^\sT \Big)\Big]^{1/2} \| \blambda_{\le \ell} \|_2\\
\le & \Big[ \Big( \sum_{k=\ell + 1}^\infty \blambda_k^\sT \bY_k^\sT \Big) (\bH + \lambda \id_n)^{-1} \bM (\bH + \lambda \id_n)^{-1} \Big( \sum_{k=\ell + 1}^\infty \blambda_k^\sT \bY_k^\sT \Big) \Big]^{1/2} \| \blambda_{\le \ell} \|_2\\
=&  T_{23}^{1/2} \| \blambda_{\le \ell} \|_2 = o_{d, \P}(1) \cdot \| \proj_{\le \ell} f_d \|_{L^2} \| \proj_{> \ell} f_d \|_{L^2}.
\end{aligned}
\end{equation}

For term $T_{13}$, we have 
\[
\begin{aligned}
\vert T_{13} \vert =&  \vert \boldf^\sT (\bH + \lambda \id_n)^{-1} \bE_{> \ell}\vert \le  \| \boldf \|_2 \| (\bH + \lambda \id_n)^{-1} \|_{\op} \| \bE_{> \ell} \|_2. 
\end{aligned}
\]
Note we have $\E[\| \boldf \|_2^2] = n \| f_d \|_{L^2}^2$, and $\| (\bH + \lambda \id_n)^{-1} \|_{\op} \le 2/(\kappa_h + \lambda)$ with high probability, and 
\[
\E [ \| \bE_{> \ell} \|_2^2 ] = n \sum_{k = \ell + 1}^\infty \xi_k(h_d)^2 \| \proj_{k} f_d \|_{L^2}^2 \le n \Big[ \max_{k \ge \ell + 1} \xi_k(h_d)^2\Big] \| \proj_{> \ell} f_d \|_{L^2}^2. 
\]
As a result, we have
\begin{equation}\label{eqn:term_R13}
\begin{aligned}
\vert T_{13} \vert \le& O_d(1) \cdot  \| \proj_{> \ell} f_d \|_{L^2} \| f_d \|_{L^2} \Big[n^2 \max_{k \ge \ell + 1} \xi_k(h_d)^2 \Big]^{1/2} / (\kappa_h + \lambda) \\
=& O_d(1) \cdot \| \proj_{> \ell} f_d \|_{L^2} \| f_d \|_{L^2} \Big[ n \max_{k \ge \ell + 1} \xi_k(h_d) \Big]/ \Big( \sum_{k \ge \ell + 1} \xi_k(h_d) B(d, k) + \lambda \Big)\\
=& o_{d, \P}(1) \cdot \| \proj_{> \ell} f_d \|_{L^2} \| f_d \|_{L^2},
\end{aligned}
\end{equation}
where the last equality used the fact that $\omega_d(d^\ell \log d)\le n \le O_d(d^{\ell + 1 - \delta})$ and Assumption \ref{ass:activation_krr}. Combining Eqs. (\ref{eqn:term_R11}), (\ref{eqn:term_R12}) and (\ref{eqn:term_R13}), we get 
\begin{equation}\label{eqn:KRR_term_T1}
T_1 =  \| \proj_{\le \ell} f_d \|_{L^2}^2 + o_{d, \P}(1) \cdot \| f_d \|_{L^2}^2. 
\end{equation}

\noindent
{\bf Step 5. Terms $T_3, T_4$ and $T_5$. } 

By Lemma \ref{lem:key_H_U_H_bound} again, we have 
\[
\begin{aligned}
\E_\beps[T_3] / \tau^2 =& \tr((\bH + \lambda \id_n)^{-1} \bM (\bH + \lambda \id_n)^{-1}) = \tr(\bY_{\le \ell} \bY_{\le \ell}^\sT / n^2) + o_{d, \P}(1), 
\end{aligned}
\]
By Lemma \ref{lem:concentration_YY}, we have 
\[
\tr(\bY_{\le \ell} \bY_{\le \ell}^\sT / n^2) = \tr(\bY_{\le \ell}^\sT \bY_{\le \ell}) / n^2 = n B / n^2 + o_{d, \P}(1) = o_{d, \P}(1). 
\]
This gives
\begin{align}\label{eqn:term_varT3}
T_3 = o_{d, \P}(1) \cdot \tau^2. 
\end{align}

Let us consider $T_4$ term:
\[
\begin{aligned}
\E_{\beps} [T_4^2 ]/\tau^2 = & \E_{\beps} [ \beps^\sT (\bH + \lambda \id_n)^{-1} \bE \bE^\sT  (\bH + \lambda \id_n)^{-1} \beps ]/\tau^2 \\
= & \bE^\sT  (\bH + \lambda \id_n)^{-2 }  \bE.
\end{aligned}
\]
Notice that by Lemma \ref{lem:concentration_YY}, Lemma \ref{lem:key_H_U_H_bound} and the definition of $\bM$,  for any integer $L$:
\[
\begin{aligned}
\| \bD_{\leq L} \bY_{\leq L }^{\sT} (\bH + \lambda \id_n)^{-2 } \bY_{\leq L } \bD_{\leq L} \|_{\op} = & \| (\bH + \lambda \id_n)^{-1  }\bY_{\leq L } \bD_{\leq L}^2  \bY_{\leq L }^{\sT} (\bH + \lambda \id_n)^{-1  } \|_{\op} \\
\leq& \| (\bH + \lambda \id_n )^{-1} \bM ( \bH + \lambda \id_n )^{-1}  \|_{\op}.\\
\leq &  \| \bY_{\leq \ell} \bY_{\leq \ell}^{\sT} / n \|_{\op}/n + o_{d, \P} (1) \cdot / n \\
 = & o_{d,\P} (1)
\end{aligned}
\]
Hence, 
\[ 
\begin{aligned}
\bE^\sT  (\bH + \lambda \id_n)^{-2 }  \bE = & \lim_{L \to \infty} \bE^\sT_{\leq L}  (\bH + \lambda \id_n)^{-2 }  \bE_{\leq L} \\
= & \lim_{L \to \infty} \blambda^\sT_{\leq L}  [\bD_{\leq L} \bY_{\leq L }^{\sT} (\bH + \lambda \id_n)^{-2 } \bY_{\leq L } \bD_{\leq L}] \blambda_{\leq L} \\
\leq & \| (\bH + \lambda \id_n )^{-1} \bM ( \bH + \lambda \id_n )^{-1}  \|_{\op} \cdot \lim_{L \to \infty} \| \blambda_{\leq L} \|_2^2 \\
\leq & o_{d,\P}(1) \cdot \| f_d \|_{L^2}^2,
\end{aligned}
\]
which gives
\begin{align}\label{eqn:term_varT4}
T_4 = o_{d, \P}(1) \cdot \tau \| f_d \|_{L^2} = o_{d,\P} (1) \cdot ( \tau^2 +  \| f_d \|_{L^2}^2 ). 
\end{align}
We decompose $T_5$ using $\boldf = \boldf_{\leq \ell} + \boldf_{> \ell}$,
\[
T_5 = T_{51} + T_{52},
\]
where
\[
\begin{aligned}
T_{51} = & \beps^\sT (\bH + \lambda \id_n )^{-1} \bM ( \bH + \lambda \id_n )^{-1}  \boldf_{\leq \ell}, \\
T_{52} = & \beps^\sT (\bH + \lambda \id_n )^{-1} \bM ( \bH + \lambda \id_n )^{-1}  \boldf_{> \ell} .
\end{aligned}
\]
First notice that
\[
\begin{aligned}
\| \bM^{1/2} ( \bH + \lambda \id_n)^{-2} \bM^{1/2} \|_{\op} =& \| (\bH + \lambda \id_n)^{-1} \bM (\bH + \lambda \id_n)^{-1} \|_{\op} =  o_{d,\P} (1) .
\end{aligned}
\]
Then by Lemma \ref{lem:key_H_U_H_bound}, we get
\[
\begin{aligned}
\E_{\beps} [T_{51}^2 ]/\tau^2 = & \E_{\beps} [ \beps^\sT (\bH + \lambda \id_n)^{-1} \bM (\bH + \lambda \id_n)^{-1} \boldf_{\leq \ell} \boldf_{\leq \ell}^\sT (\bH + \lambda \id_n)^{-1} \bM (\bH + \lambda \id_n)^{-1} \beps ]/\tau^2 \\
= & \boldf^\sT_{\leq \ell} [ (\bH + \lambda \id_n)^{-1} \bM (\bH + \lambda \id_n)^{-1} ]^2 \boldf_{\leq \ell} \\
\leq &  \| \bM^{1/2} ( \bH + \lambda \id_n)^{-2} \bM^{1/2} \|_{\op} \| \bM^{1/2} ( \bH + \lambda \id_n)^{-1 } \boldf_{\leq \ell} \|_{2}^2 \\
= &  o_{d,\P}(1)  \cdot T_{21} \\
= &  o_{d,\P} (1)  \cdot \| \proj_{\leq \ell } f_d \|_{L^2}^2.
\end{aligned}
\]
Similarly, we get 
\[
\begin{aligned}
\E_{\beps} [T_{52}^2 ]/\tau^2   = &   o_{d,\P} (1) \cdot T_{23} 
= &  o_{d,\P}  (1) \cdot \| \proj_{> \ell } f_d \|_{L^2}^2.
\end{aligned}
\]
By Markov's inequality, we deduce that
\begin{align}\label{eqn:term_varT5}
T_5 = o_{d, \P}(1) \cdot \tau ( \| \proj_{\leq \ell} f_d \|_{L^2} + \| \proj_{> \ell} f_d \|_{L^2} )= o_{d,\P} (1) \cdot ( \tau^2 +  \| f_d \|_{L^2}^2 ). 
\end{align}

\noindent
{\bf Step 6. Finish the proof. }

Combining Eqs. (\ref{eqn:KRR_term_T1}), (\ref{eqn:KRR_term_T2}), (\ref{eqn:term_varT3}), (\ref{eqn:term_varT4}) and (\ref{eqn:term_varT5}), we have
\[
\begin{aligned}
R_{\KR}( f_d, \bX, \lambda) = \| f_d \|_{L^2}^2 - 2 T_{1} + T_{2} + T_3 - 2 T_4 + 2T_5 = \| \proj_{> \ell} f_d \|_{L^2}^2 + o_{d, \P}(1) \cdot (\| f_d \|_{L^2}^2 + \tau^2),
\end{aligned}
\]
which concludes the proof.

\subsection{Auxiliary results}

\begin{lemma}\label{lem:concentration_YY}
Let $\{ Y_{kl} \}_{k \in \N, l \in [B(d, k)]}$ be the collection of spherical harmonics on $L^2(\S^{d-1}(\sqrt d))$. Let $(\bx_i)_{i \in [n]} \sim_{iid} \Unif(\S^{d-1}(\sqrt d))$. Denote
\[
\begin{aligned}
\bY_k = (Y_{kl}(\bx_i))_{i \in [n], l \in [B(d, k)]} \in \R^{n \times B(d, k)}. 
\end{aligned}
\]
Denote $B = \sum_{k=0}^\ell B(d, k)$, and
\[
\bY = (\bY_0, \ldots,  \ldots, \bY_\ell) \in \R^{n \times B}. 
\]
Then as long as $n / (B\log B) \to \infty$ as $d \to \infty$, we have 
\[
\bY^\sT \bY / n  =  \id_B + \bDelta, 
\]
with $\bDelta \in \R^{B \times B}$ and $\E[\| \bDelta \|_{\op}] = o_{d}(1)$. 
\end{lemma}
\begin{proof}[Proof of Lemma \ref{lem:concentration_YY}. ] 

Let $\bPsi = \bY^\sT \bY/n \in \R^{B \times B}$. We can rewrite $\bPsi$ as
\[
\bPsi = \frac{1}{n} \sum_{i=1}^n \bsh_i \bsh_i^\sT, 
\]
where
\[
\bsh_i = [1, Y_{11}(\bx_i), \ldots, Y_{1 B(d, 1)}(\bx_i), Y_{21}(\bx_i), \ldots, Y_{\ell B(d, \ell)}(\bx_i)]^\sT \in \R^B. 
\]

We use matrix Bernstein inequality. Denote $\bX_i = \bsh_i \bsh_i - \id_B \in \R^{B \times B}$. Then we have $\E[\bX_i] = \bzero$, and 
\[
\| \bX_i \|_{\op} \le \| \bsh_i \|_2^2 + 1 = \sum_{k = 0}^\ell \sum_{l \in [B(d, k)]}Y_{k l}(\bx_i)^2 + 1 = \sum_{k = 0}^\ell B(d, k) Q_k ( \< \bx_i , \bx_i \>) + 1 = B + 1,
\]
where we use formula \eqref{eq:GegenbauerHarmonics} and the normalization $Q_{k} (d) = 1$. Denote $V = \| \sum_{i = 1}^n \E[\bX_i^2] \|_{\op}$. Then we have 
\[
V  = n \| \E[(\bsh_i \bsh_i^\sT - \id_B)^2] \|_{\op} = n \| \E[ \bsh_i \bsh_i^\sT \bsh_i \bsh_i^\sT - 2 \bsh_i \bsh_i^\sT + \id_B ] \|_{\op} = n \| (B-1) \id_B\|_{\op} = n (B-1), 
\]
where we used $\bsh_i^\sT \bsh_i  = \| \bsh_i \|_2^2 = B$ and $\E [  \bsh_i (\bx_i)  \bsh_i^\sT ( \bx_i) ] = ( \E [ Y_{kl} (\bx_i) Y_{rs} (\bx_i) ] )_{kl,rs} = \id_B$. 
As a result, we have for any $t > 0$, 
\[
\P( \| \bPsi - \id_B \|_{\op} \ge t ) \le B \exp\{ - n^2 t^2 / [2 n (B - 1)  + 2 (B + 1)n t/ 3] \} \le \exp\{ - (n/B) t^2 / [10  (1 + t)] + \log B \}. 
\]
Integrating the tail bound proves the lemma.
\end{proof}

\begin{lemma}\label{lem:YYYY_expectation}
Let $\{ Y_{kl} \}_{k \in \N, l \in [B(d, k)]}$ be the collection of spherical harmonics on $L^2(\S^{d-1}(\sqrt d))$. Let $(\bx_i)_{i \in [n]} \sim_{iid} \Unif(\S^{d-1}(\sqrt d))$. Denote 
\[
\begin{aligned}
\bY_k = (Y_{kl}(\bx_i))_{i \in [n], l \in [B(d, k)]} \in \R^{n \times B(d, k)}. 
\end{aligned}
\]
Then for $u, s, v \in \N$ and $u \neq v$, we have 
\[
\E[ \bY_u^\sT  \bY_s \bY_s^\sT  \bY_v ] = \bzero. 
\]
For $u, s \in \N$, we have 
\[
\E[ \bY_u^\sT  \bY_s \bY_s^\sT  \bY_u ] = [ B(d, s) n + n(n-1) \delta_{us} ] \id_{B(d, u)}. 
\]
\end{lemma}

\begin{proof}[Proof of Lemma \ref{lem:YYYY_expectation}]
We have
\begin{equation}\label{eqn:YYYY_expectation}
\begin{aligned}
\E[ \bY_u^\sT  \bY_s \bY_s^\sT  \bY_v] =& \sum_{i,j\in[n]} \sum_{m \in [B(d, s)]} (\E[ Y_{up}(\bx_i) \Big(Y_{s m}(\bx_i) Y_{sm}(\bx_j)\Big) Y_{vq}(\bx_j) ] )_{p \in [B(d, u)], q \in [B(d, v)]}\\
=& \sum_{i\in[n]} \Big(\E \Big[ Y_{up}(\bx_i)  \Big(\sum_{m \in [B(d, s)]} Y_{s m}(\bx_i) Y_{sm}(\bx_i)\Big) Y_{vq}(\bx_i) \Big] \Big)_{p \in [B(d, u)], q \in [B(d, v)]}\\
& + \sum_{i \neq j\in[n]} \sum_{m \in [B(d, s)]} (\E[ Y_{up}(\bx_i) Y_{s m}(\bx_i) Y_{sm}(\bx_j) Y_{vq}(\bx_j) ] )_{p \in [B(d, u)], q \in [B(d, v)]}\\
=& B(d, s)  \sum_{i\in[n]} (\E[ Y_{up}(\bx_i) Y_{vq}(\bx_i)] )_{p \in [B(d, u)], q \in [B(d, v)]} \\
& + \sum_{i \neq j\in[n]}  \sum_{m \in [B(d, s)]} ( \delta_{u s} \delta_{pm} \delta_{sv} \delta_{qm})_{p \in [B(d, u)], q \in [B(d, v)]}\\
=& ( B(d, s) n \delta_{uv} \delta_{pq} + n(n-1) \delta_{us} \delta_{sv} \delta_{pq} )_{p \in [B(d, u)], q \in [B(d, v)]}. 
\end{aligned}
\end{equation}
This proves the lemma. 
\end{proof}

\begin{lemma}\label{lem:key_H_U_H_bound}
Let $\{ h_d \}_{d \ge 1}$ be a sequence of functions satisfying Assumption \ref{ass:activation_krr}. Let $\omega_d(d^{\ell} \log d) \le n \le O_d(d^{\ell + 1 - \delta})$. We have 
\[
\| n (\bH + \lambda \id_n)^{-1} \bM (\bH + \lambda \id_n)^{-1} - \bY \bY^\sT / n \|_{\op} = o_{d, \P}(1). 
\]
\end{lemma}

\begin{proof}[Proof of Lemma \ref{lem:key_H_U_H_bound}]
Denote 
\begin{align}
\bY_k = (Y_{kl}(\bx_i))_{i \in [n], l \in [B(d, k)]} \in \R^{n \times B(d, k)}.
\end{align}
Denote $B = \sum_{k \le \ell} B(d, k)$, and
\[
\bY = [\bY_0, \ldots, \bY_\ell] \in \R^{n \times B},
\]
and 
\[
\bD = \diag(\xi_0(h_d) \id_{B(d, 0)}, \ldots, \xi_\ell(h_d) \id_{B(d, \ell)}) \in \R^{B \times B}. 
\]
Then we have 
\[
\begin{aligned}
& n (\bH + \lambda \id_n)^{-1} \bM (\bH + \lambda \id_n)^{-1} \\
=& n (\bY \bD \bY^\sT + (\kappa_h + \lambda) \id_n + \kappa_h \bDelta_h)^{-1} (\bY \bD^2 \bY^\sT + \kappa_u (\id_n + \bDelta_u) ) (\bY \bD \bY^\sT + (\kappa_h + \lambda) \id_n + \kappa_h \bDelta_h)^{-1} \\
=& T_1 + T_2, 
\end{aligned}
\]
where $\| \bDelta_u \|_{\op}, \| \bDelta_h \|_{\op} = o_{d, \P}(1)$, and 
\[
\begin{aligned}
T_1 =&  n   \kappa_u (\bY \bD \bY^\sT + (\kappa_h + \lambda) \id_n + \kappa_h \bDelta_h)^{-1}  (\id_n + \bDelta_u) (\bY \bD^2 \bY^\sT + (\kappa_h + \lambda) \id_n + \kappa_h \bDelta_h)^{-1}, \\
T_2 =& n (\bY \bD \bY^\sT + (\kappa_h + \lambda) \id_n + \kappa_h \bDelta_h)^{-1} \bY \bD^2 \bY^\sT  (\bY \bD \bY^\sT + (\kappa_h + \lambda) \id_n + \kappa_h \bDelta_h)^{-1}.
\end{aligned}
\] 
For $T_1$, we have with high probability (note $n = O_d(d^{\ell + 1 - \delta})$) 
\[
\begin{aligned}
\| T_1 \|_{\op} \le& n \kappa_u \| (\bY \bD \bY^\sT + (\kappa_h + \lambda) \id_n + \kappa_h \bDelta_h)^{-1} \|_{\op}^2 \| \id_n + \bDelta_u \|_{\op} \le 2 n \kappa_u / (\kappa_h + \lambda)^2 \\
\le& 2 n \Big[ \sum_{k \ge \ell + 1} \xi_k(h_d)^2 B(d, k) \Big] / \Big[ \sum_{k \ge \ell + 1} \xi_k(h_d) B(d, k) + \lambda \Big]^2 \\
\le& 2 n \sup_{k \ge \ell + 1} \xi_k(h_d) / \Big[ \sum_{k \ge \ell + 1} \xi_k(h_d) B(d, k) + \lambda \Big]= o_{d} (1). 
\end{aligned}
\]

To bound $T_2$, let $\bY = \sqrt n \bO \bS \bV^{\sT}$ where $\bO \in \R^{n \times n}$ and $\bV \in \R^{B \times B}$ are orthogonal matrices, and $\bS = [\bS_\star ; \bzero] \equiv [\id_B + \bDelta_s ; \bzero] \in \R^{n \times B}$. By Lemma \ref{lem:concentration_YY} and the fact that $n = \omega_d(d^\ell \log d)$, we have $\| \bDelta_s \|_{\op} = o_{d, \P}(1)$. Then we have 
\[
\begin{aligned}
T_2 =& (\bO \bS \bV^{\sT} \bD \bV \bS^\sT \bO^{\sT} + [(\kappa_h + \lambda)/ n ] \id_n + \kappa_h/ n \cdot \bDelta_h )^{-1} (\bO \bS \bV^{\sT} \bD^2 \bV \bS^\sT \bO^{\sT}) \\
& \times (\bO \bS \bV^{\sT} \bD \bV \bS^\sT \bO^{\sT} + [(\kappa_h + \lambda)/ n ] \id_n + \kappa_h/ n \cdot \bDelta_h )^{-1}\\
=& \bO(\bS \bV^{\sT} \bD \bV \bS^\sT  + [(\kappa_h + \lambda) / n ] \id_n + \kappa_h / n \cdot \bDelta_0 )^{-1} ( \bS \bV^{\sT} \bD^2 \bV \bS^\sT) \\
& \times (\bS \bV^{\sT} \bD \bV \bS^\sT  + [(\kappa_h + \lambda) / n ] \id_n + \kappa_h / n \cdot \bDelta_0 )^{-1} \bO^\sT, \\
=&\bO (\diag(\bS_\star \bV^{\sT} \bD \bV \bS_\star/[(\kappa_h + \lambda) / n ], \bzero_{n - B}) +  \id_n + \kappa_h / (\kappa_h + \lambda) \cdot \bDelta_0 )^{-1} \\
& \times \diag(\bS_\star \bV^{\sT} \bD^2 \bV \bS_\star / [(\kappa_h + \lambda) / n ]^2, \bzero_{n - B}) \\
&\times(\diag(\bS_\star \bV^{\sT} \bD \bV \bS_\star/[(\kappa_h + \lambda) / n ], \bzero_{n - B}) + \id_n + \kappa_h / (\kappa_h + \lambda) \cdot \bDelta_0 )^{-1} \bO^\sT,
\end{aligned}
\]
where $\bDelta_0 = \bO^{\sT} \bDelta_h \bO$ and $\| \bDelta_0 \|_{\op} = o_{d, \P}(1)$. 

For a symmetric matrix $\bS_0 \in \R^{B \times B}$ and a symmetric matrix $\bA = [\bA_{11}, \bA_{12}; \bA_{21}, \bA_{22}] \in \R^{n \times n}$, we have 
\[
\begin{aligned}
&([\bA_{11}, \bA_{12}; \bA_{21}, \bA_{22}])^{-1} \diag(\bS_0, \bzero_{n - B}) ([\bA_{11}, \bA_{12}; \bA_{21}, \bA_{22}])^{-1} \\
=& [\bB_{11} \bS_0 \bB_{11}, \bB_{11} \bS_0 \bB_{12}; \bB_{21} \bS_0 \bB_{11}, \bB_{21} \bS_0 \bB_{12}] \\
=&[\id_B; - \bA_{22}^{-1} \bA_{21}] \bB_{11} \bS_0 \bB_{11}   [\id_B, - \bA_{12} \bA_{22}^{-1}]. 
\end{aligned}
\]
where
\[
\begin{aligned}
\bB_{11} =& (\bA_{11} - \bA_{12} \bA_{22}^{-1} \bA_{21})^{-1}, \\
\bB_{21} =& \bB_{12}^\sT = - \bA_{22}^{-1} \bA_{21} \bB_{11}. \\
\end{aligned}
\]
Taking 
\[
\begin{aligned}
\bS_0 =& \bS_\star \bV^{\sT} \bD^2 \bV \bS_\star/[(\kappa_h + \lambda) / n ]^2,\\
\bA_{11} =& \bS_\star \bV^{\sT} \bD \bV \bS_\star / [(\kappa_h + \lambda) / n ]+ \id_{B} + [\kappa_h/(\kappa_h + \lambda)] \bDelta_{0, 11}, \\
\bA_{12} =&  [\kappa_h/(\kappa_h + \lambda)] \bDelta_{0, 12}, \\
\bA_{22} = & \id_{n - B} + [\kappa_h/(\kappa_h + \lambda)] \bDelta_{0, 22}, 
\end{aligned}
\]
with $\| \bDelta_{0,11} \|_{\op} , \| \bDelta_{0,12} \|_{\op} , \| \bDelta_{0,22} \|_{\op} = o_{d,\P} (1)$.
This gives
\[
T_2 = \bO \Big\{ [\id_B; \bDelta_1]\bB_{11}\bS_0 \bB_{11} [\id_B, \bDelta_1^\sT] \Big\} \bO^\sT,
\]
with $\bDelta_1 = - (\id_{n - B} + [\kappa_h/(\kappa_h + \lambda)] \bDelta_{0, 22} )^{-1} [\kappa_h/(\kappa_h + \lambda)] \bDelta_{0, 12}$ and $\| \bDelta_1 \|_{\op} = o_{d, \P}(1)$.

Now we look at $\bB_{11}\bS_0 \bB_{11}$. We have 
\[
\begin{aligned}
\bB_{11} =& (\bS_\star \bV^{\sT} \bD \bV \bS_\star / [(\kappa_h + \lambda) / n ] + \id_{B} + [\kappa_h/(\kappa_h + \lambda)] \bDelta_{0, 11} \\
&- [\kappa_h/(\kappa_h + \lambda)]^2 \bDelta_{0, 12} [\id_{n - B} + [\kappa_h/(\kappa_h + \lambda)] \bDelta_{0, 22}]^{-1} \bDelta_{0, 21})^{-1}  \\
=& (\bS_\star \bV^{\sT} \bD \bV \bS_\star / [(\kappa_h + \lambda) / n ] + \id_{B} + \bDelta_2)^{-1} , \\
=& (\bS_\star \bV^\sT \{ \bD  / [(\kappa_h + \lambda) / n ] + \id_{B} + \bDelta_3 \} \bV \bS_\star)^{-1}, 
\end{aligned}
\]
where
\[
\bDelta_2 = [\kappa_h/(\kappa_h + \lambda)] \bDelta_{0, 11} -[\kappa_h/(\kappa_h + \lambda)]^2 \bDelta_{0, 12} [\id_{n - B} + [\kappa_h/(\kappa_h + \lambda)] \bDelta_{0, 22}]^{-1} \bDelta_{0, 21} \in \R^{B \times B}, 
\]
and 
\[
\bDelta_3 = \bV [\bS_\star^{-1}(\id_B + \bDelta_2) \bS_\star^{-1} - \id_B] \bV^\sT \in \R^{B \times B},
\]
and $\| \bDelta_2 \|_{\op}, \| \bDelta_3 \|_{\op} = o_{d, \P}(1)$. Define 
\[
\bD_1 = \bD  / [(\kappa_h + \lambda) / n ]. 
\]
We have
\[
\begin{aligned}
\bB_{11} \bS_0 \bB_{11} =& (\bS_\star \bV^\sT \{ \bD_1 + \id_{B} + \bDelta_3 \} \bV \bS_\star)^{-1} \bS_\star \bV^{\sT} \bD_1^2 \bV \bS_\star (\bS_\star \bV^\sT \{ \bD_1 + \id_{B} + \bDelta_3 \} \bV \bS_\star)^{-1}\\
=& \bS_\star^{-1} \bV^\sT (\bD_1 + \id_{B} + \bDelta_3)^{-1} \bD_1^2 (\bD_1 + \id_B + \bDelta_3)^{-1} \bV \bS_\star^{-1} \\
=& \bS_\star^{-1} \bV^\sT (\id_B + \bD_1^{-1} + \bD_1^{-1}\bDelta_3)^{-1} (\id_B + \bD_1^{-1} + \bDelta_3\bD_1^{-1})^{-1} \bV \bS_\star^{-1}.  
\end{aligned}
\]
Note we have 
\[
\lambda_{\min}(\bD_1) = \min_{k \le \ell} [ n \xi_k(h_d)] / \Big[ \sum_{k \ge \ell + 1} \xi_k(h_d) B(d, k)  + \lambda \Big], 
\]
and by Assumption \ref{ass:activation_krr} and $n = \omega_d(d^\ell \log d)$ we have $\lambda_{\min} (\bD_1) = \omega_d(1)$. As long with the fact that $\| \bS_\star - \id_B \|_{\op} = o_{d, \P}(1)$, we have 
\[
\| \bB_{11} \bS_0 \bB_{11} - \id_B \|_{\op} = o_{d, \P}(1). 
\]
As a result, we have 
\[
T_2 = \bO [\diag(\id_B, \bzero_{n - B}) + \bDelta_t ] \bO^\sT,
\]
with $\| \bDelta_t \|_{\op} = o_{d, \P}(1)$. Finally, we have 
\[
\bY \bY^\sT / n = \bO \diag(\bS_\star^2, \bzero_{n - B}) \bO^\sT = \bO [\diag(\id_B, \bzero_{n - B}) + \bDelta_y ]\bO^\sT,
\]
with $\| \bDelta_y \|_{\op} = o_{d, \P}(1)$. This proves the proposition. 
\end{proof}

\begin{lemma}\label{lem:lem_for_error_bound_R11}
Let $\{ h_d \}_{d \ge 1}$ be a sequence of functions satisfying Assumption \ref{ass:activation_krr}. Let $\omega_d(d^{\ell} \log d) \le n \le O_d(d^{\ell + 1 - \delta})$. We have
\[
\| \bY^\sT( \bH + \lambda \id_n)^{-1}\bY \bD - \id_B \|_{\op} = o_{d, \P}(1). 
\]
\end{lemma}

\begin{proof}[Proof of Lemma \ref{lem:lem_for_error_bound_R11}]~

By Proposition \ref{prop:Delta_bound}, we have $\bH + \lambda \id_n = \bY \bD \bY^\sT + (\kappa_h + \lambda) \id_n + \kappa_h \bDelta_h$ with $\| \bDelta_h \|_{\op} = o_{d, \P}(1)$. Denote the singular value decomposition $\bY = \sqrt n \bO \bS \bV^\sT$, with $\bO \in \R^{n \times n}$, $\bV \in \R^{B \times B}$ be two orthogonal matrices, and $\bS = [\bS_\star ; \bzero] = [\id_n + \bDelta_s ; \bzero] \in \R^{n \times B}$, with $\| \bDelta_s \|_{\op} = o_{d, \P}(1)$ (Lemma \ref{lem:concentration_YY}). Then we have
\[
\begin{aligned}
&\bY^\sT( \bH + \lambda \id_n)^{-1}\bY \bD - \id_B \\
=& \bV \bS^\sT \bO^\sT (\bO \bS \bV^\sT \bD \bV \bS^\sT \bO^\sT + [(\kappa_h + \lambda) / n] (\id_n + \bDelta_1))^{-1} \bO \bS \bV^\sT \bD - \id_B \\
=& \bV \bS^\sT ( \bS \bV^\sT \bD \bV \bS^\sT  +  [(\kappa_h + \lambda)  / n ] (\id_n + \bDelta_2))^{-1} \bS \bV^\sT \bD - \id_B \\
=& \bV [\bS_\star, \bzero] ( \diag(\bS_\star \bV^\sT \bD \bV \bS_\star / [(\kappa_h + \lambda) / n], \bzero_{n - B})  + \id_n + \bDelta_2)^{-1} [\bS_\star, \bzero]^\sT \bV^\sT \bD  / [(\kappa_h + \lambda) / n] - \id_B \\
=& \bV \bS_\star (\bS_\star \bV^\sT \bD \bV \bS_\star / [(\kappa_h + \lambda) / n] + \bLambda_3)^{-1} \bS_\star \bV^\sT \bD  / [(\kappa_h + \lambda) / n] - \id_B \\
=& \bV ( \bV^\sT \bD \bV / [(\kappa_h + \lambda) / n] + \bLambda_4)^{-1}\bV^\sT \bD  / [(\kappa_h + \lambda) / n] - \id_B \\
=& (\bD_1 + \bLambda_5)^{-1} \bD_1 - \id_B \\
=& (\id_B + \bD_1^{-1} \bLambda_5)^{-1}  - \id_B, \\
\end{aligned}
\]
where $\| \bDelta_i \|_{\op} = o_{d, \P}(1)$ for $i \in \{1, 2\}$ and $\| \bLambda_i - \id_B \|_{\op} = o_{d, \P}(1)$ for $i \in \{3, 4, 5\}$, and
\[
\bD_1 \equiv n \bD / (\kappa_h + \lambda).
\]
and 
\[
\lambda_{\min}(\bD_1) = \min_{k \le \ell} [ n \xi_k(h_d)] / \Big[ \sum_{k \ge \ell + 1} \xi_k(h_d) B(d, k)  + \lambda \Big]. 
\]
By Assumption \ref{ass:activation_krr}, we have $\lambda_{\min}(\bD_1) = \omega_d(1)$. This proves the lemma. 
\end{proof}

\section*{Acknowledgements}

This work was partially supported by grants NSF DMS-1613091, CCF-1714305, IIS-1741162, and
ONR N00014-18-1-2729, NSF DMS-1418362, NSF DMS-1407813.

\bibliographystyle{amsalpha}

\begin{thebibliography}{GMMM20}

\bibitem[AB09]{anthony2009neural}
Martin Anthony and Peter~L Bartlett, \emph{Neural network learning: Theoretical
  foundations}, cambridge university press, 2009.

\bibitem[ADH{\etalchar{+}}19]{arora2019fine}
Sanjeev Arora, Simon~S Du, Wei Hu, Zhiyuan Li, and Ruosong Wang,
  \emph{Fine-grained analysis of optimization and generalization for
  overparameterized two-layer neural networks}, arXiv:1901.08584 (2019).

\bibitem[AM15]{alaoui2015fast}
Ahmed~El Alaoui and Michael~W Mahoney, \emph{Fast randomized kernel ridge
  regression with statistical guarantees}, Advances in Neural Information
  Processing Systems, 2015, pp.~775--783.

\bibitem[AZLS18]{allen2018convergence}
Zeyuan Allen-Zhu, Yuanzhi Li, and Zhao Song, \emph{A convergence theory for
  deep learning via over-parameterization}, arXiv:1811.03962 (2018).

\bibitem[Bac13]{bach2013sharp}
Francis Bach, \emph{Sharp analysis of low-rank kernel matrix approximations},
  Conference on Learning Theory, 2013, pp.~185--209.

\bibitem[Bac17a]{bach2017breaking}
\bysame, \emph{Breaking the curse of dimensionality with convex neural
  networks}, The Journal of Machine Learning Research \textbf{18} (2017),
  no.~1, 629--681.

\bibitem[Bac17b]{bach2017equivalence}
\bysame, \emph{On the equivalence between kernel quadrature rules and random
  feature expansions}, The Journal of Machine Learning Research \textbf{18}
  (2017), no.~1, 714--751.

\bibitem[Bar93]{barron1993universal}
Andrew~R Barron, \emph{Universal approximation bounds for superpositions of a
  sigmoidal function}, IEEE Transactions on Information theory \textbf{39}
  (1993), no.~3, 930--945.

\bibitem[BHMM18]{belkin2018reconciling}
Mikhail Belkin, Daniel Hsu, Siyuan Ma, and Soumik Mandal, \emph{Reconciling
  modern machine learning and the bias-variance trade-off}, arXiv:1812.11118
  (2018).

\bibitem[BHX19]{belkin2019two}
Mikhail Belkin, Daniel Hsu, and Ji~Xu, \emph{Two models of double descent for
  weak features}, arXiv:1903.07571 (2019).

\bibitem[BTA11]{berlinet2011reproducing}
Alain Berlinet and Christine Thomas-Agnan, \emph{Reproducing kernel hilbert
  spaces in probability and statistics}, Springer Science \& Business Media,
  2011.

\bibitem[CB18]{chizat2018global}
Lenaic Chizat and Francis Bach, \emph{On the global convergence of gradient
  descent for over-parameterized models using optimal transport}, Advances in
  neural information processing systems, 2018, pp.~3036--3046.

\bibitem[CDV07]{caponnetto2007optimal}
Andrea Caponnetto and Ernesto De~Vito, \emph{Optimal rates for the regularized
  least-squares algorithm}, Foundations of Computational Mathematics \textbf{7}
  (2007), no.~3, 331--368.

\bibitem[Chi11]{chihara2011introduction}
Theodore~S Chihara, \emph{An introduction to orthogonal polynomials}, Courier
  Corporation, 2011.

\bibitem[COB19]{chizat2019lazy}
Lenaic Chizat, Edouard Oyallon, and Francis Bach, \emph{On lazy training in
  differentiable programming}, Advances in Neural Information Processing
  Systems, 2019, pp.~2933--2943.

\bibitem[CST{\etalchar{+}}00]{cristianini2000introduction}
Nello Cristianini, John Shawe-Taylor, et~al., \emph{An introduction to support
  vector machines and other kernel-based learning methods}, Cambridge
  University Press, 2000.

\bibitem[Cyb89]{cybenko1989approximation}
George Cybenko, \emph{Approximation by superpositions of a sigmoidal function},
  Mathematics of control, signals and systems \textbf{2} (1989), no.~4,
  303--314.

\bibitem[DHM89]{devore1989optimal}
Ronald~A DeVore, Ralph Howard, and Charles Micchelli, \emph{Optimal nonlinear
  approximation}, Manuscripta mathematica \textbf{63} (1989), no.~4, 469--478.

\bibitem[DJ89]{donoho1989projection}
David~L Donoho and Iain~M Johnstone, \emph{Projection-based approximation and a
  duality with kernel methods}, The Annals of Statistics (1989), 58--106.

\bibitem[DLL{\etalchar{+}}18]{du2018gradient2}
Simon~S Du, Jason~D Lee, Haochuan Li, Liwei Wang, and Xiyu Zhai, \emph{Gradient
  descent finds global minima of deep neural networks}, arXiv:1811.03804
  (2018).

\bibitem[DZPS18]{du2018gradient}
Simon~S Du, Xiyu Zhai, Barnabas Poczos, and Aarti Singh, \emph{Gradient descent
  provably optimizes over-parameterized neural networks}, arXiv:1810.02054
  (2018).

\bibitem[EF14]{costas2014spherical}
Costas Efthimiou and Christopher Frye, \emph{Spherical harmonics in p
  dimensions}, World Scientific, 2014.

\bibitem[EK10a]{el2010information}
Noureddine El~Karoui, \emph{On information plus noise kernel random matrices},
  The Annals of Statistics \textbf{38} (2010), no.~5, 3191--3216.

\bibitem[EK10b]{el2010spectrum}
\bysame, \emph{The spectrum of kernel random matrices}, The Annals of
  Statistics \textbf{38} (2010), no.~1, 1--50.

\bibitem[GJP95]{girosi1995regularization}
Federico Girosi, Michael Jones, and Tomaso Poggio, \emph{Regularization theory
  and neural networks architectures}, Neural computation \textbf{7} (1995),
  no.~2, 219--269.

\bibitem[GKKW06]{gyorfi2006distribution}
L{\'a}szl{\'o} Gy{\"o}rfi, Michael Kohler, Adam Krzyzak, and Harro Walk,
  \emph{A distribution-free theory of nonparametric regression}, Springer
  Science \& Business Media, 2006.

\bibitem[GMMM19]{ghorbani2019limitations}
Behrooz Ghorbani, Song Mei, Theodor Misiakiewicz, and Andrea Montanari,
  \emph{Limitations of lazy training of two-layers neural network}, Advances in
  Neural Information Processing Systems, 2019, pp.~9108--9118.

\bibitem[GMMM20]{OursUnpub}
\bysame, \emph{Which data are hard to fit with linearized neural networks?}, In
  preparation, 2020.

\bibitem[GSJW19]{geiger2019disentangling}
Mario Geiger, Stefano Spigler, Arthur Jacot, and Matthieu Wyart,
  \emph{Disentangling feature and lazy learning in deep neural networks: an
  empirical study}, arXiv:1906.08034 (2019).

\bibitem[HMRT19]{hastie2019surprises}
Trevor Hastie, Andrea Montanari, Saharon Rosset, and Ryan~J Tibshirani,
  \emph{Surprises in high-dimensional ridgeless least squares interpolation},
  arXiv:1903.08560 (2019).

\bibitem[Hor91]{hornik1991approximation}
Kurt Hornik, \emph{Approximation capabilities of multilayer feedforward
  networks}, Neural networks \textbf{4} (1991), no.~2, 251--257.

\bibitem[JGH18]{jacot2018neural}
Arthur Jacot, Franck Gabriel, and Cl{\'e}ment Hongler, \emph{Neural tangent
  kernel: Convergence and generalization in neural networks}, Advances in
  neural information processing systems, 2018, pp.~8571--8580.

\bibitem[LR18]{liang2018just}
Tengyuan Liang and Alexander Rakhlin, \emph{Just interpolate: Kernel
  "ridgeless" regression can generalize}, arXiv:1808.00387 (2018).

\bibitem[LRZ19]{liang2019risk}
Tengyuan Liang, Alexander Rakhlin, and Xiyu Zhai, \emph{On the risk of
  minimum-norm interpolants and restricted lower isometry of kernels}, arXiv
  preprint arXiv:1908.10292 (2019).

\bibitem[LXS{\etalchar{+}}19]{lee2019wide}
Jaehoon Lee, Lechao Xiao, Samuel~S Schoenholz, Yasaman Bahri, Jascha
  Sohl-Dickstein, and Jeffrey Pennington, \emph{Wide neural networks of any
  depth evolve as linear models under gradient descent}, arXiv:1902.06720
  (2019).

\bibitem[Mai99]{maiorov1999best}
VE~Maiorov, \emph{On best approximation by ridge functions}, Journal of
  Approximation Theory \textbf{99} (1999), no.~1, 68--94.

\bibitem[MBM16]{mei2016landscape}
Song Mei, Yu~Bai, and Andrea Montanari, \emph{The landscape of empirical risk
  for non-convex losses}, {\sf arXiv:1607.06534} (2016).

\bibitem[Mha96]{mhaskar1996neural}
Hrushikesh~N Mhaskar, \emph{Neural networks for optimal approximation of smooth
  and analytic functions}, Neural computation \textbf{8} (1996), no.~1,
  164--177.

\bibitem[MM94]{mhaskar1994dimension}
Hrushikesh~Narhar Mhaskar and Charles~A Micchelli, \emph{Dimension-independent
  bounds on the degree of approximation by neural networks}, IBM Journal of
  Research and Development \textbf{38} (1994), no.~3, 277--284.

\bibitem[MM19]{mei2019generalization}
Song Mei and Andrea Montanari, \emph{The generalization error of random
  features regression: Precise asymptotics and double descent curve},
  arXiv:1908.05355 (2019).

\bibitem[MMM19]{mei2019mean}
Song Mei, Theodor Misiakiewicz, and Andrea Montanari, \emph{Mean-field theory
  of two-layers neural networks: dimension-free bounds and kernel limit},
  arXiv:1902.06015 (2019).

\bibitem[MMN18]{mei2018mean}
Song Mei, Andrea Montanari, and Phan-Minh Nguyen, \emph{A mean field view of
  the landscape of two-layer neural networks}, Proceedings of the National
  Academy of Sciences (2018).

\bibitem[Nea96]{neal1996priors}
Radford~M Neal, \emph{Priors for infinite networks}, Bayesian Learning for
  Neural Networks, Springer, 1996, pp.~29--53.

\bibitem[Pet98]{petrushev1998approximation}
Pencho~P Petrushev, \emph{Approximation by ridge functions and neural
  networks}, SIAM Journal on Mathematical Analysis \textbf{30} (1998), no.~1,
  155--189.

\bibitem[Pin99]{pinkus1999approximation}
Allan Pinkus, \emph{Approximation theory of the mlp model in neural networks},
  Acta numerica \textbf{8} (1999), 143--195.

\bibitem[RR08]{rahimi2008random}
Ali Rahimi and Benjamin Recht, \emph{Random features for large-scale kernel
  machines}, Advances in neural information processing systems, 2008,
  pp.~1177--1184.

\bibitem[RR17]{rudi2017generalization}
Alessandro Rudi and Lorenzo Rosasco, \emph{Generalization properties of
  learning with random features}, Advances in Neural Information Processing
  Systems, 2017, pp.~3215--3225.

\bibitem[RVE18]{rotskoff2018neural}
Grant~M Rotskoff and Eric Vanden-Eijnden, \emph{Neural networks as interacting
  particle systems: Asymptotic convexity of the loss landscape and universal
  scaling of the approximation error}, arXiv:1805.00915 (2018).

\bibitem[SHN{\etalchar{+}}18]{soudry2018implicit}
Daniel Soudry, Elad Hoffer, Mor~Shpigel Nacson, Suriya Gunasekar, and Nathan
  Srebro, \emph{The implicit bias of gradient descent on separable data}, The
  Journal of Machine Learning Research \textbf{19} (2018), no.~1, 2822--2878.

\bibitem[SS18]{sirignano2018mean}
Justin Sirignano and Konstantinos Spiliopoulos, \emph{Mean field analysis of
  neural networks}, arXiv:1805.01053 (2018).

\bibitem[{Sze}39]{szego1939orthogonal}
{Szeg\H{o}, Gabor}, \emph{Orthogonal polynomials}, vol.~23, American
  Mathematical Soc., 1939.

\bibitem[Tsy08]{tsybakov2008introduction}
Alexandre~B Tsybakov, \emph{Introduction to nonparametric estimation}, Springer
  Science \& Business Media, 2008.

\bibitem[VW18]{vempala2018polynomial}
Santosh Vempala and John Wilmes, \emph{Polynomial convergence of gradient
  descent for training one-hidden-layer neural networks}, arXiv:1805.02677
  (2018).

\bibitem[YS19]{yehudai2019power}
Gilad Yehudai and Ohad Shamir, \emph{On the power and limitations of random
  features for understanding neural networks}, arXiv:1904.00687 (2019).

\bibitem[ZCZG18]{zou2018stochastic}
Difan Zou, Yuan Cao, Dongruo Zhou, and Quanquan Gu, \emph{Stochastic gradient
  descent optimizes over-parameterized deep relu networks}, arXiv:1811.08888
  (2018).

\end{thebibliography}

\newcommand{\etalchar}[1]{$^{#1}$}
\providecommand{\bysame}{\leavevmode\hbox to3em{\hrulefill}\thinspace}
\providecommand{\MR}{\relax\ifhmode\unskip\space\fi MR }
\providecommand{\MRhref}[2]{%
  \href{http://www.ams.org/mathscinet-getitem?mr=#1}{#2}
}
\providecommand{\href}[2]{#2}

\newpage

\appendix

\section{Numerical results with ridge regression}
\label{app:Ridge}

\begin{figure}
\includegraphics[width=0.49\linewidth]{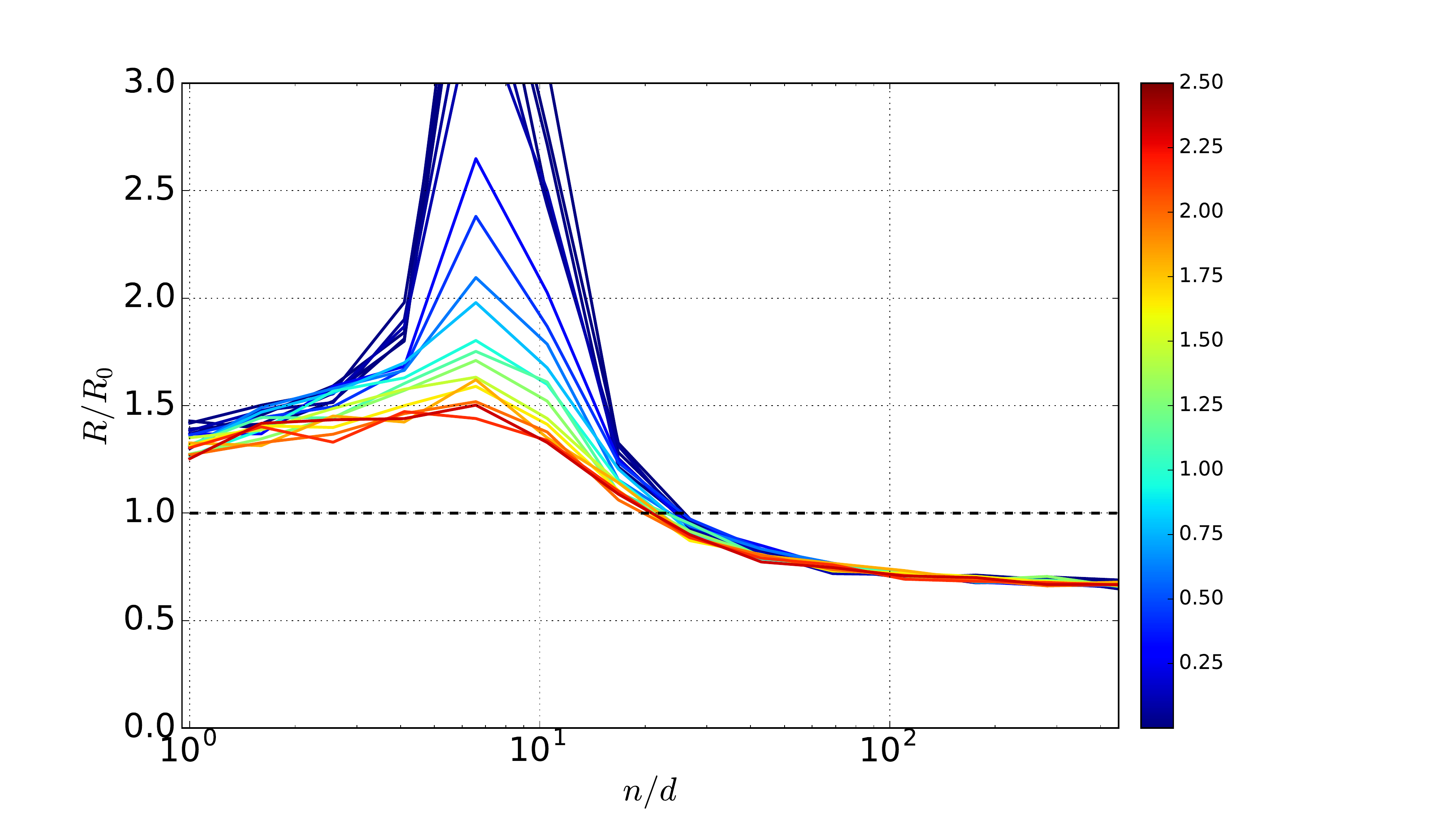}
\includegraphics[width=0.49\linewidth]{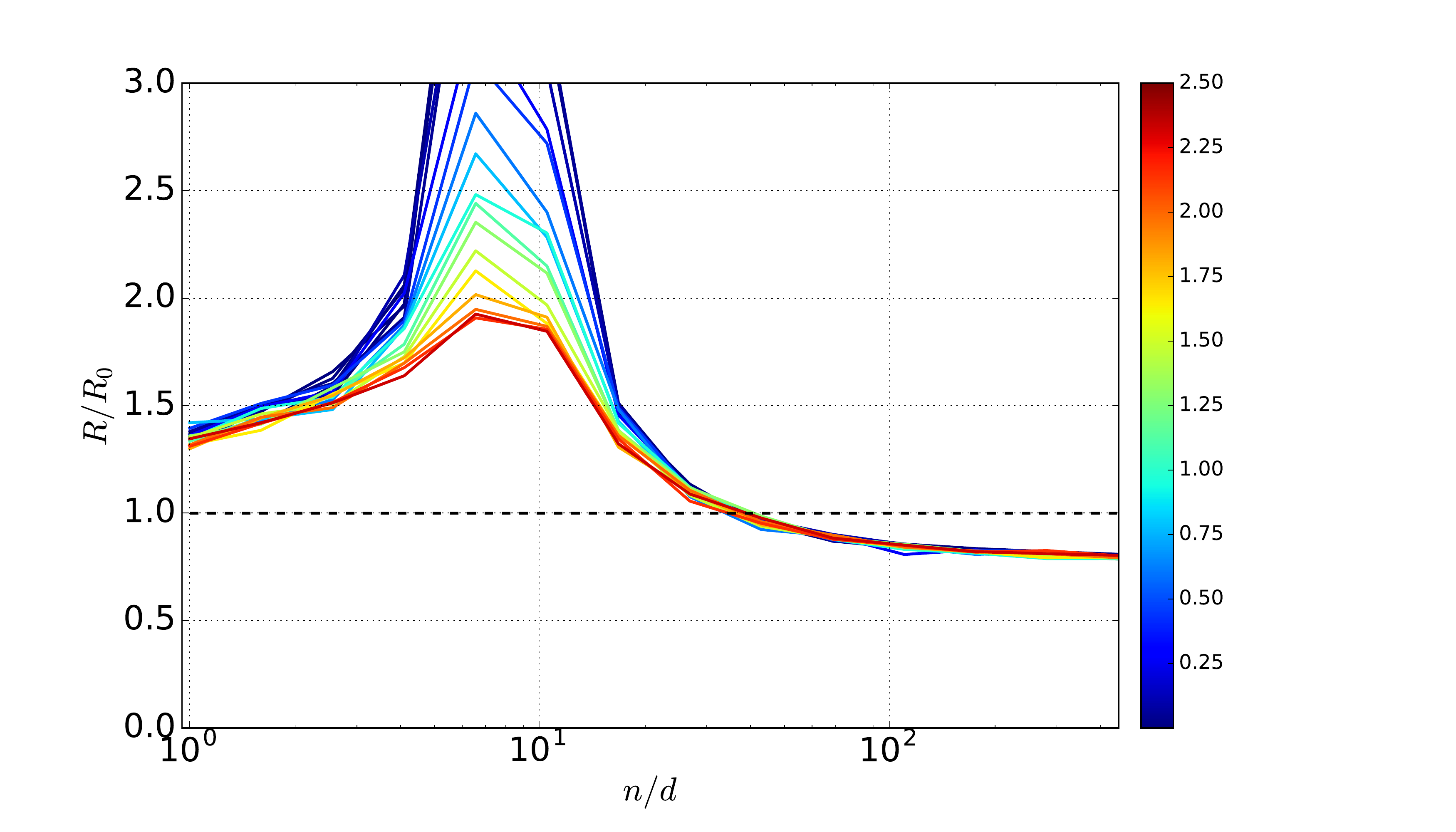}
\caption{Risk of the random features model for learning a quadratic
  function $f_{\star,2}$. We fit the model coefficients using ridge regression, with different curves corresponding to 
different values of the regularization parameter. Test error is computed on $n_{\stest}=1500$ fresh samples.
Left frame: $d=30$, $N=240$. Right frame: $d=50$, $N=400$.}\label{fig:RF-Ridge-Quad}
\end{figure}

\begin{figure}
\includegraphics[width=0.49\linewidth]{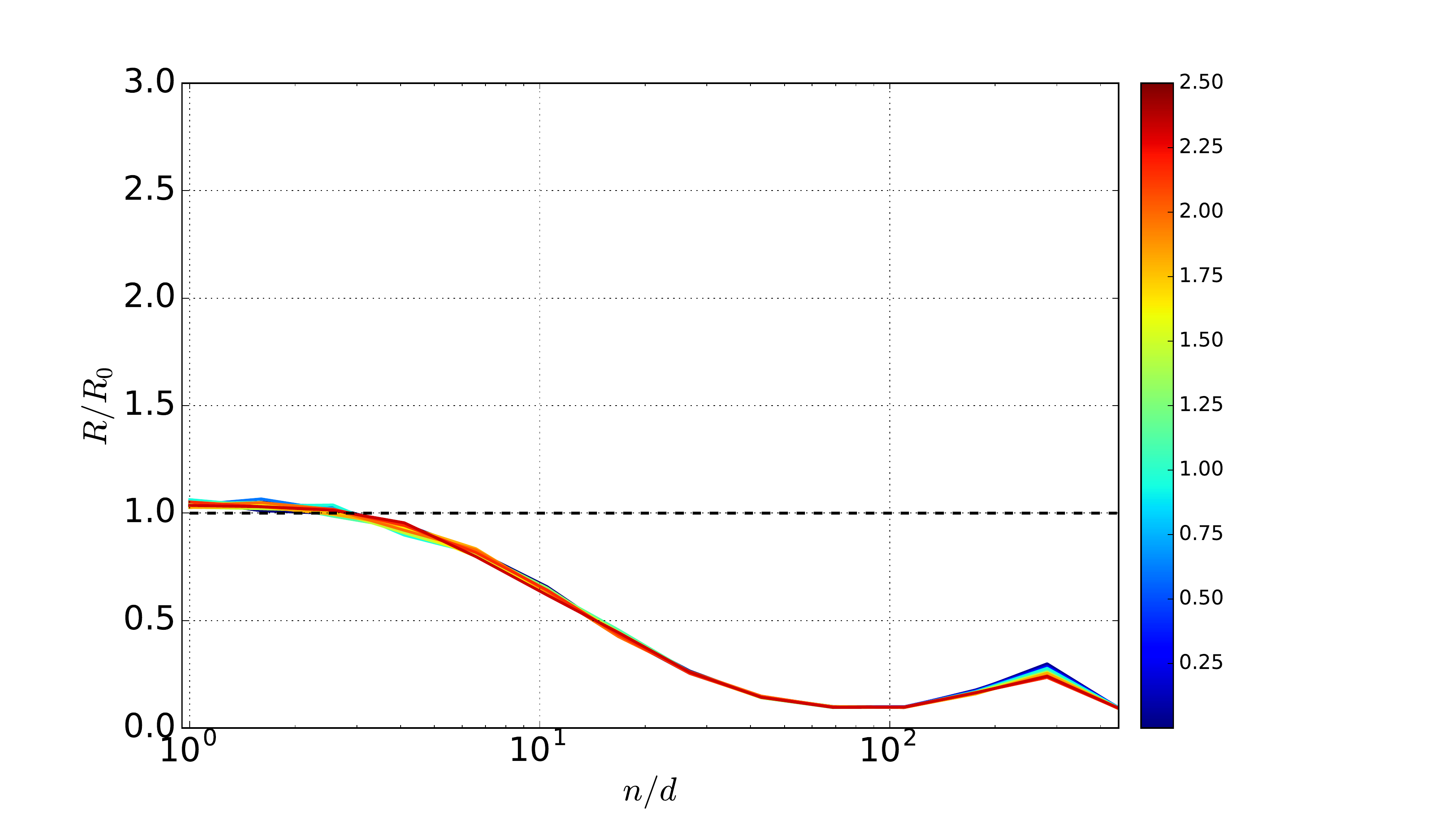}
\includegraphics[width=0.49\linewidth]{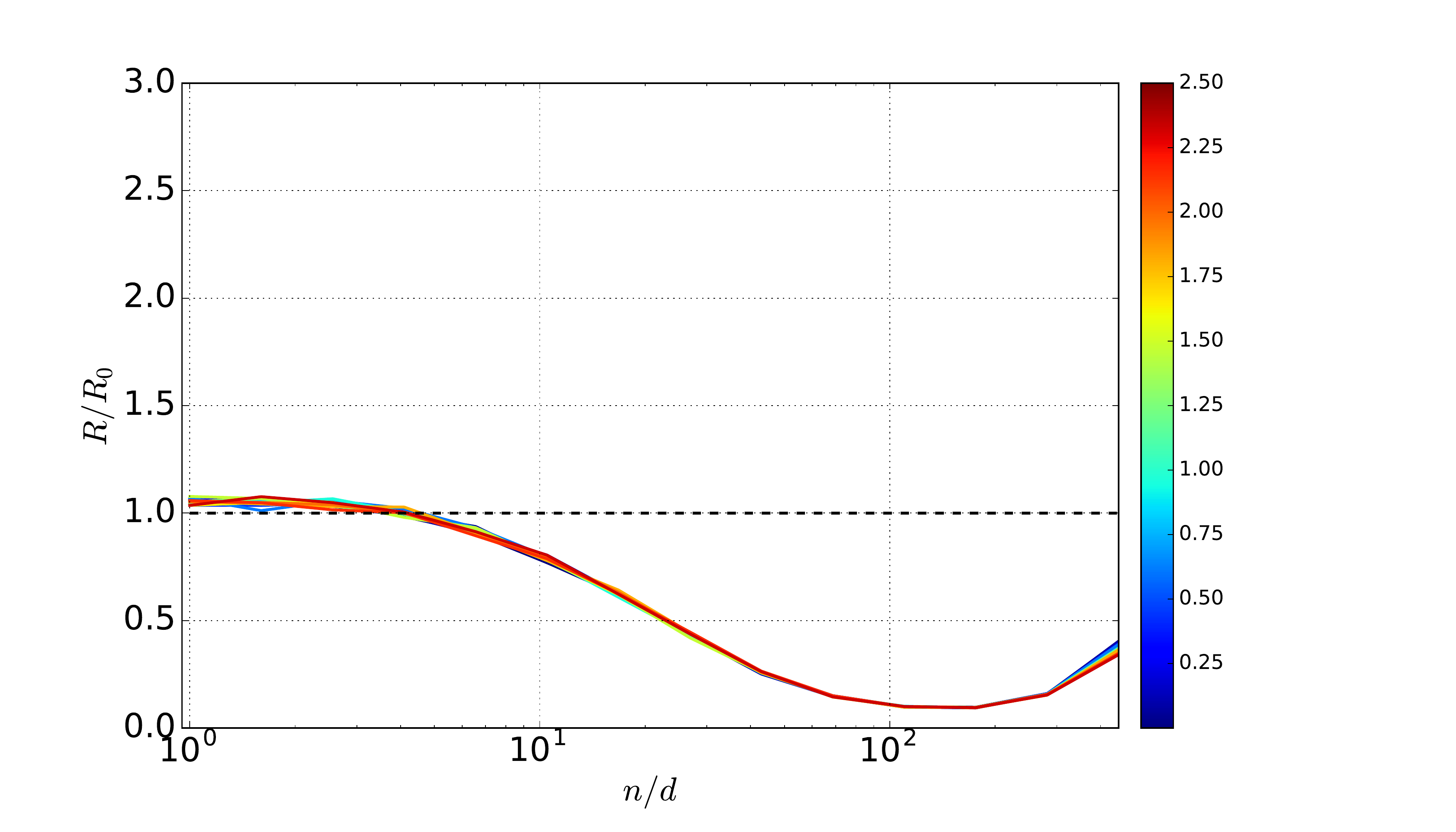}
\caption{Risk of the neural tangent model for learning a quadratic
  function $f_{\star,2}$. The other settings are the same as in Figure \ref{fig:RF-Ridge-Quad}. Test error is computed on $n_{\stest}=1500$ fresh samples. Left frame: $d=30$, $N=240$. Right frame: $d=50$, $N=400$.}\label{fig:NTK-Ridge-Quad}
\end{figure}

\begin{figure}
\includegraphics[width=0.49\linewidth]{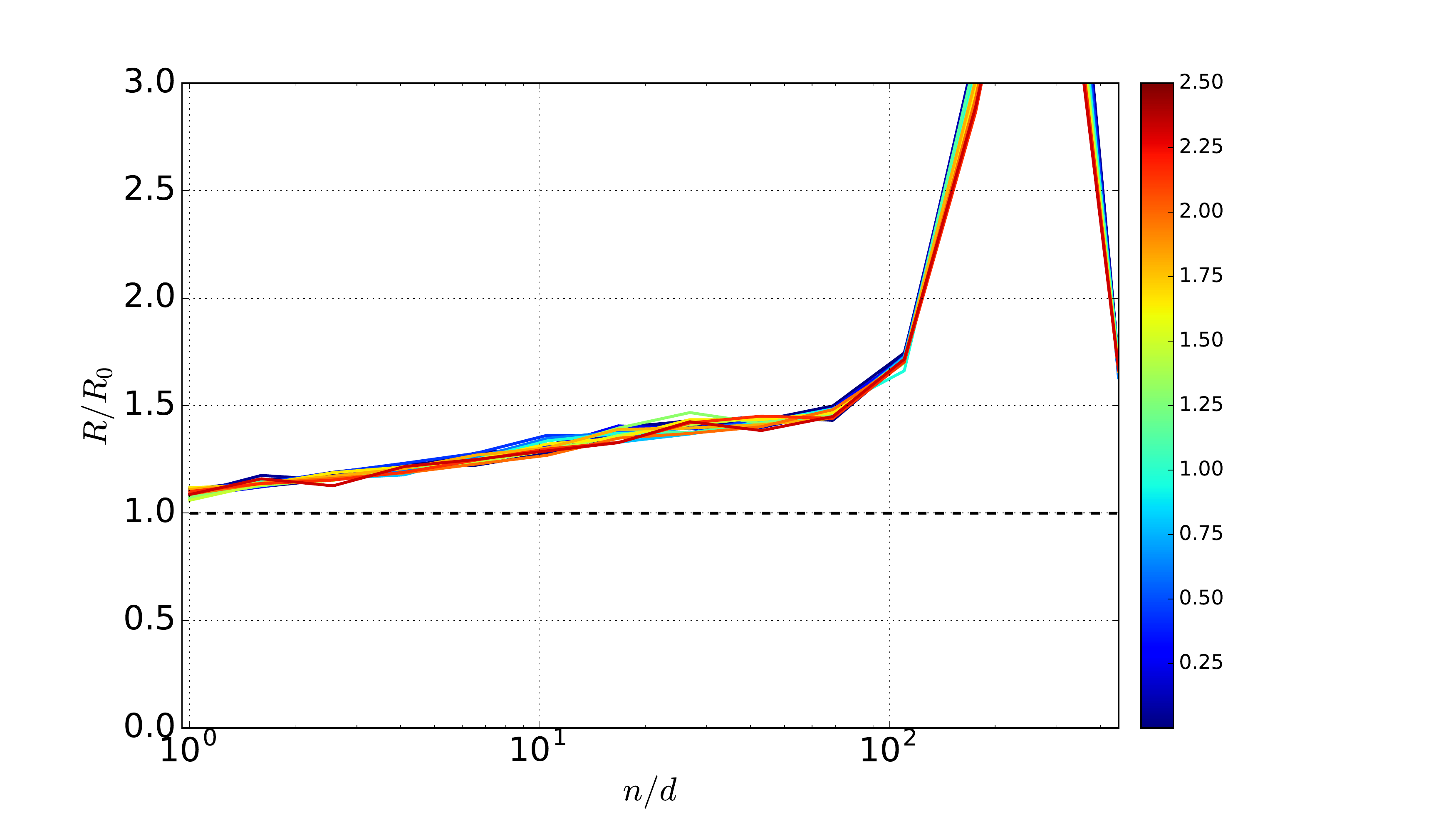}
\includegraphics[width=0.49\linewidth]{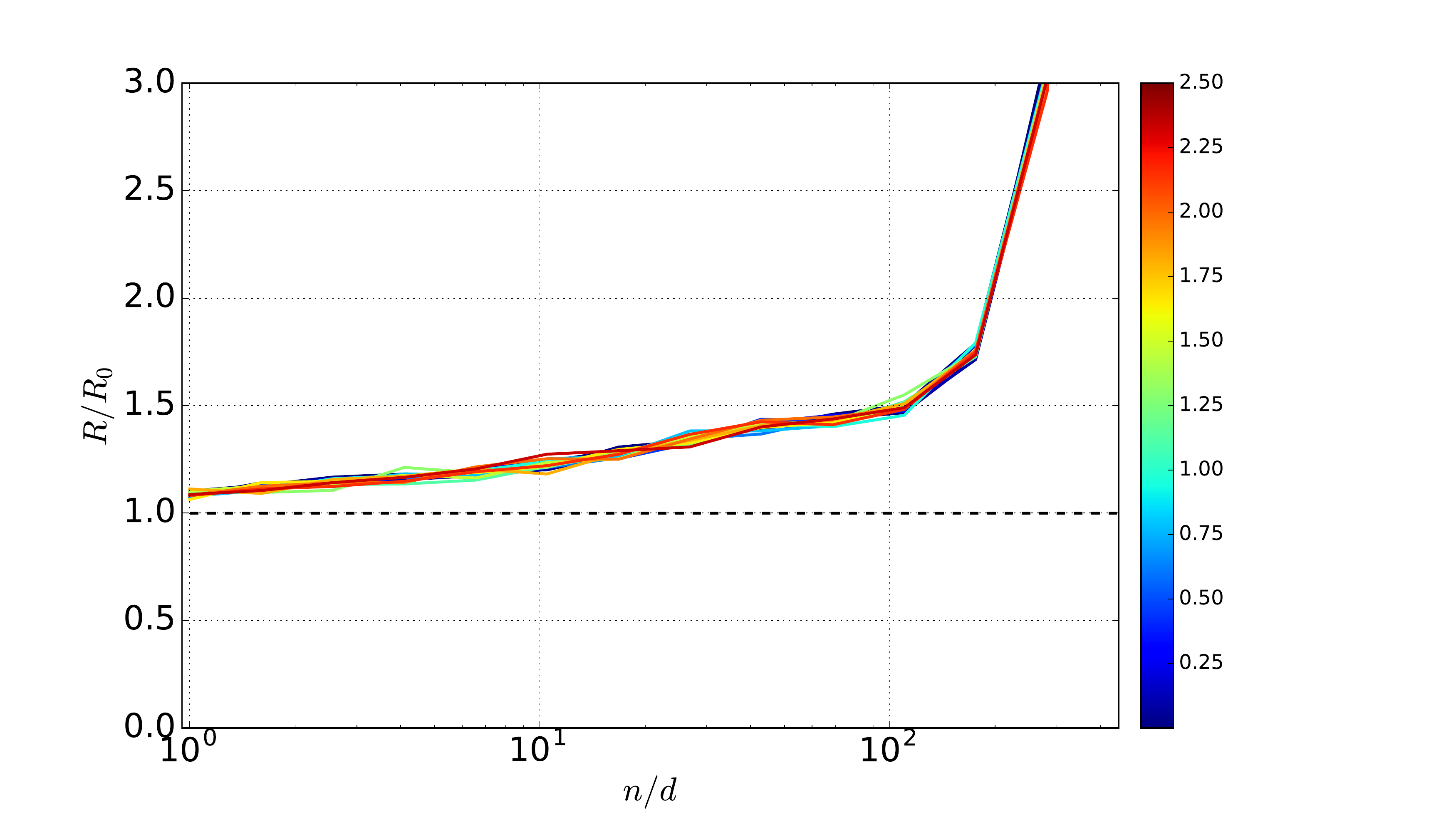}
\caption{Risk of the neural tangent model for learning a third order polynomial
  $f_{\star,2}$. The other settings are the same as in Figure \ref{fig:RF-Ridge-Quad}.  Left frame: $d=30$, $N=240$. Right frame: $d=50$, $N=400$.}\label{fig:NTK-Ridge-ThirdDeg}
\end{figure}

The reader might wonder whether the numerical results presented in
Section  \ref{sec:Numerical} might change significantly if we changed the method to estimate the coefficients $\ba = (a_i)_{i\le N}\in\reals^N$
(for the model \RF) or $\ba = (\ba_i)_{i\le N}\in\reals^{Nd}$. Our main results --Theorem \ref{thm:RF_lower_upper_bound} and Theorem \ref{thm:NT_lower_upper_bound}.(a)-- predict that 
the result should not change qualitatively: these models are limited because they cannot approximate the target function $f_{\star}$ (unless this is a low degree polynomial),
regardless of the choice of the representative $f\in \cF_{\RF}$ or $f\in\cF_{\NT}$.

In order to verify this prediction numerically, we repeated the experiments of Section \ref{sec:Numerical} using ridge regression.
We form a matrix $\bZ\in\reals^{n\times p}$ containing the $p$ covariates (with $p=N$ for \RF, and $p=Nd$ for \NT),
whereby $Z_{ij}= \sigma(\<\bw_j,\bx_i\>)$ for \RF, and $Z_{i,(j_1j_2)} = (\bx_i)_{j_2}\sigma(\<\bw_{j_1},\bx_i\>)$ for \NT. Letting $y_i =f_{\star}(\bx_i)$,
we estimate the coefficients $\ba$ via
\begin{align}
\hat\ba = \arg\min_{\ba\in\reals^p}\Big\{\|\by-\bZ\ba\|_2^2+\lambda\|\ba\|_2^2\Big\}\, .
\end{align}
The results are reported in Figures \ref{fig:RF-Ridge-Quad}, \ref{fig:NTK-Ridge-Quad}, \ref{fig:NTK-Ridge-ThirdDeg}, and are consistent with the ones
of Section  \ref{sec:Numerical}. Regularization does not help: it only reduces the peak at $n\approx p$, as expected from \cite{hastie2019surprises}, but not the large $n$
behavior. 

(Note that for \RF\, we do not report results for $d=100$, in Fig.~\ref{fig:RF-Ridge-Quad}. As in Fig.~\ref{fig:RF-SecondDeg},
the resulting risk is slightly below the baseline $R_0$: this effect vanishes for $d\gtrsim 100$.) 

\end{document}